\documentclass[a4paper]{amsart}
\usepackage[all]{xy}
\usepackage{hyperref, fullpage}
\usepackage{enumerate}
\usepackage{amssymb}
\usepackage{tikz}
\usepackage{pdfsync}
\usetikzlibrary{matrix,arrows,decorations.pathreplacing}

\makeatletter
\def\underbrace#1{%
  \@ifnextchar_{\tikz@@underbrace{#1}}{\tikz@@underbrace{#1}_{}}}
\def\tikz@@underbrace#1_#2{%
  \tikz[baseline=(a.base)] {\node[inner sep=2] (a) {\(#1\)};
  \draw[line cap=round,decorate,decoration={brace,amplitude=4pt}]
    (a.south east) -- node[pos=0.5,below,inner sep=7pt] {\(\scriptstyle #2\)} (a.south west);}}
\makeatother
\usepackage{graphicx}
\usepackage[all]{xy}
\usepackage{mathrsfs}
\usepackage{color}

\usepackage[english]{babel}
\usepackage{url}

\usepackage[utf8]{inputenc}
\usepackage[T1]{fontenc}
 
\usepackage{amsmath, amscd}
\usepackage{amsthm}
\usepackage{amsfonts}
\usepackage{amssymb}
\usepackage{stmaryrd}
\usepackage{hyperref}

\theoremstyle{plain}
\newtheorem{thm}{Theorem}[section]
\newtheorem{prop}[thm]{Proposition}
\newtheorem{lemma}[thm]{Lemma}
\newtheorem{crl}[thm]{Corollary}
\theoremstyle{definition}
\newtheorem{dfn}[thm]{Definition}
\newtheorem{example}[thm]{Example}

\theoremstyle{remark}
\newtheorem{rmk}[thm]{Remark}

\newtheorem*{claim*}{Claim}
\newtheorem{claim}{Claim}

\newcommand{\G}{\mathcal{G}}
\renewcommand{\H}{\mathcal{H}}

\renewcommand{\L}{\mathcal{L}}
\newcommand{\V}{\mathcal{V}}
\newcommand{\rmap}{\longrightarrow}
\newcommand{\lmap}{\longleftarrow}

\renewcommand{\epsilon}{\varepsilon}

\newcommand{\CG}{C_{\text{def}}^*(\mathcal{G})}
\newcommand{\CkG}{C_{\text{def}}^k(\mathcal{G})}
\newcommand{\Gk}{\mathcal{G}^{(k)}}

\newcommand{\R}{\mathbb{R}}
\newcommand{\g}{\mathfrak{g}}

\newcommand{\x}{\times}
\newcommand{\lx}{\ltimes}
\newcommand{\arrows}{\rightrightarrows}
\newcommand{\tto}{\rightrightarrows}

\newcommand{\dezero}{{\frac{d}{d\epsilon}\big{|}}_{\epsilon=0}}

\begin{document}
\title{Deformations of Lie groupoids}

\author{Marius Crainic}\email{crainic@math.uu.nl}
\address{Depart. of Math., Utrecht University, 3508 TA Utrecht,
The Netherlands}

\author{Jo\~ao Nuno Mestre}\email{jnmestre@gmail.com}
\address{Depart. of Math., Utrecht University, 3508 TA Utrecht,
The Netherlands}

\author{Ivan Struchiner}\email{ivanstru@ime.usp.br}
\address{Depart. of Math., University of S\~ao Paulo, Cidade Universit\'aria, S\~ao Paulo, Brasil}

\thanks{The first author was supported by ERC starting grant no 279729 and the NWO-Vici grant no. 639.033.312. The second author was supported by FCT grant  SFRH/BD/71257/2010 under the POPH/FSE programmes. The third author acknowledges the support of CNPq 304840/2013-0 and FAPESP 2013/16753-8}

\begin{abstract} We study deformations of Lie groupoids by means of the cohomology which controls them. This cohomology turns out to provide an intrinsic model for the cohomology of a Lie groupoid with values in its adjoint representation. We prove several fundamental properties of the deformation cohomology including Morita invariance, a van Est theorem, and a vanishing result in the proper case.  
Combined with Moser's deformation arguments for groupoids, we obtain several rigidity and normal form results.
\end{abstract}

\maketitle

\setcounter{tocdepth}{1}

\tableofcontents

\section{Introduction}

A central problem in geometry is that of understanding the behaviour of  geometric structures under deformations; each class of geometric structures comes with its deformation theory, including a 
cohomology theory that controls such deformations. The aim of this paper is to investigate the cohomology theory controlling deformations of a large class of geometric structures and to use it to prove several rigidity results. The geometric structures that 
we have in mind are those that can be modelled by Lie groupoids and it includes Lie groups (and bundles of such), Lie group actions on manifolds, foliations, the symplectic groupoids of Poisson geometry, etc. In other words, we study the deformation theory of Lie groupoids $\G$ and the resulting deformation cohomology $H^{*}_{\textrm{def}}(\G)$.
The cohomology is built in such a way that deformations of a groupoid $\G$ give rise to $2$-cocycles inducing elements in $H^{2}_{\textrm{def}}(\G)$; $1$-cochains that transgress these $2$-cocycles (when they exist) are then used to produce flows that allow us to prove rigidity results. Intuitively, this allows us to think of $H^{2}_{\textrm{def}}(\G)$ as the first order approximation (i.e. tangent space) of the moduli space of deformations of $\G$. We became aware of the existence of such a cohomology while searching for a geometric proof of Zung's linerization theorem for proper Lie groupoids \cite{ivan}; the same cohomology also arises naturally when looking at the VB-interpretation of the adjoint representation \cite{mehta}. 

Before we give more details on deformations and rigidity results, let us first describe some of the main properties/results regarding the deformation cohomology. The first one we would like to mention is Morita invariance:

\begin{thm}[(Morita invariance)]
If two Lie groupoids are Morita equivalent, then their deformation cohomologies are isomorphic.
\end{thm}

Intuitively, this means that the deformation cohomology only depends on ``the transverse geometry of the groupoid''. This is very much related to Haefliger's philosophy/approach to the transverse geometry of foliations via the associated \'etale groupoids; the role of the groupoid(s) was to model (desingularize) leaf spaces of foliations; the notion of Morita equivalence of groupoids comes in for a simple reason: there is no canonical \'etale groupoid modelling a given leaf space, but several of them (e.g. any complete transversal to the foliation gives rise to one) - and the fact that two groupoids correspond to the same leaf space can be recognized by the fact that they are Morita equivalent. In other words, Haefliger's point of view is that the transverse geometry of foliations is the part of the geometry of \'etale groupoids which is Morita invariant. A slight generalization of this philosophy is the interpretation of Lie groupoids as ``atlases'' for differentiable stacks; again, two groupoids correspond to the same stack if and only if they are Morita equivalent. Hence the previous theorem allows us to talk about ``the deformation cohomology of differentiable stacks''.

The deformation cohomology $H^{*}_{\textrm{def}}(\G)$ is also related to differentiable cohomologies of $\G$. As we will recall in subsection \ref{sub-Differentiable cohomology}, $H^{*}_{\textrm{diff}}(\G, E)$ (simply denoted $H^{*}(\G, E)$ in this paper) makes sense as soon as we have a representation $E$ of $\G$. It is not true that, in general, $H^{*}_{\textrm{def}}(\G)$ is isomorphic to $H^*(\G, E)$ for some representation $E$ of $\G$ (and this is very much related to the fact that ``the adjoint representation of $\G$'' does not make sense as a representation in the usual sense - see also below). However, $H^{*}_{\textrm{def}}(\G)$ can often be related to differentiable cohomology. This is best illustrated in the regular case, i.e., when all the orbits of $\G$ have the same dimension. In this case, $\G$ comes with two natural representations in the classical sense: the normal representation $\nu$ (on the normal bundle of the orbits) and the isotropy representation $\mathfrak{i}$ (made of the Lie algebras $\mathfrak{i}_x$ of the Lie groups $\G_x$ consisting of arrows that start and end at $x\in M$). These are recalled in more detail in subsections \ref{sec-The isotropy representation} and \ref{sec-The normal representation}. We will show that $H^{*}_{\textrm{def}}(\G)$ is related to the standard differentiable cohomology with coefficients in these representations by a long exact sequence:

\begin{prop}[(The regular case)] The deformation cohomology of a regular Lie groupoid $\G$ fits into a long exact sequence 
\[ \cdots\longrightarrow H^{k}(\G;\mathfrak{i})\stackrel{r}{\rmap} H^k_{\operatorname{def}}(\G)\stackrel{\pi}{\rmap} H^{k-1}(\G;\nu)\stackrel{K}{\rmap} H^{k+1}(\G;\mathfrak{i})\rmap\cdots ,\]
where $\nu$ and $\mathfrak{i}$ are the normal and the isotropy representation of $\G$, respectively.
\end{prop}

The spaces of invariants
\[ \Gamma(\mathfrak{i})^{\textrm{inv}}:= H^{0}(\G;\mathfrak{i}),  \textrm{and} \ \Gamma(\nu)^{\textrm{inv}} := H^{0}(\G;\nu)\]
are of independent interest; with some extra care (explained in subsections \ref{sec-The isotropy representation} and \ref{sec-The normal representation}), they make sense even in the non-regular case. In our analysis in low degrees we will see that, in general, $H^{0}_{\textrm{def}}(\G)$ is always isomorphic to $\Gamma(\mathfrak{i})^{\textrm{inv}}$
and one has a low degrees exact sequence (Proposition \ref{the-LOS}):
\[ 0 \rmap  H^1(\G, \mathfrak{i}) \stackrel{r}{\rmap} H_{\text{def}}^1(\G) \stackrel{\pi}{\rmap}  H^0(\G, \nu)\stackrel{K}{\rmap} H^2(\G, \mathfrak{i}) \stackrel{r}{\rmap} H_{\text{def}}^2(\G).\]

Another important property of $H^{*}_{\textrm{def}}(\G)$, which will be essential also for proving rigidity results,  is its behaviour for proper groupoids. Recall that a Lie groupoid $\G$ over a manifold $M$ is said to be proper if $\G$ is Hausdorff and the map $(s, t): \G\rmap M\times M$, which associates to an arrow its source and target, is proper; this generalizes the compactness of Lie groups and the properness of Lie group actions. While proper groupoids are the first candidates for rigidity phenomena, an essential step in proving such rigidity theorems is to understand the behaviour of  $H^{*}_{\textrm{def}}(\G)$ in low degrees. Of particular importance will be the vanishing of $H^{2}_{\textrm{def}}(\G)$ which, due to its interpretation as the tangent space to the moduli space of deformations, could be called ``infinitesimal rigidity''. However, the understanding of $H^{1}_{\textrm{def}}(\G)$ (related to families of automorphisms) and even of $H^0_{\mathrm{def}}(\G)$, will be important as well. We will show: 

\begin{thm}[(Vanishing)] Let $\G$ be a proper Lie groupoid. Then 
\[ H_{\mathrm{def}}^k(\G)= 0 \ \textrm{for\ all}\ k\geq 2\]
while $H_{\mathrm{def}}^1(\G)\cong \Gamma(\nu)^{\mathrm{inv}}$, $H_{\mathrm{def}}^0(\G)\cong \Gamma(\mathfrak{i})^{\mathrm{inv}}$.
\end{thm}

Besides the relevance to deformations, there is yet another guiding principle that is worth having in mind when thinking about $H_{\text{def}}^*(\G)$: it makes sense of the (differentiable) cohomology with coefficients in ``the adjoint representation'' (think e.g. about the case of Lie groups \cite{Copper}). This guiding principle is very useful when computing the deformation cohomology in terms of ordinary differentiable cohomology. 
The last quotes indicate the subtleties that arise when looking for ``the adjoint representation of a Lie groupoid''. 
Indeed, as remarked already in the early years of Lie groupoids, one of the subtleties which make Lie groupoids harder to handle than Lie groups is the fact that the notion of adjoint representation does not make sense when restricting to the classical notion of representation. With the more recent concept of ``representation up to homotopy'' \cite{camilo}, we now understand that:
\begin{itemize}
\item the adjoint representation $Ad$ of a Lie groupoid $\G$ makes sense 
intrinsically as an isomorphism class of representations up to homotopy of $\G$. 
\item to represent $Ad$ by an actual representation up to homotopy $Ad_{\sigma}$ one needs to choose an Ehresmann connection $\sigma$ on $\G$ (this will be recalled in Section \ref{sec-adjoint}).
\item by the very definition of representations up to homotopy, they serve as coefficients for differentiable cohomology. However, to have an explicit model for the cohomology with coefficients in the adjoint representation, one still has to choose a connection $\sigma$ and consider the associated $H^*(\G, Ad_{\sigma})$.
\end{itemize}

One of the main features of $H^{*}_{\textrm{def}}(\G)$ is that it provides an intrinsic model for these cohomologies, independent of any auxiliary choices:

\begin{thm}[(The adjoint representation)]
If $\G$ is a Lie groupoid, then for any connection $\sigma$ on $\G$ one has a canonical isomorphism of $H^{*}_{\mathrm{def}}(\G)$ with $H^*(\G, Ad_{\sigma})$.
\end{thm}

We would like to point out that we have decided to write the paper in such a way that it does not assume prior knowledge of the adjoint representation as a representation up to homotopy. On the contrary, using the paradigm
that ``the adjoint representation serves as coefficients of the cohomology theory that controls deformations'' (i.e. of 
$H^{*}_{\textrm{def}}(\G)$), we will slowly guide the reader towards the final outcome. In particular, the reader will encounter along the way several comments on ``the adjoint representation'', slowly revealing its full structure.

There is yet another way to look at $H^{*}_{\textrm{def}}(\G)$, related to the fact that the theory of Lie groupoids comes with an infinitesimal counterpart - that of Lie algebroids. From this point of view, our cohomology is the global analogue of the deformation cohomology $H^{*}_{\textrm{def}}(A)$ of the Lie algebroid $A$ of $\G$, which was studied in \cite{marius_ieke}. In the same way that the differentiable cohomology of a Lie group(/oid) is related to the cohomology of the corresponding Lie algebra(/oid) by a van Est map, we will prove that:

\begin{thm}[(Van Est isomorphisms)] Let $\G$ be a Lie groupoid with Lie algebroid $A$. There exists a canonical 
chain map (the van Est map) 
\[ \V:C^*_{\mathrm{def}}(\G)\rmap C^*_{\mathrm{def}}(A).\]
Moreover, if $\G$ has $k$-connected $s$-fibres then the map induced in cohomology 
\[ \V:H^p_{\mathrm{def}}(\G)\rmap H^p_{\mathrm{def}}(A)\] 
is an isomorphism for all $p\leq k$.
\end{thm}

It is worth insisting a bit more on the similarities and differences with the corresponding infinitesimal theory $H^{*}_{\textrm{def}}(A)$ from \cite{marius_ieke}. As there, one of the main subtleties of $H^{*}_{\textrm{def}}(\G)$
comes from the fact that the adjoint representation is only defined as a representation up to homotopy. However, the setting of Lie groupoids, because of its non-linear nature, comes with even more subtleties - to the extent that we revisit even the very definition of Lie groupoids - giving rise to the Appendix of the paper. On the other hand, and in contrast with the infinitesimal theory, for 
$H^{*}_{\textrm{def}}(\G)$ one can prove the vanishing result mentioned above, with direct consequences to rigidity.

As we have already mentioned, one of the main motivations for studying the deformation cohomology $H^{*}_{\textrm{def}}(\G)$ comes from its relevance to deformations and rigidity results.
We are interested in general deformations:

\begin{dfn}\label{dfn-general-def} 
Let $\G\tto M$ be a groupoid over $M$, with structure maps denoted $s, t, m, i, u$ (the source, target, multiplication, inversion, unit map, respectively). A \emph{(smooth) deformation of $\G$}
is a family 
\[ \tilde{\G}= \{\G_{\epsilon}: \epsilon \in I\} \ \ \textrm{of\ groupoids}\ \G_{\epsilon}\tto M_{\epsilon} \]
smoothly parametrized by $\epsilon$ in an open interval $I$ containing $0$, such that $\G_{0}= \G$ as groupoids. We will denote the structure maps of $\G_{\epsilon}$ by $s_{\epsilon}, t_{\epsilon}, m_{\epsilon}, i_{\epsilon}, u_{\epsilon}$.

The deformation is called \emph{strict} if $\G_{\epsilon}= \G$ as manifolds; it is called \emph{$s$-constant} if, furthermore, $s_{\epsilon}$ does not depend on $\epsilon$. The \emph{constant deformation} is the one with $\G_{\epsilon}= \G$ as groupoids. 

Two deformations $\{\G_{\epsilon}: \epsilon\in I\}$ and $\{\G_{\epsilon}': \epsilon\in I'\}$ are called \emph{equivalent} if there exists a family of groupoid isomorphisms $\phi^{\epsilon}: \G_{\epsilon}\rmap \G_{\epsilon}'$, smoothly parametrized by $\epsilon$ in an open interval containing $0$, such that $\phi_0= \textrm{Id}$. We say that $\tilde{\G}$ is \emph{trivial} if it is equivalent to the constant deformation. 
\end{dfn}

In general, a family $\{M_{\epsilon}: \epsilon \in I\}$ of manifolds smoothly parametrized by $\epsilon$ can be understood as a manifold $\tilde{M}$ together with a submersion $\tilde{\pi}: \tilde{M}\rmap I$, so that $M_{\epsilon}$ is just the fiber of $\pi$ above $\epsilon$; similarly for families of groupoids - see Definition \ref{dfn-families-gpds}.

As we shall see, any $s$-constant deformation as above induces a deformation cocycle
\[ \xi_{0}\in C^{2}_{\textrm{def}}(\G)\]
whose cohomology class depends only on the equivalence class of the deformation. It is interesting to keep in mind that the {\it single}  cocycle $\xi_0$ encodes the variation of {\it all} the structure maps $t_{\epsilon}, m_{\epsilon},  i_{\epsilon}, u_{\epsilon}$; to be able to do that, we first have to revisit the very definition of groupoids and note that everything is encoded in the source map and the operation $\bar{m}(g, h)= gh^{-1}$; the resulting precise axioms are worked out in the Appendix. 

A similar construction applies to general deformations (i.e. not necessarily $s$-constant, and not even strict); the price to pay for the greater generality is that we will no longer have a $2$-cocycle that is canonical (i.e. independent of auxiliary choices), but only a canonical cohomology class in $H^{2}_{\textrm{def}}(\G)$. Using these cocycles/classes and the vanishing theorem stated above, we will deduce several rigidity theorems. We mention here:

\begin{thm}
Any strict deformation of a compact groupoid is trivial. 
\end{thm}

This theorem can be seen as mutual generalization of the results of Palais on rigidity and deformations of actions of Lie groups \cite{palais}, \cite{palais3},  and those of Coppersmith \cite{Copper} on deformations of Lie groups. In fact, our proposition \ref{prop: action} shows that in the case of an {\em action groupoid}, our deformation cohomology sits in a long exact sequence which relates the deformation cohomology of a Lie group (as in \cite{Copper}), and the deformation cohomology of an action of a fixed group (as in \cite{palais3}).

As we already mentioned, the deformation cohomology $H^{*}_{\textrm{def}}(\G)$ arises naturally when studying a related phenomena, namely the linearization of Lie groupoids. In essence, this is due to a very simple remark: 
for any real function $f= f(x)$ which vanishes at $0$, its linearization around $0$ can be written as a limit of rescales of $f$:
\[ f'(0)x= \lim_{\epsilon \to 0} f_{\epsilon}\ \ \ \textrm{where} \ f_{\epsilon}(x)= \frac{1}{\epsilon} f(\epsilon x).\]
A groupoid version of this is that the linearization of a Lie groupoid around a fixed point comes with a canonical (strict) deformation whose members for $\epsilon\neq 0$ are isomorphic to the original groupoid (around the fixed point). Using this and a local version of the last theorem, we immediately deduce a generalization of "Zung's linearisability theorem", which was proven recently using different methods by del Hoyo and Fernandes in \cite{matias_rui_metrics}

\begin{thm}[(Linearisation theorem)] If $\G$ is an s-proper groupoid and $N\subset M$ is invariant, then $\G$ is linearizable around $N$. 
\end{thm}

Zung's theorem corresponds to the special case when $N$ is a fixed point of $\G$ (see also \cite{ivan} for more on the relation to Zung's theorem).

The relationship between $H^{*}_{\textrm{def}}(\G)$ and deformations also give rise to variation maps for families of Lie groupoids, very much in the spirit of the Kodaira-Spencer map associated to a family of complex manifolds \cite{Kodaira} and other similar variation maps. We will be looking at families of groupoids in the following sense:

\begin{dfn}\label{dfn-families-gpds} A \emph{family of Lie groupoids parametrized by a smooth manifold $B$}, 
\[ \G \tto M \stackrel{\pi}{\rmap} B,\] 
consists of a Lie groupoid $\G$ over a manifold $M$ and a surjective submersion $\pi$ from $M$ to $B$ such that $\pi\circ s=\pi \circ t$. For $b\in B$ we will denote by $\G_{b}$ the resulting groupoid over the fiber $M_b= \pi^{-1}(b)$ of $\pi$ above $b$. We say that it is a \emph{proper family} if $\G$ is proper. 

Two families $\G \tto M \stackrel{\pi}{\rmap} B$ and $\G' \tto M' \stackrel{\pi'}{\rmap} B$ are said to be \emph{isomorphic} if 
there exists an isomorphism of groupoids $F: \G\rmap \G'$ with base map $f: M\rmap M'$ compatible with $\pi$ and $\pi'$ (i.e., $\pi'\circ f= \pi$). 
\end{dfn}

Looking at the variation of the groupoids $\G_b$ in directions of curves $\gamma$ in $B$ (i.e. applying the previous ideas to the deformations $\{\G_{\gamma(\epsilon)}\}$ of $\G_{\gamma(0)}$), we obtain the variation maps of the family, 
\[ \text{Var}_b: T_bB\rmap H^2_\text{def}(\G_b).\]
Again, it is possible to prove several rigidity results for families; we mention here the simplest one:

\begin{thm}[(Local triviality of compact families)] Any compact family of Lie groupoids is locally trivial i.e., with the previous notations, for any $b\in B$, there exists a neighbourhood $U$ of $b$ in $B$ such that the resulting family parametrized by $U$ is isomorphic to the trivial family $\G_{b}\times U$.  
\end{thm}

\subsection*{Acknowledgements}
We would like to thank Rui Loja Fernandes and Florian Sch\"atz for helpful discussions related to this paper. We would also like to thank Rui Loja Fernandes for his interpretation and formulation of the example in Remark \ref{rmk-palais}, which made the exposition much clearer.

\section{The deformation complex}\label{sec:deformation_complex}

In this section we introduce the deformation complex of a Lie groupoid.

\subsection{Some notations/terminology}
We start by fixing some notations/terminology. For a Lie groupoid $\G \tto M$, we denote by $s,t,m,u,$ and $i$ its source, target, multiplication, unit, and inversion maps respectively. When there is no danger of ambiguity, we write $m(g, h)= gh$, $i(g)= g^{-1}$ and we identify $x\in M$ with the corresponding unit $u(x)\in \G$. We also write $g: x\rmap y$ to indicate that
$g\in \G$, $x= s(g)$, $y= t(g)$. The $s$ and $t$-fibers above $x\in M$ are denoted
\[ \G(x, -)= s^{-1}(x), \ \G(-, x)= t^{-1}(x).\]
For $g: x\rmap y$ in $\G$, we consider the corresponding right translation 
\[ R_g: \G(y, -)\rmap \G(x, -)\]
and similarly the left translation $L_g$ which maps $t$-fibers to $t$-fibers. Their differentials will be denoted by $r_g$ and $l_g$, respectively. 

Recall that the Lie algebroid $A$ of $\G$ is, as a vector bundle, the restriction of $T^s\G= \textrm{Ker}(ds)$ to $M$ (pulled-back via the unit map $u: M\hookrightarrow \G$), so that
\[ A_x= T_{x} \G(x, -) \ \ \textrm{for\ all} \ x\in M.\]
The anchor of $A$ is the vector bundle map $\rho: A\rmap TM$ given by the differential of $t$. Using right translations, any $\alpha\in \Gamma(A)$ induces a right invariant vector field $\overrightarrow{\alpha}$ on $\G$ (necessarily tangent to the $s$-fibers, so that right invariance makes sense): 
\[ \overrightarrow{\alpha}(g)=r_g \alpha_{t(g)}\]
This construction identifies $\Gamma(A)$ with the space $\mathfrak{X}^{s}_{\textrm{inv}}(\G)$ of right invariant vector fields on $\G$; in turn, this induces the Lie algebra bracket $[\cdot, \cdot]$ on $\Gamma(A)$. Similarly, any $\alpha\in \Gamma(A)$ induces a left invariant vector field $\overleftarrow{\alpha}$ on $\G$, given by
\begin{equation}
\label{left-inv}
\overleftarrow{\alpha}(g)=l_g\circ di(\alpha_{s(g)}).
\end{equation}

\subsection{The deformation complex} 
We are now ready to introduce the deformation complex of a Lie groupoid. We will need the the division map $\bar{m}$ of $\G$,
\[\bar{m}(p,q) = pq^{-1}, \text{ for all } p,q \in \G, \text{ such that } s(p) = s(q);\]
Its advantage over the multiplication map $m$, especially when it comes to deformations, is explained in subsection \ref{s-constant def}. We also consider the space of strings of $k$-composable arrows
$$\G^{(k)} = \{(g_1,\ldots,g_k) : s(g_i) = t(g_{i+1}) \text{ for all } 1\leq i\leq k-1\}.$$

\begin{dfn}
\label{dfn-diff-coh}
The {\it deformation complex} $(\CG,\delta)$ of the Lie groupoid $\G$, whose cohomology is denoted $H_{\text{def}}^*(\G)$, is defined as follows. For $k\geq 1$, the $k$-cochains $c\in \CkG$ are the smooth maps 
\[ c:\Gk\longrightarrow T\G, \ \ (g_1,\ldots,g_k)\mapsto c(g_1,\ldots,g_k)\in T_{g_1}\G,\] 
which are $s$-projectable in the sense that 
\[ ds\circ c(g_1,g_2,\ldots,g_k)= : s_{c}(g_2, \ldots, g_{k})\] 
does not depend on $g_1$; the resulting $s_{c}$ is called the $s$-projection of $c$. The differential of $c\in \CkG$ is defined by
\begin{align*}(\delta c)(g_1,\ldots,g_{k+1})  = &- d\bar{m}(c(g_1g_2,\ldots,g_{k+1}),c(g_2,\ldots,g_{k+1}))  \\
&+\sum_{i=2}^k(-1)^{i} c(g_1,\ldots, g_ig_{i+1},\ldots, g_{k+1}) +(-1)^{k+1}c(g_1,\ldots, g_k).
\end{align*}
For $k= 0$, $C_{\text{def}}^0(\mathcal{G}): = \Gamma(A)$ and the differential of $\alpha\in \Gamma(A)$ is defined by
\[ \delta(\alpha)=\overrightarrow{\alpha}+\overleftarrow{\alpha}\in C_{\text{def}}^1(\mathcal{G}).\]
\end{dfn}

Note that, for $k= 0$, one may think of a section of $A$ as a map $c:\G^{(0)}=M\longrightarrow T\G$, with $c(1_x)=c_x \in T_{1_x}\G$, such that $ds\circ c_x=0_x$.

\begin{lemma} $(\CG,\delta)$ is, indeed, a cochain complex.
\end{lemma}

\begin{proof}
First of all, $\delta$ is well-defined, i.e., $\delta(c)\in C^{k+1}_{\textrm{def}}(\G)$  for $c\in C^{k}_{\textrm{def}}(\G)$; indeed, applying $ds$ to the formula for $\delta(c)(g_1, \ldots, g_{k+1})$ (and using $s(\bar{m}(a, b))= t(b)$) one finds that $\delta(c)$ is $s$-projectable with $s$-projection 
\begin{align} \label{s-projection}
s_{\delta c}(g_2, \ldots, g_{k+1})  = &- dt(c(g_2, \ldots, g_{k+1}))+  \\
& + \sum_{i= 2}^{k} (-1)^{i} s_{c}(g_2, \ldots, g_{i}g_{i+1}, \ldots, g_{k+1})+ (-1)^{k+1} s_{c}(g_2, \ldots, g_{k}). \nonumber \end{align}

For $\alpha\in \Gamma(A)$, it holds that $\delta(\alpha)$ is $s$-projectable to $\rho(\alpha)$,  the image of $\alpha$ by the anchor map,  since $ds\circ r_g=0$ and $ds\circ l_g \circ di=dt$.

To check that $\delta$ squares to zero, note that after cancelling the pairs of terms with opposite signs, the expression $\delta(\delta c)(g_1,\ldots,g_{k+2})$ becomes
\begin{align*} 
\delta(\delta c)(g_1, \ldots & , g_{k+2})= \\   &= d\bar{m}[d\bar{m}(c(g_1g_2g_3,\ldots,g_{k+2}),c(g_3,\ldots,g_{k+2})),d\bar{m}(c(g_2g_3,\ldots,g_{k+2}),c(g_3,\ldots,g_{k+2}))]\\
&+d\bar{m}\left(\sum_{i=3}^{k+1}(-1)^{i}c(g_1g_2,\ldots,g_ig_{i+1},\ldots,g_{k+2}),\sum_{i=3}^{k+1}(-1)^{i}c(g_2,\ldots,g_ig_{i+1},\ldots,g_{k+2})\right)\\
&-d\bar{m}(c(g_1g_2g_3,\ldots,g_{k+2}),c(g_2g_3,\ldots,g_{k+2}))\\
&+\sum_{i=3}^{k+1}(-1)^{i+1}d\bar{m}(c(g_1g_2,\ldots, g_ig_{i+1},\ldots, g_{k+2}),c(g_2,\ldots, g_ig_{i+1},\ldots, g_{k+2})). \\
\end{align*}
At this point, using the associativity axiom of the division map, i.e., $$\bar{m}(\bar{m}(g,k),\bar{m}(h,k))=\bar{m}(g,h)$$ in the first line of the expression, and linearity of $d\bar{m}$ in the second line, we see that the first and second lines become the symmetric of the third and fourth, respectively.
\end{proof}

\begin{example}
We will see several (classes of) examples throughout the paper. Let us mention here the simplest one: when $\G$ is a Lie group $G$ (hence $M = \{\ast\}$ is a point). Then, using the trivialization 
$TG\cong G\times \mathfrak{g}$ induced by right translations ($\mathfrak{g}$ being the Lie algebra of $G$) we obtain an identification of  $C^*_\mathrm{def}(G)$ with the complex $C^*(G, Ad)$ computing the differentiable cohomology
of $G$ with coefficients in the adjoint representation, hence
\[ H_{\text{def}}^*(G)\cong H^{*}(G, Ad) .\]
This is to be expected since the right hand side is the cohomology which controls deformations of Lie groups \cite{Copper}. 
\end{example}

\subsection{Differentiable cohomology} 
\label{sub-Differentiable cohomology}
For a better perspective, and for the later use, let us recall here the ordinary differentiable cohomology of Lie groupoids. Let $\G\tto M$ be a Lie groupoid; for a representation $E\rmap M$ of $\G$, the action $E_x\rmap E_y$ induced by an arrow $g: x\rmap y$ will be denoted $v\mapsto g\cdot v$.

\begin{dfn} The \emph{(differentiable) cohomology of the Lie groupoid $\G$ with coefficients in the representation $E$}, denoted $H^*(\G, E)$, is the cohomology of the complex $(C^*(\G, E), \delta)$, where 
$k$-cochains are the smooth maps
\[u:\Gk\longrightarrow E, \ \ (g_1,\ldots,g_k)\mapsto u(g_1,\ldots,g_k)\in E_{t(g_1)}\]
and the differential is given by  
\begin{align}\label{eq-diff-action}(\delta u)(g_1,\ldots,g_{k+1}) & =g_1\cdot u(g_2,\ldots,g_{k+1})  \\
&+\sum_{i=1}^k(-1)^{i} u(g_1,\ldots, g_ig_{i+1},\ldots, g_{k+1})\nonumber \\ &+(-1)^{k+1} u(g_1,\ldots, g_k).\nonumber
\end{align}
\end{dfn}

Note that, in degree zero, $H^0(\G, E)= \Gamma(E)^{\textrm{inv}}$ is the space of sections of $E$ that are invariant with respect to the action of $\G$. When $E$ is the trivial line bundle with the trivial action, the resulting complex is denoted by $C^*(\G)$. It comes together with a graded product - the cup product, given by
\begin{equation}\label{eq-cup-pr} 
(u\cdot v)(g_1, \ldots , g_{k+k'})= u(g_1, \ldots, g_k) v(g_{k+1}, \ldots , g_{k+ k'}) 
\end{equation}
for $u\in C^k(\G)$, $v\in C^{k'}(\G)$. The same formula makes $C(\G, E)$ into a right graded $C(\G)$-module; the fact that $E$ is a representation is encoded in the differential of $C(\G, E)$, which makes it into a (right)  $C(\G)$- DG-module. Note that, using precisely the same formulas for the cup-product and the same arguments, one has:

\begin{lemma}\label{def-DG-mod} $(C^{*}_{\mathrm{def}}(\G), \delta)$ is a (right) $(C^*(\G), \delta)$- DG-module. 
\end{lemma}

\begin{rmk}\label{for-later-use} For the later use note also that the spaces $C^k(\G, E)$ make sense for any vector bundle $E\rmap M$:
\[ C^k(\G, E)= \Gamma(\G^{(k)}, t^*E) \]
there $t$ on $\G^{(k)}$ picks the target of the first arrow. Also, whenever we have a quasi-action of $\G$ on $E$, i.e., a smooth operation that associates to any arrow $g: x\rmap y$ a linear map $\lambda_g: E_x\rmap E_y$
depending smoothly on $g$, one has an induced operator $\delta_{\lambda}$ on $C^*(\G, E)$ defined by exactly the same formulas as $\delta$, but using the quasi-action; $\delta_{\lambda}$ is still a graded derivation on the $C(\G)$-module $C(\G, E)$ (actually, any graded derivation is of type $\delta_{\lambda}$ for some quasi-action).  The associativity of the quasi-action is equivalent to $\delta_{\lambda}^{2}= 0$.
\end{rmk}

\begin{rmk}It will be often useful to consider a smaller complex computing deformation cohomology. The \emph{normalized deformation complex} of a Lie groupoid $\G$ is the subcomplex $\widehat{C}_\textrm{def}^*(\G)$ of the deformation complex $C_\textrm{def}^*(\G)$, which in degree $k\geq 2$ is composed of those cochains $c\in C_\textrm{def}^k(\G)$ satisfying $$c(1_x,g_2,\ldots,g_k)=s_c(g_2,\ldots,g_k)\qquad \text{and}\qquad c(g_1,\ldots,1_x,\ldots,g_k)=0.$$ In degree $1$, the only condition is that $c(1_x)=s_c(x)$, and in degree $0$ there is no condition, i.e., $\widehat{C}_\textrm{def}^0(\G)=\Gamma(A)$. It is a simple computation to check that $\widehat{C}_\textrm{def}^*(\G)$ is, indeed, a subcomplex - one only has to remember the expression for $s_{\delta c}$ (equation \eqref{s-projection}). 

The proof that $\widehat{C}_\textrm{def}^*(\G)$ is quasi-isomorphic to $C_\textrm{def}^*(\G)$ can be seen as a particular case of the proof of theorem \ref{thm: morita} (see remark \ref{rmk: normalized} and proposition \ref{prop: normalized}), so we postpone the discussion until then.
\end{rmk}

\begin{rmk}[(Related to the adjoint representation)]\label{1-adj} As we have already mentioned in the introduction, another guiding principle that is worth having in mind when thinking about $H_{\text{def}}^*(\G)$ is that it plays the role of ``differentiable cohomology with coefficients in the adjoint representation''. The reason for the quotes is that there is no adjoint representation in the classical sense. Actually, one of the main problems with the very notion of representation is that there are only very few representations that come for free and make sense intrinsically for all Lie groupoids. The situation is slightly better in the regular case, when one has at hand the normal and the isotropy representations, denoted $\nu$ and $\mathfrak{i}$ and recalled in the next two sections. However, even in this case, and although ``the adjoint representation'' is closely related to $\nu$ and $\mathfrak{i}$ (it contains them!), its structure is still subtle and requires the ``up to homotopy'' version of representations. This remark is the first one of a series of remarks that will guide the reader towards the full structure of the adjoint representation.
\end{rmk}

\section{First examples}

We now look at some basic examples. The order we choose is based on the simplicity of the structure of the adjoint representations involved. 

\subsection{Gauge groupoids}

Recall that any principal $G$-bundle $\pi: P\rmap M$ ($G$ is a Lie group) has an associated gauge groupoid, which is the quotient of the pair groupoid $P\times P\rmap P$ (with source and target the two projections) modulo the diagonal action of $G$:
\[ \G= P\times_{G}P \tto M .\] 
Recall also that the adjoint bundle of $P$ is defined as the vector bundle
\[ P[\mathfrak{g}]= (P\times_{G} \mathfrak{g}) \cong \textrm{Ker}(d\pi: TP/G \rmap TM).\]
This bundle will also be be discussed later on, in the general context, when it will show up as the kernel of the anchor map. In this case the Lie algebroid of $\G$ is $TP/G$, its space of sections is $\mathfrak{X}(P)^G$, the Lie algebroid bracket of $TP/G$ comes from the Lie bracket of vector fields on $P$ and the anchor is induced by the differential of $\pi$. Important for us is the fact that $\mathrm{Ker}(\rho) = P[\mathfrak{g}]$ is a representation of $\G$: indeed, any class $[p, q]\in \G$ viewed as an arrow from $x= \pi(p)$ to $y= \pi(q)$ induces the action
\[ P[\mathfrak{g}]_x\rmap P[\mathfrak{g}]_y, [p, u]\mapsto [q, u].\]

\begin{prop} If $G$ is a Lie group and $\G$ is the gauge groupoid associated to a principal $G$-bundle $P$, then 
there are canonical isomorphisms
\[ H^{*}_{\mathrm{def}}(\G)\cong H^{*}(\G, P[\mathfrak{g}]) \cong H^{*}(G, \mathfrak{g})\]
(where the last group is the differentiable cohomology of the Lie group $G$ with coefficients in its adjoint representation).
\end{prop}

There are various ways to look at this result. The first isomorphism will follow from our later general results (e.g. on the regular case); the isomorphism between the first and last groups follows directly from the Morita invariance of deformation cohomology (theorem \ref{thm: morita}); the last isomorphism can also be seen as an immediate consequence of Morita invariance of differentiable cohomology \cite{VanEst}.

\begin{rmk}[(Related to the adjoint representation)]\label{2-adj}
In the spirit of Remark \ref{1-adj}, we see that the candidate for ``the adjoint representation'' of the gauge groupoid is given by the adjoint bundle $P[\mathfrak{g}]$.  
\end{rmk}

\subsection{Foliation groupoids}
Next we look at the Lie groupoids that arise from foliation theory (such as holonomy or homotopy groupoids of foliations), i.e. which integrate foliations; here we identify a foliation with its tangent bundle and we interpret it as a Lie algebroid with the inclusion as anchor. It then makes sense to talk about the integrability of a foliation by a groupoid. We see that a foliation groupoid is a Lie groupoid $\G\tto M$ with the property that the anchor of the associated Lie algebroid is injective or, in a global formulation, with the property that the isotropy groups of $\G$ are discrete \cite{marius_ieke2}. They come with a regular foliation
$\mathcal{F}$ on $M$ (the image of the anchor); the resulting normal bundle
\[ \nu:= TM/\mathcal{F} \]
is then a representation of $\G$, where the action is given by linear holonomy (see \cite{marius_ieke2} but also our general discussion from subsection \ref{sec-The normal representation} below) which, in turn, is a global manifestation of the (foliated) Bott connection on $\nu$ (cf. e.g. \cite{Hei}). 
The resulting cohomology $H^{*}(\G, \nu)$ is the groupoid counterpart of the foliated cohomology $H^*(\mathcal{F}, \nu)$ which, in turn, was shown by Heitsch \cite{Hei} to control deformations of the foliation (the two are related by a 
Van Est map - see \cite{VanEst}). Therefore the expectation that $H^{*}(\G, \nu)$ is related to deformations of foliation groupoids, i.e. to $H^{*}_{\textrm{def}}(\G)$. Moreover, while 
deformations of Lie groupoids give rise to $2$-cocycles in deformation cohomology (think e.g. of Lie groups and see also below), deformations of foliations give rise to degree $1$ classes in the cohomology with coefficients in $\nu$; therefore one
also expects a degree shift. And, indeed:

\begin{prop} For any foliation groupoid $\G\tto M$ one has canonical isomorphisms:
\[ H^{*}_{\mathrm{def}}(\G) \cong H^{*-1}(\G, \nu)\]
where the isomorphism sends a cocycle $c\in C^{k}_{\mathrm{def}}(\G)$ into $[s_{c}]$ - the class modulo $\mathcal{F}$ of the $s$-projection of $c$ (see Definition \ref{dfn-diff-coh}).
\end{prop}

\begin{proof}
A simple computation shows that $c\mapsto [s_c]$ is a chain map. Since it is also surjective, it suffices to show that its kernel, call it $C^{*}$, is acyclic. So, assume that $c\in C^k$, i.e. $c\in C^{k}_{\textrm{def}}(\G)$  has the property that $s_c$ takes values in $\mathcal{F}$. We show that $c= \delta(c')$ for some $c'$ with the property that $s_{c'}$ takes values in $\mathcal{F}$ (we will actually achieve $s_{c'}= 0$). Namely we set
\[ c'(g_1, \ldots, g_{k-1}):= - r_{g_1}(s_c(g_1, g_2, \ldots, g_{k-1})\]
where we identify $\mathcal{F}$ with the Lie algebroid of $\G$ to make sense of right translations. It is clear that 
$s_{c'}= 0$. We are left with showing that $c= \delta(c')$. Using the fact that the map $(ds, dt): T\G\rmap TM\times TM$ is injective, it suffices to show that $ds\circ c= ds\circ \delta(c')$ and similarly for $dt$. For the first one
use again that $ds$ kills $c'$ and that $dt(r_g(\alpha))= \rho(\alpha)\cong \alpha$ for $\alpha\in \mathcal{F}$
and we see that, after applying $(ds)$ to the formula for $\delta(c')(g_1, \ldots, g_k)$ we are left with 
\[  -dt(c'(g_2, \ldots, g_k))= s_c(g_2, \ldots, g_k)= ds(c(g_1, \ldots , g_k)).\]
Hence we are left with showing that $dt\circ c= dt\circ \delta(c')$. Applying $dt$ to the formula for $\delta(c')(g_1, \ldots, g_k)$ we find 
\[ -dt(c'(g_1g_2, \ldots, g_k))+ \sum_{i=2}^{k-1}(-1)^{i} dt(c'(g_1,\ldots, g_ig_{i+1},\ldots, g_{k}))
+ (-1)^{k}dt(c'(g_1,\ldots, g_k))\]
which, by the previous arguments, is
\[ s_c(g_1g_2, g_3, \ldots, g_k))+ \sum_{i=2}^{k-1}(-1)^{i+1} s_c(g_1, g_2,\ldots, g_ig_{i+1},\ldots, g_{k})
+ (-1)^{k+1}s_c(g_1, g_2,\ldots, g_{k-1})).\]
Comparing with the formula (\ref{s-projection}) for the $s$-projection of $\delta(c)$ (which vanishes because $\delta(c)$ does), we find precisely $dt(c(g_1, \ldots, g_k))$. 
\end{proof}

\begin{rmk}[(Related to the adjoint representation)]\label{3-adj}
In the spirit of Remarks \ref{1-adj} and \ref{2-adj} we see that the candidate for the adjoint representation of $\G$ in this case is $\nu[1]$- viewed as a graded representation of $\G$ concentrated in degree $1$ (to make up for the shift in the proposition).
\end{rmk}

\subsection{Action groupoids}
One of the first classes of examples in which the deformation cohomology can be understood in terms of (differentiable) cohomology with coefficients in representations is that of action groupoids. 
So, let us assume that $G$ is a Lie group acting on a manifold $M$. Recall that the action groupoid $\G=G\lx M\arrows M$
is the product $G\times M$, with $s(g, x)= x$, $t(g, x)= gx$ and $(g, x)(h, y)= (gh, y)$. The corresponding Lie algebroid is 
\[ A= \mathfrak{g}_M= \mathfrak{g}\x M,\]
the trivial vector bundle over $M$ with fiber the Lie algebra $\mathfrak{g}$ of $G$, the anchor given by the infinitesimal action of $\mathfrak{g}$ on $M$
\[ \rho: \mathfrak{g}_{M} \rmap TM, \ \rho(v, x)= \frac{d}{d\epsilon}|_{\epsilon= 0} \exp(\epsilon v) x ,\]
and the bracket is uniquely determined by the Leibniz identity and the condition that, on constant sections $c_v$ with $v\in \mathfrak{g}$, it restricts to the bracket of $\mathfrak{g}$:
\[ [c_u, c_v]= c_{[u, v]} .\]
As is the convention with every Lie algebroid in this paper, we identify the Lie algebra of $G$ with the space of right invariant vector fields on $G$. Note that a representation $E$ of $\G$ is the same thing as an equivariant bundle over $M$; then $\Gamma(E)$ is naturally a $G$-module and $H^{*}(\G, E)$ can be interpreted as the resulting differentiable cohomology $H^{*}(G, \Gamma(E))$. 
The action groupoid $\G$ has two natural representation:
\begin{itemize} 
\item $\mathfrak{g}_M$ itself, with the action induced by the adjoint action of $G$ on $\mathfrak{g}$.
Note that 
\[ H^{*}(\G,\mathfrak{g}_M)= H^{*} (G, C^{\infty}(M, \mathfrak{g}))\] 
is a bundle-like version of $H^{*}(G;\mathfrak{g})$, which controls deformations of the Lie group $G$, as explained by Coppersmith \cite{Copper}.
\item $TM$, using the induced action of $G$ on $TM$.
Note that 
\[ H^{*}(\G,TM)= H^{*}(G, \mathfrak{X}(M))\]
arises in the work of Palais  \cite{palais3} as the cohomology controlling deformations of the Lie group action (keeping $G$ fixed). 
\end{itemize}

Therefore it is not surprising that, in this case, these cohomologies are closely related to the deformation cohomology of the groupoid: 

\begin{prop}\label{prop: action}The deformation cohomology of an action Lie groupoid $\G=G\lx M$ fits into a long exact sequence \[\cdots\longrightarrow H^{k-1}(\G,TM)\longrightarrow H^k_{\operatorname{def}}(\G)\longrightarrow H^k(\G, \mathfrak{g}_M)\stackrel{\rho_*}{\longrightarrow} H^k(\G,TM)\longrightarrow\cdots\]
where $\rho_{*}$ is induced by the infinitesimal action $\rho: \mathfrak{g}_{M} \rmap TM$. 
\end{prop}

\begin{proof}
It suffices to remark that there is a short exact sequence
\[ 0\longrightarrow C^{k-1}(\G,TM)\stackrel{j}{\longrightarrow} \CkG\stackrel{\pi}{\longrightarrow} C^k(\G,\mathfrak{g}_M)\longrightarrow 0\]
compatible with the differentials and to identify the boundary map of the long exact sequence with the map $\rho_*$. The deformation cochains of $\G$ take a composable string $(\gamma_1, \ldots, \gamma_k)$ of $\G$, with 
$\gamma_1= (g, x)$, to 
\[ T_{(g, x)}(G\times M)= T_gG\times T_xM \cong \mathfrak{g}\times T_{x}M= \mathfrak{g}\times T_{t(\gamma_2)}M\]
where the isomorphism is the one induced by the right translations of $G$. Moreover, the $TM$-component of the cochain, since it is the projection by the source map, depends only on $(\gamma_2, \ldots, \gamma_k)$; in other words, we obtain a decomposition 
\[ C^{k}_{\textrm{def}}(\G)\cong C^k(\G,\mathfrak{g}_M)\oplus C^{k-1}(\G,TM) .\]
This is not compatible with the differentials, but the natural inclusion on the first component (called $j$) and projection on the second component (called $\pi$) are (a simple computation) - therefore the desired short exact sequence. The identification of $\rho_*$ with the boundary map is straightforward. Or, a bit more conceptually: 
computing the differential of $C^{k}_{\textrm{def}}(\G)$ with respect to the previous decomposition, one finds the formula $(c_1, c_2)\mapsto (\partial(c_1), - \rho_*(c_1)- \partial(c_2))$, i.e. $C^{*}_{\textrm{def}}(\G)$ is isomorphic to the mapping cone of $\rho_*$.  
\end{proof}

\begin{rmk}[(Related to the adjoint representation)]\label{4-adj}
Note that the previous proof shows that $H^{*}_{\operatorname{def}}(\G)$ is isomorphic to the differentiable cohomology of $\G$ with coefficients in the cochain complex of representations:
\begin{equation}\label{adj-action}
 \mathfrak{g}_M\stackrel{\rho}{\rmap} TM,
\end{equation}
($\mathfrak{g}_M$ in degree $0$ and $TM$ in degree $1$). Hence, a remark on the adjoint representation similar to Remarks \ref{2-adj} and \ref{3-adj} shows that we have to look at even more general structures: cochain complexes of representations. Then (\ref{adj-action}) becomes the natural candidate for the adjoint representation in this case. 
\end{rmk}

\subsection{Bundles of Lie groups}

We now take a closer look at the case of bundles of Lie groups parametrized by $M$; these correspond to Lie groupoids $\G$ over $M$ for which the source map coincides with the target map; we denote them by 
\[ \pi: \G\rmap M.\]
Its Lie algebroid is just the bundle of Lie algebras $\mathfrak{g}$ consisting of the Lie algebras $\mathfrak{g}_x$ of the Lie groups $\G_x$ (and the anchor is zero). A representation of $\G$ is then just a collection of representations of each of the groups $\G_x$ which are smoothly parametrized by $x\in M$ and fit into a vector bundle over $M$.There are two such representations of $\G$ that will be relevant for us: $\mathfrak{g}$ itself, endowed with the (fiberwise) adjoint actions and $TM$ with the trivial action. The resulting cohomologies are related by a certain ``curvature map'' 
\[ K : H^*(\G, TM) \rmap H^{*+2}(\G, \mathfrak{g})\]
which arises by evaluating a cohomology class with values in the $\textrm{Hom}$-bundle:
\[ \textrm{Var}\in H^2(\G, \textrm{Hom}(TM, \mathfrak{g})) .\]
$\textrm{Var}$ arises as the obstruction to the existence of a (Ehresmann) connection on $\G$ which is compatible with the multiplication- in the sense that the induced parallel transport respects the group structure on the fibers. We call such connections multiplicative. To construct the class $\textrm{Var}$ one starts with an arbitrary connection
on $\pi: \G\rmap M$, interpreted as a splitting $\sigma_g: T_{\pi(g)}M\rmap T_g\G$ of $d\pi$ (smooth in $g\in \G$). The fact that $\sigma$ is multiplicative is equivalent to the condition that, for any $x\in M$, $g, h\in \G_x$, $X_x\in T_xM$, one has 
\[ \sigma_{gh}(X_x)= (dm)(\sigma_g(X_x), \sigma_h(X_x)) .\]
For our general $\sigma$ (multiplicative or not), the difference between the two terms lives in the vertical space of $\pi$ at $gh$, hence it is obtained by right translations (in $\G_x$) of an element in $\mathfrak{g}_x$; that element defines
\[ \textrm{Var}_{\sigma}(g, h)(X_x)\in \mathfrak{g}_x.\]
It is not difficult to check now that $\textrm{Var}_{\sigma}$ is a differentiable cocycle, whose cohomology class does not depend on $\sigma$ and whose vanishing is equivalent to the existence of a multiplicative connection. This defines $\textrm{Var}$.

Note also that $\textrm{Var}$ induces at each $x\in M$ a linear map
\[ \textrm{Var}_{x}: T_xM\rmap H^2(\G_x, \mathfrak{g}_x) \ (\ldots = H^{2}_{\mathrm{def}}(\G_x))\]
which encodes the variations of the group structure along directions in $M$ (to be explained in subsection \ref{subsec-var_map} for general Lie groupoids).

\begin{prop} For any bundle of Lie groups $\pi: \G\rmap M$, interpreted as a groupoid with source and target equal to $\pi$, the deformation cohomology fits into a long exact sequence
\[ \cdots \rmap H^{k}(\G, \mathfrak{g})\stackrel{r}{\rmap} H^{k}_{\mathrm{def}}(\G) \rmap H^{k-1}(\G, TM)
\stackrel{K}{\rmap} H^{k+1}(\G, \mathfrak{g}) \rmap \cdots .\]
\end{prop}

\begin{proof} Again, one has a short exact sequence
\[ C^k(\G, \mathfrak{g})\stackrel{r}{\rmap} C^{k}_{\textrm{def}}(\G) \stackrel{\pi}{\rmap} C^{k-1}(\G, TM)\]
where $\pi$ takes the base component of a cochain, so that the kernel of $\pi$ consists of $c\in C^{k}_{\textrm{def}}(\G)$ which take values in the vertical spaces of $\pi$ i.e., modulo right translations, come from $C^k(\G, \mathfrak{g})$. It is straightforward to check that $\pi$ and $r$ are chain maps (and it will be discussed in the more general context also later on). Hence we are left with identifying the boundary map in the 
induced long exact sequence,
\[ \partial: H^{k-1}(\G, TM)\rmap H^{k+1}(\G, \mathfrak{g}).\]
Consider a cocycle $u\in C^{k-1}(\G, TM)$. By the definition of $\partial$, one has to write $u= \pi(c)$ for some $c$, $\delta(c)$ comes, via $r$, from a cocycle $v\in C^{k+1}(\G, \mathfrak{g})$ and then $\partial([u])= [v]$. Fixing a connection $\sigma$ on $\G$ there is a canonical choice for $c$:
\[ c(g_1, \ldots, g_{k})= \sigma_{g_1}(u(g_2, \ldots, g_k)).\]
Compute now $\delta(c)(g_1, \ldots, g_{k+1})$. The first component in the resulting sum is
\[ - d\bar{m}(\sigma_{g_1g_2}(u(g_3, \ldots, g_{k+1}), \sigma_{g_2}(u(g_3, \ldots, g_{k+1})).\]
To handle this, note that, in general, for any pair of composable arrows $(g, h)$ and $X\in T_{s(h)}M$, 
\begin{align*}d\bar{m}(\sigma_{gh}(X), \sigma_h(X)) & = d\bar{m}(dm(\sigma_{g}(X), \sigma_h(X))+ r_{gh}(\textrm{Var}_{\sigma}(g, h)(X), \sigma_h(X))  \\
&= d\bar{m}(dm(\sigma_{g}(X), \sigma_h(X)), \sigma_h(X))+ d\bar{m}(r_{gh}(\textrm{Var}_{\sigma}(g, h)(X), 0_h)\\
&= \sigma_{g}(X)+ r_g\textrm{Var}_{\sigma}(g, h)(X)
\end{align*}
where for the first equality we have used the definition of $\textrm{Var}_{\sigma}$ and for the last one 
the differentiated identities $\bar{m}(m(a, b), b)= b$ and $\bar{m}(agh, h)= ag= R_g(a)$. Hence the first in the  
sum from $\delta(c)$ is 
\[- \sigma_{g_1}(u(g_3, \ldots, g_{k+1}))- r_{g_1}\textrm{Var}_{\sigma}(g_1, g_2)(u(g_3, \ldots, g_{k+1})).\]
The other terms are 
\[ \sigma_{g_1}(u(g_2g_3, \ldots, g_{k+1})+ \ldots+ (-1)^{k+1}u(g_1, \ldots , g_k)).\]
Hence, adding up and using that $\delta(u)$ is zero, we find that
\[ v(g_1, \ldots, g_{k+1})= r_{g_1^{-1}}(\delta(c)(g_1, \ldots, g_{k+1})= - \textrm{Var}_{\sigma}(g_1, g_2)(u(g_3, \ldots, g_{k+1})).\]
\end{proof}

\begin{rmk}[(Related to the adjoint representation)] \label{5-adj}
It is not so clear anymore how to continue the series of remarks \ref{2-adj}, \ref{3-adj} and \ref{4-adj} on the adjoint representation, so that it applies also to bundles of Lie groups. The trouble comes from the presence of variation (curvature). Indeed, while it is clear that the relevant graded representation is $\mathfrak{g}[0]\oplus TM[1]$ (with the zero differential), it is not so clear how to interpret $K$. Already at this point, for this very simple class of examples, we need the notion of representation up to homotopy (still to be recalled!). 
\end{rmk}

\subsection{Relation to Poisson geometry}\label{poisson}

Recall that a Poisson manifold is a manifold $M$ equipped with a bivector $\pi \in \Lambda^2(T^*M)$ such that the Poisson bracket defined by $\{f,g\}=\pi(df,dg)$ satisfies the Jacobi identity on $C^\infty(M)$. It gives rise to a Lie algebroid structure on $T^*M$ with anchor $\pi^\sharp$ given by $\beta(\pi^\sharp(\alpha))=\pi(\alpha,\beta)$ for all 1-forms $\alpha$ and $\beta$ and Lie bracket given by \[[\alpha,\beta]_\pi=L_{\pi^\sharp(\alpha)}(\beta)-L_{\pi^\sharp(\beta)}(\alpha)-d(\pi(\alpha,\beta)).\] The global counterpart of a Poisson manifold (whenever it exists) is a symplectic groupoid, i.e., a Lie groupoid $\Sigma$ equipped with a symplectic form $\omega\in \Omega^2(\Sigma)$ which is multiplicative: \[m^*\omega=pr_1^*\omega+pr_2^*\omega.\]

On the other hand, given any vector bundle $A$, there is a 1-1 correspondence between Lie algebroid structures on $A$ and Poisson structures on the dual vector bundle $A^*$ which are linear along the fibres. A Lie groupoid $\G$ with Lie algebroid $A$ gives rise to a symplectic groupoid integrating $A^*$ - namely $T^*\G$ (see \cite{GS} for the groupoid structure on $T^*\G$). Moreover, deformations of $\G$ give rise to deformations of the Poisson structure of $A^*$, which are controlled by the Poisson cohomology $H^*_\pi(A^*)$, which is closely related to the differentiable cohomology $H^*(T^*\G)$ via a van Est map (see \cite{marius_ieke}). Hence one expects a relation between $C^{*}_{\textrm{def}}(\G)$ and the differentiable cohomology complex $C^{*}(T^*\G)$ (an inclusion!). We obtain in this way, the complex described by Gracia-Saz and Mehta in \cite{mehta}.

Indeed, $C^{*}_{\textrm{def}}(\G)$ can be identified with the subcomplex $C_\mathrm{proj}^{*}(T^*\G)$ of $C^{*}(T^*\G)$ consisting of \textit{left-projectable linear cochains}. This subcomplex arises in \cite{mehta} by looking at the VB-groupoid interpretation of the adjoint representation, where $C_\mathrm{proj}^{*}(T^*\G)$ is introduced as the VB-groupoid complex of the groupoid $T\G$. Let us now look in more detail at the relation between these complexes.


First of all, there is a natural subcomplex $C^*_\mathrm{lin}(T^*\G)$ of $C^{*}(T^*\G)$, where $k$-cochains are those which are linear in $(T^*\G)^{(k)}$. Inside $C^*_\mathrm{lin}(T^*\G)$ there is the space of left-projectable linear cochains $C^*_\mathrm{proj}(T^*\G)$, where a linear $k$-cochain $u$ is called left-projectable if it satisfies 2 conditions: 

\begin{enumerate}

\item{the value of $u(\xi_1,\ldots,\xi_k)$ depends only on $\xi_1$ and on the base points $g_2,\ldots, g_k$, i.e.,  $$u(0_h,\xi_1,\ldots,\xi_{k-1})=0;$$}
\item{the cochain $u$ is left-invariant in the first argument, in the sense that $$u(0_h\cdot \xi_1,\ldots,\xi_k)=u(\xi_1,\ldots,\xi_k),$$}
\end{enumerate}
for any $(\xi_1,\ldots,\xi_k)\in (T^*\G)^{(k)}$ such that $\xi_i\in T^*_{g_i}\G$ and $h\in \G$ such that $(0_h,\xi_1)\in (T^*\G)^{(2)}$,

It is not hard to check from the definitions that the differential of $C^{*}(T^*\G)$ restricts to $C^*_\mathrm{lin}(T^*\G)$ and to $C^*_\mathrm{proj}(T^*\G)$ making them into subcomplexes. Putting together proposition 5.5 and theorem 5.6 of \cite{mehta} (in the case where the VB-groupoid of {\it loc. cit.} is $T\G$), one obtains:


\begin{prop} There is an isomorphism of right $(C^*(\G),\delta)$-DG-modules $$\phi: C^{*}_{\mathrm{def}}(\G) \rmap C_\mathrm{proj}(T^*\G)$$ given by \begin{equation*}\phi(c)(\xi_1,\ldots,\xi_k)=\xi_1(c(g_1,\ldots,g_k))
\end{equation*} for any $\xi_1,\ldots,\xi_k\in T^*\G^{(k)}$ such that $\xi_i\in T^*_{g_i}\G$.
\end{prop}

\section{Low degrees}

Here we look at the deformation cohomology in low degrees (mainly $0$ and $1$). This will already bring into the discussion the two natural representations of $\G$, namely the isotropy and the normal representation (both part of the adjoint representation!), and will also reveal the presence of curvature. Throughout this section we fix a Lie groupoid $\G\tto M$ with Lie algebroid denoted by $A$. 

\subsection{The isotropy representation} 
\label{sec-The isotropy representation}
The discussion in degree zero will require the isotropy representation of $\G$. First of all, consider the isotropy bundle 
 \[ \mathfrak{i} = \mathfrak{i}_A: = \textrm{Ker}(\rho: A\rmap TM).\]
The bracket of $A$ induces a Lie algebra bracket on each fiber $\mathfrak{i}_x$ and, using the groupoid $\G$,  $ \mathfrak{i}_x$ is just the Lie algebra of the isotropy group at $x$ (arrows starting and ending at $x$). Moreover, conjugation by $g: x\rmap y$ in $\G$ induces, after differentiation at units, the ``action'' of $\G$ on  $\mathfrak{i}$
\[ ad_g: \mathfrak{i}_x\rmap \mathfrak{i}_y.\]
When $\G$ is regular (in the sense that its orbits have the same dimension or, equivalently, that $\rho$ has constant rank), $\mathfrak{i}$ is a (smooth) representation of $\G$. In general, it  is only a 
set-theoretic representation of $\G$. However, one can still make sense of the space of its smooth (invariant) sections:  define $\Gamma(\mathfrak{i})$ by requiring the smoothness as
sections of $A$:
\[ \Gamma(\mathfrak{i})= \textrm{Ker}(\rho: \Gamma(A)\rmap \Gamma(TM)),\]
and then
\[ H^0(\G, \mathfrak{i})= \Gamma(\mathfrak{i})^{\textrm{inv}}:= \{ \alpha\in \Gamma(\mathfrak{i}): ad_g(\alpha(x))= \alpha(y) \ \ \forall \ g: x\rmap y\ \textrm{in}\ \G \}.\]

\begin{prop}\label{degree-0}
For any Lie groupoid $\G$, $H_{\mathrm{def}}^0(\G)\cong H^0(\G, \mathfrak{i})= \Gamma(\mathfrak{i})^{\mathrm{inv}}$.
\end{prop}

\begin{proof}
We have to show that $\alpha\in \Gamma(A)$ satisfies $\overrightarrow{\alpha}+\overleftarrow{\alpha}= 0$ if and only if $\alpha$ is an invariant section of $\textrm{Ker}(\rho)$. After applying $dt$ to the equation we see that $\rho(\alpha)$ must be $0$. This implies that $(di)(\alpha_x)= - \alpha_x$ for all $x\in M$ and then the condition $r_g(\alpha_y)+ l_g(di)(\alpha_x)= 0$ for $g: x\rmap y$ becomes the invariance of $\alpha$. 
\end{proof}

\begin{rmk} The action of $\G$ on $ \mathfrak{i}$ has as infinitesimal counterpart the action of $A$ given by the Lie bracket. The elements of  $\Gamma( \mathfrak{i})$ that are invariant with respect to the infinitesimal action (i.e. $\alpha\in \Gamma( \mathfrak{i})$ with $[u, \alpha]= 0$ for all $u\in \Gamma(A)$) are precisely the ones in the center $Z(\Gamma(A))$ of the Lie algebra $\Gamma(A)$.  Hence, by the standard arguments, we have $\Gamma( \mathfrak{i})^{\textrm{inv}}\subseteq Z(\Gamma(A))$ and equality holds when  $\G$ has connected s-fibers.
\end{rmk}

\subsection{$H^*(\G, \mathfrak{i})$ and its contribution to $H^{*}_{\textrm{def}}(\G)$} 
Even when $\mathfrak{i}$ is not of constant rank (i.e. $\G$ is not regular), one can still make sense of the differentiable cohomology $H^*(\G, \mathfrak{i})$. The defining complex $C^*(\G, \mathfrak{i})$ is defined as
in Subsection \ref{sub-Differentiable cohomology}, where the smoothness of the cochains is obtained by interpreting them as $A$-valued. The fact that the differential of $C^{*}(\G, \mathfrak{i})$ preserves smoothness follows e.g. from the relationship with the deformation complex: one has an inclusion (compatible with the differentials!): 
\begin{equation}\label{inj0}
r: C^{*}(\G, \mathfrak{i}) \rmap C_{\text{def}}^*(\G)
\end{equation}
which associates to a differentiable cochain $u$ with values in 
$\mathfrak{i}$ the deformation cochain $c_u$ given by
\[ c_u(g_1, \ldots, g_k):= r_{g_1}(u(g_1, \ldots , g_k)).\]
The induced map in cohomology,
\begin{equation}\label{inj}
r: H^{*}(\G, \mathfrak{i}) \rmap H_{\text{def}}^*(\G)
 \end{equation}
will be shown to be injective in degree $1$, but may fail to be so in higher degrees.

\subsection{Degree $1$ and multiplicative vector fields}
We now start looking at degree $1$ by re-interpreting $H_{\text{def}}^1(\G)$ in terms of multiplicative vector fields. Recall that, for any groupoid $\G\tto M$, $T\G\tto TM$ is canonically a groupoid with structure maps the differentials of the structure maps of $\G$; with this, a vector field $X\in \mathcal{\G}$ is called \textbf{multiplicative} if, as a map $X: \G\rmap T\G$, it is a morphism of groupoids. In other words,
$X$ must be projectable both along $s$ as well as along $t$ to some vector field $V\in \mathfrak{X}(M)$, $du(V_x) = X_{u(x)}$ and 
\[ X_{gh} = dm(X_g, X_h)\]
for any pair $(g, h)$ of composable arrows. A particular class of multiplicative vector fields are those of type
$\overrightarrow{\alpha}+\overleftarrow{\alpha}$ with $\alpha \in \Gamma(A)$ - whose flows give rise to inner automorphisms of $\G$; therefore they are called \textbf{inner multiplicative vector fields}.
  
\begin{prop} For any Lie groupoid $\G$ one has
\[H_{\mathrm{def}}^1(\G) = \frac{\text{\emph{multiplicative vector fields on} }\G}{\text{\emph{inner multiplicative vector fields on }}\G}.\]
\end{prop}

\begin{proof} One has to check that a vector field $X\in \mathfrak{X}(\G)$ is a cocycle in $C^{1}_{\textrm{def}}(\G)$ if and only if, as a map $\G\rmap T\G$, it is a groupoid morphism; but this is an immediate consequence of the characterization of groupoid morphisms from the Appendix (Corollary \ref{crl: m bar}). 
\end{proof}

Note also that the multiplicativity of $X$ is equivalent with the fact that the flow of $X$ is compatible with the groupoid structure - and this is how multiplicativity (and $H^{1}_{\textrm{def}}$) will come in the proof of rigidity results. So, for later use, here is a more precise statement. We will use the following general notations:  
for a vector field $X\in \mathfrak{X}(M)$ we denote by $\phi_{X}$ its flow, $\phi_{X}^{\epsilon}(x)= \phi_{X}(x, \epsilon)$ and by 
\[ \mathcal{D}(X)= \{(x, \epsilon)\in M\times \mathbb{R}: \phi_{X}(x, \epsilon) \ \textrm{is\ defined}\}\subset M\times\mathbb{R} \]
its domain. For each $\epsilon> 0$, we consider the
$\epsilon$-slice $\mathcal{D}_{\epsilon}(V)\subset M$ ($x\in M$ with the property that
the integral curve of $V$ through $x$ is defined at time $\epsilon$). The following will be needed later on.

\begin{lemma}\label{lemma-on-flows} If $\G \tto M$ is a Lie groupoid and $X\in \mathfrak{X}(\G)$ is multiplicative, then the flow of $X$ preserves the groupoid structure, wherever defined. 
More precisely, denoting by $V\in \mathfrak{X}(M)$ the base field of $X$, then, for any $\epsilon \geq 0$,  $\mathcal{D}_{\epsilon}(X)\subset \G$ is an open subgroupoid of $\G$ with base $\mathcal{D}_{\epsilon}(V)$, and 
\[ \phi_{X}^{\epsilon}: \mathcal{D}_{\epsilon}(X)\rmap \G\]
 is a morphism of groupoids covering the flow of $V$.

If $\G$ is proper then, moreover,  
\[ \mathcal{D}_{\epsilon}(X)= \G|_{\mathcal{D}_{\epsilon}(V)},\]
i.e. the flow $\phi_{X}^{\epsilon}(g)$ is defined precisely when $\phi_{V}^{\epsilon}(s(g))$ and $\phi_{V}^{\epsilon}(t(g))$ are (and this holds under the weaker hypothesis on $X$ that it is $s$ and $t$-projectable to some $V\in \mathfrak{X}(M)$).

\end{lemma}

\begin{proof}
The first part is well-known (see \cite{macxu}). For the second part, the inclusion $``\subset$'' is clear; we prove the reverse one. Let $g\in \G$, consider  the maximal integral curve $\gamma_s$
of $V$ through $s(g)$, similarly $\gamma_t$, and let $(a, b)$ be the intersection of the domains of $\gamma_s$ and $\gamma_t$. Denoting by $I$ the domain of the maximal curve $\gamma$ of $X$ through $g$, we know that $I\subset (a, b)$ and we have to show that equality holds. We show that for all $(u, v)\subset I$ with $a< u\leq v< b$, one must have $[u, v]\subset I$: indeed, for such $u$ and $v$, 
\[ \gamma((u, v))\subset \{ a\in \G: s(a)\in \gamma_s([u, v]), t(a)\in \gamma_t([u, v]) \} \] 
where the last subspace of $\G$ is compact because $\G$ is proper. Since $\gamma((u, v))$ is relatively compact, it follows that $[u, v]$ is contained in $I$.
\end{proof}

\subsection{The normal representation $\nu$ and $H^{0}(\G, \nu)$} 
\label{sec-The normal representation}
Next we have a closer look at $H_{\text{def}}^1(\G)$ in a way similar to that from Proposition \ref{degree-0}. The discussion will bring into the picture another important ``representation'' of $\G$ - the normal one (another piece of the adjoint representation!). Consider the normal bundle 
\[ \nu:= \textrm{Coker}(\rho)= TM/\rho(A).\]
As in the case of $\mathfrak{i}$, this is a smooth vector bundle only in the regular case, but, in general, we can still talk about its fibers and, as we shall see, make sense of ``its space of smooth (invariant) sections''. First of all, still as for $\mathfrak{i}$, any arrow $g: x\rmap y$ of $\G$ induces an action
\[ ad_g: \nu_x\rmap \nu_y .\]
Explicitly, for $v\in \nu_x$, one chooses a curve $g(\epsilon): x(\epsilon)\rmap y(\epsilon)$ in $\G$ with $g(0)= g$ and such that $\dot{x}(0)\in T_{x}M$ represents $v$, and then $ad_g(v)$ is the class of
 $\dot{y}(0)\in T_{y}M$. One can check that this construction does not depend on the choices involved. The following is immediate:

\begin{lemma}\label{invariance-nu} For any $g: x\rmap y$, the action $ad_g: \nu_x\rmap \nu_y$ is uniquely determined by the condition that,  for any vector $X_g\in T_g\G$, it sends the class of $ds(X_g)$ modulo $\rho(A)$ to that of  $dt(X_g)$. 
In particular, for $V\in \mathfrak{X}(M)$, the condition that 
\[ M\ni x \mapsto [V_x]\in \nu_x \]
is invariant is equivalent to the fact that for a (any) $s$-lift $X\in \mathfrak{X}(\mathcal{\G})$ of $V$ and any $g: x\rmap y$, there exists $\eta(g)\in A_{t(g)}$ such that
\begin{equation}\label{inv-eqt}
V_{t(g)}= dt(X_g)+ \rho(\eta(g)).
\end{equation}
\end{lemma}

In the general (i.e. the possibly non-regular) case, since $\nu$ is a quotient, making sense of (the space of) smooth sections of $\nu$ is more subtle than for $\mathfrak{i}$; and the same for defining  $H^*(\G, \nu)$ - and here we will only take care of degree zero, i.e. making sense of ``invariant sections''.
First of all, one defines
\[ \Gamma(\nu):= \mathfrak{X}(M)/\textrm{Im}(\rho) .\]
Note that any $[V]\in \Gamma(\nu)$ induces a set theoretic section $M\ni x\mapsto [V_x]\in \nu_x$, but the two objects are now different. Similarly, the invariance of $[V]\in \Gamma(\nu)$ is a stronger (or better: a smooth version) of the pointwise invariance
of the induced set theoretical section. More precisely, with the last part of the previous lemma in mind, it is natural to define the invariance of $[V]\in \Gamma(\nu)$ by requiring that for a/any $s$-lift $X\in\mathfrak{X}(\G)$ of $V$, the invariance equation (\ref{inv-eqt}) holds for some smooth section 
$\eta$ over $\G$ of $t^*A$. Replacing $X$ by $X'$ given by $X_{g}'= X_g+ r_g(\eta(g))$, we arrive at the following:

\begin{dfn} We say that $[V]\in \Gamma(\nu)$ is invariant if there exists a vector field $X\in \mathfrak{X}(\G)$ which is both $s$ as well as $t$-projectable to $V$ - in which case we say that $X$ is an $(s, t)$-lift of $V$. The resulting space of invariant elements is denoted
\[ H^0(\G, \nu)= \Gamma(\nu)^{\mathrm{inv}}\subset \Gamma(\nu).\]
\end{dfn}

From the previous discussion, it is clear that, in the regular case, one recovers the usual space of sections of the smooth bundle $\nu$ and its invariant sections. In general, we obtain:

\begin{lemma} One has a natural linear map
\[ \pi: H^{1}_\mathrm{def}(\G)\rmap \Gamma(\nu)^{\mathrm{inv}}\] 
which associates to a multiplicative vector field $X$ on $\G$ the class modulo $\textrm{Im}(\rho)$ of the vector field on $M$ associated with $X$.
\end{lemma}

\begin{rmk}[(Related to the adjoint representation)]
It is instructive to keep in mind the (intuitive for now) interpretations of this discussion in terms of the adjoint representation. As indicated by the previous examples, both $\mathfrak{i}$ as well as $\nu$ contribute to the adjoint representation. The fact that $H_{\text{def}}^1(\G)$ is related to $H^0(\G, \nu)$ indicates that the contribution of $\nu$ involves a degree shift by one; so, a first guess would be that the adjoint representation is represented by $\mathfrak{i}[0]\oplus \nu[1]$ (the first one in degree $0$, the second one in degree $1$). There are two remarks to have in mind here:
\begin{itemize}
\item  in order to stay within smooth vector bundles also in the non-regular case, one should think of $\mathfrak{i}[0]\oplus \nu[1]$ as the cohomology of the length two complex 
\[ 0\rmap A\stackrel{\rho}{\rmap} TM\rmap 0\]
with $A$ in degree zero and $TM$ in degree $1$. The idea is to think of such complexes of vector bundles as being smooth representatives of their (possibly non-smooth) cohomology bundles, and then work with such complexes ``up to homotopy'' (quasi-isomorphisms).
\item however, even in the regular case when no smoothness issues arise and one could (try to) use the graded representation $\mathfrak{i}[0]\oplus \nu[1]$, the resulting cohomology (in degree $1$) would surject onto $H^0(\G, \nu)$ - which is certainly not the case for the deformation cohomology and $\pi$. This indicates a more subtle structure of the adjoint representation, related to the cokernel of $\pi$.
\end{itemize}
\end{rmk}

\subsection{The first manifestation of curvature}\label{The first manifestation of curvature}
Next, we look closer at the kernel and cokernel of $\pi$. As we shall see, this will bring the cohomology $H^{*}(\G, \mathfrak{i})$ back to our attention. Looking for the cokernel reveals the presence of ``curvature''. The precise meaning of ``curvature'' will become clear later on (see also the next remark); here we note its manifestation on the cohomology of lower degrees.

\begin{lemma}
For $[V]\in \Gamma(\nu)^{\mathrm{inv}}$, choosing an $(s, t)$-lift $X\in \mathfrak{X}(\G)$, $\delta(X) \in C^{2}_{\mathrm{def}}(\G)$ takes values in the subcomplex (see (\ref{inj0}))
\[ C^2(\G, \mathfrak{i})\stackrel{r}{\hookrightarrow} C^{2}_{\mathrm{def}}(\G),\]
and it defines a cocycle in $H^2(\G, \mathfrak{i})$ whose cohomology class does not depend on the choice of $X$; hence one has an induced linear map (the cohomological curvature in degree $0$)
\[ K: \Gamma(\nu)^{\mathrm{inv}}\rmap H^{2}(\G, \mathfrak{i}).\]
\end{lemma}

\begin{proof} Note that the inclusion (\ref{inj0}) identifies $C^*(\G, \mathfrak{i})$ with the subcomplex of $C^{*}_{\textrm{def}}(\G)$ consisting of deformation cochains that take values in $\textrm{Ker}(ds)\cap \textrm{Ker}(dt)$; therefore we will work only inside the deformation complex. Computing $ds(\delta(X)(g, h))$ we find
\[ ds(d\bar{m}(X_{gh}, X_{h}))- ds(X_g)=  dt(X_h)- ds(X_g)= V_{t(h)}- V_{s(g)}= 0\]
and similarly for $dt(\delta(X)(g, h))$; hence $\delta(X)$ lives in the subcomplex. Moreover, if $X'$ is another $(s, t)$-lift, then $c:= X'- X$ is in the subcomplex hence $\delta(X')= \delta(X)+ \delta(c)$ represent the same class in  
$H^{2}(\G, \mathfrak{i})$. Finally, this class only depends on the class $[V]$. Indeed, if $[V]= [V']$, then we find $\alpha\in \Gamma(A)$ such that $V'= V+ \rho(\alpha)$; then, if $X$ is an $(s, t)$-lift of $V$, we can use $X'= X+ \overrightarrow{\alpha}+\overleftarrow{\alpha}= X+ \delta(\alpha)$ as the $(s, t)$-lift of $V'$ but then $\delta(X')= \delta(X)$. 
\end{proof}

\begin{rmk}[(Related to the adjoint representation)] Continuing the previous remark on the adjoint representation, the presence of $K$  is a manifestation of the more subtle structure of the adjoint representation: it shows that there is a certain interaction between 
$\nu$ and $\mathfrak{i}$, that allows one to move from $\nu$ to $\mathfrak{i}$, via a differentiable 2-cocycle (hence a 2-cocycle with values in $\textrm{Hom}(\nu, \mathfrak{i})$ in the regular case, and some kind of cocycle
with values in $\textrm{Hom}(TM, A)$ in general).
\end{rmk}

Putting all the maps together, we have:

\begin{prop}
\label{the-LOS}
One has an exact sequence
\[ 0 \rmap  H^1(\G, \mathfrak{i}) \stackrel{r}{\rmap} H_{\mathrm{def}}^1(\G) \stackrel{\pi}{\rmap}  \Gamma(\nu)^{\mathrm{inv}} \stackrel{K}{\rmap} H^2(\G, \mathfrak{i}) \stackrel{r}{\rmap} H_{\mathrm{def}}^2(\G).\]
\end{prop}

\begin{proof} For the injectivity of the first map $r$, assume that $c\in C^{1}_{\textrm{def}}(\G)$ comes from 
$C^1(\G, \mathfrak{i})$ (i.e. $c(g)$ is killed by both $ds$ and $dt$, for all $g$) and it is also exact. Hence $c= \delta(\alpha)$ for some $\alpha\in \Gamma(A)$; since $dt$ sends $\delta(\alpha)= \overrightarrow{\alpha}+\overleftarrow{\alpha}$ to $\rho(\alpha)$, we see that $\alpha\in \Gamma(\mathfrak{i})$ hence
$c$ is a coboundary in the sub-complex $C^*(\G, \mathfrak{i})$. 

For the exactness at $H_{\text{def}}^1(\G)$, compute the kernel of $\pi$: it consists of classes $[X]$, where $X$ is a multiplicative vector field on $\G$ with the property that its base field $V$ is zero in $\Gamma(\nu)$, i.e. it is of type $V= \rho(\alpha)$ for some $\alpha\in \Gamma(A)$; replacing $X$ by $X- \delta(\alpha)$, we deal with classes $[X]$ with the property that the base field is zero, i.e. coming from the inclusion (\ref{inj0}). 

Next, we take care of the kernel of $K$: it consists of classes $[V]\in \Gamma(\nu)^{\textrm{inv}}$ with the property that, choosing $X\in \mathfrak{X}(\G)$ an $(s, t)$-lift of $V$, one has $\delta(X)= \delta(c)\in C^{2}_{\textrm{def}}(\G)$, for some $c$ that is killed by $ds$ and $dt$; replacing $X$ by $X- c$, we see that we deal with classes $[V]$ which admit an $(s, t)$-lift $Y$ that is closed in the deformation complex, i.e. which is multiplicative as a vector field. Hence $\textrm{Ker}(K)= \textrm{Im}(\pi)$. 

Finally, the kernel of the last map $r$: by the definition of $K$, it is clear that $r\circ K= 0$; conversely, if
$[c]\in H^2(\G, \mathfrak{i})$ is in the kernel of $r$, we have $r(c)= \delta(X)\in C^{2}_{\textrm{def}}(\G)$ for some $X\in C^{1}_{\textrm{def}}(\G)$; however, the fact that $r(c)$ (hence $\delta(X)$) takes values in the kernels of $ds$ and $dt$, implies that $X$ will be $(s, t)$-projectable to some $V\in \mathfrak{X}(M)$, hence, by the construction of $K$,  $[c]= K([V])$. 
\end{proof}

\section{Degree $2$ and deformations}
\label{degree2}

In this section we indicate the relevance of deformation cohomology to deformations of Lie groupoids by explaining how such deformations give rise to $2$-cocycles.

\subsection{The case of $(s, t)$-constant deformations} 
\label{sec-s-t-const}
Let $\{\G_{\epsilon}: \epsilon\in I\}$ be a strict deformation of $\G$ (see the introduction); we would like to study  the variation of the groupoid structure. This variation is, at least intuitively, measured by the variation of the structure maps, such as of the expressions of type $m_{\epsilon}(g, h)$ around $\epsilon= 0$. As mentioned in the appendix, to make sense of this, one encounters the problem that if $(g, h)$ are composable with respect to the original groupoid structure, they may fail to be composable for $\G_{\epsilon}$. Although the appendix indicates the way to proceed (using $\bar{m}$), let us first assume first that $s_{\epsilon}$ and $t_{\epsilon}$ do not depend on $\epsilon$ and proceed in a more classical way. 

In this case
\[ - \frac{d}{d\epsilon}|_{\epsilon= 0} m_{\epsilon}(g, h) \in T_{gh}\G \]
is well-defined for any $(g, h)\in \G^{(2)}$ and is killed by $ds$ and $dt$ (the choice of the sign will soon become clear). Being killed by $ds$ means that it lives in the image of the right translation $r_{gh}: A_{t(g)}\rmap T_{gh}\G$, while being killed also by $dt$ means that it comes from the isotropy $\mathfrak{i}_{t(g)}$. Hence we end up with a differential cochain
\[ u_0\in C^{2}(\G, \mathfrak{i})\]
with the defining property 
\[ \frac{d}{d\epsilon}|_{\epsilon= 0} m_{\epsilon}(g, h)= - r_{gh}(u_0(g, h)) \in T_{gh}\G.\]
Moreover, differentiating the associativity equation $m_{\epsilon}( m_{\epsilon}(g, h), k)=  m_{\epsilon}(g,  m_{\epsilon}(h, k))$ with respect to $\epsilon$ at $0$, we find that $u_0$ is a cocycle. As it will follow from the more general discussion, or checked directly (but see also \cite{alan} where strict deformations were first discussed): 

\begin{lemma} The resulting cohomology class 
\[ [u_0]\in H^{2}(\G, \mathfrak{i}) \]
only depends on the equivalence class of the deformation. 
\end{lemma}

Before we pass to the more general case of $s$-constant deformations, we consider the image $\xi_0$ of $u_0$ by the inclusion 
\[ r: C^{2}(\G, \mathfrak{i}) \hookrightarrow C^{2}_{\textrm{def}}(\G), \]
(see (\ref{inj0})) or, explicitly, 
\[ \xi_0(g, h)= r_g(u_0(g, h))= - r_{h^{-1}} \frac{d}{d\epsilon}|_{\epsilon= 0} m_{\epsilon}(g, h) .\]
For a more convenient formula, differentiate at $\epsilon= 0$ the identity $m_{\epsilon}(\bar{m}_{\epsilon}(m_0(g, h), h),h)= m_0(g, h)$:
\[ \frac{d}{d\epsilon}|_{\epsilon= 0} m_{\epsilon}(\bar{m}_0(gh, h), h)+ (dm_0)_{g, h}\left( \frac{d}{d\epsilon}|_{\epsilon= 0}  \bar{m}_{\epsilon}(gh, h), 0_h\right)= 0;\]
using also the fact that in any groupoid $(dm)_{g, h}(X_g, 0_h)= r_{h}(X_g)$ and applying $r_{h^{-1}}$ we find 
\[ \xi_0(g, h)=  \frac{d}{d\epsilon}|_{\epsilon= 0} \bar{m}_{\epsilon}(gh, h) .\]

\subsection{The case of $s$-constant deformations I: the direct approach} \label{s-constant def}
The advantage of the last expression is that it makes sense for all $s$-constant deformations. Of course, this is also very much in the spirit of the appendix, which teaches us that we should look at the variation of 
$\bar{m}_{\epsilon}$ (and only at it).

\begin{dfn} 
Given an $s$-constant deformation $\{\G_{\epsilon}: \epsilon\in I\}$ of $\G$, the associated deformation cocycle is  defined as
\begin{equation*}
\xi_0\in  C^{2}_{\textrm{def}}(\G) \ \textrm{given\ by}\ :\ \    \xi_0(g, h)= \frac{d}{d\epsilon}|_{\epsilon= 0} \bar{m}_{\epsilon}(gh, h) \in T_g\G .
\end{equation*}
\end{dfn}

\begin{lemma} $\xi_0$ is indeed a cocycle and its cohomology class $[\xi_0]\in H^{2}_{\mathrm{def}}(\G)$ depends only on the equivalence class of
the deformation.
\end{lemma}

\begin{proof}
The first part follows again from the associativity, but written in terms of the division: 
\[ \bar{m}_\epsilon(\bar{m}_\epsilon(u,w),\bar{m}_\epsilon(v,w))=\bar{m}_\epsilon(u,v)\]
Indeed, by differentiating at $\epsilon = 0$ we we obtain 
\begin{align*}&d\bar{m}_0\left(\dezero\bar{m}_\epsilon(u,w),\dezero\bar{m}_\epsilon(v,w)\right) +\\ &+\left(\dezero\bar{m}_\epsilon\right)(uw^{-1},vw^{-1})-\left(\dezero\bar{m}_\epsilon\right)(u,v)=0,
\end{align*}
and by choosing $u=g_1g_2g_3$, $v=g_2g_3$ and $w=g_3$, this means precisely that
\[\delta\xi_0(g_1,g_2,g_3)= d\bar{m}_0(\xi_0(g_1g_2,g_3), \xi_0(g_2,g_3))-\xi_0(g_1,g_2g_3)+\xi_0(g_1,g_2)=0.\]

The second part follows similarly. If $\phi^\epsilon: \G_\epsilon \rmap \G'_\epsilon$ is an equivalence of deformations let $$ X = \dezero \phi^\epsilon \in C^1_{\mathrm{def}}(\G).$$ Then, by differentiating the expression 
$$\bar{m}'_\epsilon(\phi^\epsilon(gh),\phi^\epsilon(h)) = \phi^\epsilon \bar{m}_\epsilon(gh,h)$$
we obtain that $$\xi'_0(g,h) - \xi_0(g,h) = \delta X(g,h).$$  
\end{proof}

\begin{example}
Let $X\in \mathfrak{X}(\G)$ be a vector field on a groupoid $\G\tto M$ and assume it is $s$-projectable to some $V\in \mathfrak{X}(M)$ (hence $X\in C^{1}_{\textrm{def}}(\G)$). To avoid irrelevant technicalities, we assume $X$ to be complete and let $\phi_{X}^{\epsilon}$ be its flow. Then one obtains a new family of groupoid structures on $\G$ by pulling-back the original structure along $\phi_{X}^{\epsilon}$; hence
\[ s_{\epsilon}= \phi_{V}^{-\epsilon}\circ s\circ \phi_{X}^{\epsilon}, \quad m_{\epsilon}= \phi^{-\epsilon}\circ m\circ (\phi_{X}^{\epsilon}, \phi_{X}^{\epsilon}), \quad \textrm{etc} .\]
This deformation is $s$-constant since $X$ is $s$-projectable, and it is trivial by the very construction. Hence we
know that the associated cocycle is exact; and, computing it, one finds precisely $\delta(X)\in C^{2}_{\textrm{def}}(\G)$. 
\end{example}

\begin{rmk}
One can remove the condition that $s_{\epsilon}$ is constant and treat general strict deformations; the price to pay for this generality is that we will no longer have a canonical $2$-cocycle, but only a canonical $2$-cohomology class.
Let us indicate how to proceed in a direct fashion (an alternative route will be described in detail, for general deformations, a bit later). On follows the obvious idea: given an arbitrary strict deformation $\G_{\epsilon}$ of $\G$, replace it by an equivalent one, $\G_{\epsilon}'$, which is $s$-constant.
That amounts to finding a family of diffeomorphisms $\phi^{\epsilon}: \G\rmap \G$ so that, when we pull-back the groupoid structure of $\G_{\epsilon}$ via $\phi^{\epsilon}$, the new structure $\G_{\epsilon}'$ has constant $s$. Assume for simplicity that $\phi^{\epsilon}$ is the identity on units, so the condition is that $s_{\epsilon}\circ \phi^{\epsilon}$ does not depend on $\epsilon$, and we look for $\phi^{\epsilon}$ of type $\phi^{\epsilon}_{X}$ - the ``flow'' of a time-dependent vector field $\tilde{X}= \{X^{\epsilon}\}$ on $\G$ (see the next paragraph for the notations). Differentiating with respect to $\epsilon$, the desired equation becomes
\begin{equation}\label{strange-eq}
(ds_{\epsilon})(X^{\epsilon}(g))+ \frac{d}{d\epsilon} s_{\epsilon}(g)= 0. 
\end{equation}
Since each $(ds_{\epsilon})$ is surjective, it is clear that such a family $X^{\epsilon}$ exists. Going backwards the only problem is the possible lack of completeness; however, the local flow is all that is needed to make sense of the $2$-cocycle associated to the resulting $m_{\epsilon}'$. The resulting cohomology class will only depend on the deformation we started with.
\end{rmk}

Next we indicate how the deformation cocycles can be used to establish rigidity results. But first some generalities on flows. By a time dependent vector field on a manifold $M$ (or 1-parameter family of vector fields on $M$) we mean a family $\tilde{X}= \{X^{\epsilon}: \epsilon\in I\}$ consisting of vector fields $X^{\epsilon}\in \mathfrak{X}(M)$ depending smoothly on $\epsilon$ in an open interval $I\subset \mathbb{R}$ containing $0$. Such a time-dependent vector field $\tilde{X}$ has a flow $\phi_{\tilde{X}}^{t, s}$ with a double dependence in the parameters; it consists of (local) diffeomorphisms
\[ \phi_{\tilde{X}}^{t, s}: M\rmap M,\]
the solutions to the system
\[ \frac{d}{dt} \phi_{\tilde{X}}^{t, s}(x)= X(t, \phi_{\tilde{X}}^{t, s}(x)),
\ \ \phi_{\tilde{X}}^{s, s}(x)= x .\]
The flow relations are $\phi_{\tilde{X}}^{s, s}= \textrm{Id}$, $\phi^{t, u}_{\tilde{X}}\phi^{u, s}_{\tilde{X}}= \phi^{t, s}_{\tilde{X}}$ - valid modulo the usual issues on the domains of definitions.
When interested only on what happens for parameters close to $0$, it suffices to consider the family of (local) diffeomorphisms of $M$ given by 
\[ \phi^{\epsilon}_{\tilde{X}}:= \phi^{\epsilon, 0}_{\tilde{X}} \]
and then, for small parameters, 
 \[ \phi_{\tilde{X}}^{s+ \epsilon, s}= \phi^{s+ \epsilon}_{\tilde{X}}\circ (\phi^{s}_{\tilde{X}})^{-1}. \]
When $X$ is autonomous $\phi^{\epsilon}_{\tilde{X}}$ is the usual flow, and $\phi_{\tilde{X}}^{s+ \epsilon, s}= \phi_{\tilde{X}}^{\epsilon}$ only depends on $\epsilon$.

The flows of type $\phi_{\tilde{X}}^{t, s}$ are suited to relate the members $\G_{s}$ and $\G_{t}$ of $s$-constant deformations:

\begin{prop}\label{rigidity-lemma} Let $\{\G_{\epsilon}: \epsilon\in I\}$ be an $s$-constant deformation of $\G$, and consider the induced deformation cocycles $\xi_{\epsilon} \in C^{2}_{\mathrm{def}}(\G_{\epsilon})$. Assume that for all $\epsilon$ small enough, one finds $X^{\epsilon}$ such that
\begin{equation}\label{cocy-eq} 
\delta(X^{\epsilon})= \xi_{\epsilon} \ \ \textrm{in}\ C^{2}_{\mathrm{def}}(\G_{\epsilon}). 
\end{equation}
and assume that the resulting time dependent vector field $\tilde{X}= \{X^{\epsilon}\}$ is smooth. Then, for $t$ and $s$ close to $0$, $\phi_{\tilde{X}}^{t, s}$ is a locally defined morphism from $\G_{s}$ to $\G_{t}$,
covering the similar flow of $V^{\epsilon}= ds(X^{\epsilon})$. 

Moreover, if $\G$ is proper, then $\phi_{\tilde{X}}^{t, s}(g)$ is defined whenever $\phi_{\tilde{V}}^{t, s}(s(g))$  and $\phi_{\tilde{V}}^{t, s}(t(g))$ are.
\end{prop}

This lemma can be proved directly but our reinterpretations from the next subsection (Proposition \ref{reinterpret-cocycle-eqs}) will show that it is a consequence of Lemma \ref{lemma-on-flows} on flows of multiplicative vector fields.

\subsection{The case of $s$-constant deformations II: reinterpretation using the groupoid $\tilde{\G}$} 
Here we provide another way of looking at the deformation cocycles, less intuitive but easier to work with. 
This is based on the reinterpretation of strict deformations of $\G$ as (germs of) families parametrized by an interval $I$: one can identify a strict deformation $\tilde{\G}= \{\G_{\epsilon}: \epsilon\in I\}$ with the groupoid 
\[ \tilde{\G}:= \G\times I\ \ \textrm{over}\ \tilde{M}:= M\times I,\]
with structure maps
\[ \tilde{s}(g, \epsilon)= (s_{\epsilon}(g), \epsilon),\ \tilde{m}((g, \epsilon), (h, \epsilon))= (m_{\epsilon}(g,h), \epsilon), \ \textrm{etc}.\]
This allows us to reinterpret the deformation cocycle associated to an $s$-constant deformation as follows.  

\begin{prop}
\label{re-write-s-def}
Let $\tilde{\G}= \{\G_{\epsilon}\}$ be an $s$-constant deformation of the Lie groupoid $\G$. Then, interpreting $\frac{\partial}{\partial \epsilon}$ as an element of $C^{1}_{\mathrm{def}}(\tilde{\G})$ and considering 
\[ \xi= -\delta\left(\frac{\partial}{\partial \epsilon}\right)\in C^{2}_{\mathrm{def}}(\tilde{\G}),\]
one has 
\[ \xi_{0}= \xi|_{\G_0} \in C^{2}_{\mathrm{def}}(\G_0).\]
\end{prop}

Note that, implicit in this statement, is also the fact that the restriction of $\xi$ to $\G_0\subset \tilde{\G}$
takes values in $T\G_0\subset T\tilde{\G}$ (warning: the operation of restricting elements of $C^{*}_{\textrm{def}}(\tilde{\G})$ to $\G_{0}$, even when it produces elements in $C^{*}_{\textrm{def}}(\G_0)$, is not compatible with the differentials).

\begin{proof}
To check the previous identity we use that, as a vector field on 
$\tilde{\G}$,
\begin{equation}\label{fr-fr-eps}
\frac{\partial}{\partial \epsilon}(g, 0)= \frac{d}{d\epsilon}|_{\epsilon= 0} (g, \epsilon)\in T_{(g, 0)}\tilde{\G}
\end{equation}
hence on elements $g\equiv (g, 0)$, $h \equiv (h, 0)$, 
\begin{align*}
\xi((g, 0), (h, 0)) & = (d\bar{\tilde{m}})\left(\frac{\partial}{\partial \epsilon}(gh, 0), \frac{\partial}{\partial \epsilon}(h, 0)\right)- \frac{\partial}{\partial \epsilon}(g, 0)\\
 & = \frac{d}{d\epsilon}|_{\epsilon= 0} \bar{\tilde{m}}((gh, \epsilon), (h, \epsilon))- \frac{\partial}{\partial \epsilon}(g, 0)\\
 & = \frac{d}{d\epsilon}|_{\epsilon= 0} (\bar{m}_{\epsilon}(gh, h), \epsilon)- \frac{\partial}{\partial \epsilon}(g, 0)\\
 & = \frac{d}{d\epsilon}|_{\epsilon= 0} \bar{m}_{\epsilon}(gh, h)= \xi_{0}(g, h)
\end{align*}
\end{proof}

\begin{rmk} The interpretation given in the proposition makes substantial use of the assumption that the deformation is $s$-constant, not only in the proof, but already right from the start when we interpreted $\frac{\partial}{\partial \epsilon}$ as an element of $C^{1}_{\textrm{def}}(\tilde{\G})$. Indeed, when the source is not constant, this vector field on $\tilde{\G}$ is not even $\tilde{s}$-projectable. This is very much related to the choice of the family of vector fields $X^{\epsilon}$ in the previous subsection; actually, the equation (\ref{strange-eq}) precisely means that the resulting vector field on $\tilde{\G}$:
\[ \tilde{X}(g, \epsilon)= X^{\epsilon}(g)+ \frac{\partial}{\partial \epsilon}(g, 0)\]
is $s$-projectable (... to $\frac{\partial}{\partial \epsilon}\in \mathfrak{X}(\tilde M)$).
\end{rmk}

Next, we reinterpret Proposition \ref{rigidity-lemma} in terms of the groupoid $\tilde{\G}$. But first we need another general remark on time-dependent vector fields  $\tilde{X}= \{X^{\epsilon}: \epsilon\in I\}$ on a manifold $M$; such an $\tilde{X}$ can be identified with the vector field on $M\times I$:
\[ \tilde{X}(x, \epsilon)= X^{\epsilon}(x) + \frac{\partial}{\partial \epsilon} ;\]
then the flow $\phi_{\tilde{X}}^{t, s}$ mentioned above is related to the standard flow of $\tilde{X}$ by 
\[ \phi_{\tilde{X}}^{\epsilon}(x, s)=  (\phi_{\tilde{X}}^{s+ \epsilon, s}(x), s+ \epsilon) \]
so that  $\phi_{\tilde{X}}^{s+ \epsilon, s}$ is seen as moving $M\times \{s\}$ to $M\times \{s+\epsilon \}$. We will apply this to the time dependent vector fields that arise in Lemma \ref{rigidity-lemma}.

\begin{prop}\label{reinterpret-cocycle-eqs} Consider an $s$-constant deformation as in Proposition \ref{rigidity-lemma} and the associated groupoid $\tilde{\G}$. Then a smooth family $X^{\epsilon}$ of vector fields on $\G$
satisfies the cocycle equations (\ref{cocy-eq}) if and only the induced vector field on $\tilde{\G}$
\[ \tilde{X}(g, \epsilon)= X^{\epsilon}(g)+ \frac{\partial}{\partial \epsilon}\in \mathfrak{X}(\tilde{\G})\]
is multiplicative.
\end{prop}

\begin{proof}
Start from the multiplicativity equation for $\tilde{X}$:
\[ \tilde{X}(g,\epsilon)= (d\bar{\tilde{m}})_{(g, \epsilon), (h, \epsilon)} (\tilde{X}(m_{\epsilon}(g, h), \epsilon), \tilde{X}(h, \epsilon)),\]
where $\bar{\tilde{m}}$ is the division map associated to $\tilde{\G}$. Writing $\tilde{X}$ in terms of $X^{\epsilon}$, we see that the expression on the right-hand side is the sum of two expressions:
\[ (d\bar{\tilde{m}})_{(g, \epsilon), (h, \epsilon)}(X^{\epsilon}(m_{\epsilon}(g, h)), X^{\epsilon}(h)) \ \textrm{and} \ (d\bar{\tilde{m}})_{(g, \epsilon), (h, \epsilon)}\left(\frac{\partial}{\partial \epsilon}, \frac{\partial}{\partial \epsilon}\right).\]
Since $X^{\epsilon}$ is tangent to the fiber $\G_{\epsilon}$ where $\bar{\tilde{m}}$ restricts to $\bar{m}_{\epsilon}$, the first expression is just
\[ (d\bar{m}_{\epsilon})_{(g, h)} (X^{\epsilon}(m_{\epsilon}(g, h)), X^{\epsilon}(h)).\]
For the second expression we use again (\ref{fr-fr-eps}) (and the obvious analogue at $\epsilon\neq 0$) and we find
\begin{align*}
(d\bar{\tilde{m}})_{(g, \epsilon), (h, \epsilon)}\left(\frac{\partial}{\partial \epsilon}, \frac{\partial}{\partial \epsilon}\right) & = \frac{d}{ds}|_{s= 0}  \bar{\tilde{m}}((m_{\epsilon}(g, h), \epsilon+ s), (h, \epsilon+ s))\\
 & = \frac{d}{ds}|_{s= 0}  (\bar{m}_{\epsilon+ s} (m_{\epsilon}(g, h), h), \epsilon+ s)\\
 & = \xi_{\epsilon}(g, h)+ \frac{\partial}{\partial \epsilon} .
\end{align*}
Putting everything together, the multiplicativity equation for $\tilde{X}$ becomes:
\[ X^{\epsilon}(g)= (d\bar{m}_{\epsilon})_{(g, h)} (X^{\epsilon}(m_{\epsilon}(g, h)), X^{\epsilon}(h))+ \xi_{\epsilon}(g, h)\]
i.e. the equations (\ref{cocy-eq}).
\end{proof}

\subsection{General deformations} 
The previous subsection indicates how to proceed in the case of general deformations. Let us first reformulate Definition \ref{dfn-general-def} in terms of families of groupoids (Definition \ref{dfn-families-gpds}).

\begin{dfn}
A deformation of a Lie groupoid $\G$ is a family of Lie groupoids 
\[  \tilde{\G}\tto \tilde{M} \stackrel{\pi}{\rmap} I\]
parametrized by an open interval $I\subset \mathbb{R}$ containing $0$ such that $\G_{0}= \G$ (as Lie groupoids). Two such deformations $\tilde{\G}$ and $\tilde{\G}'$ are said to be equivalent if, after eventually restricting them to smaller open intervals around the origin, they are isomorphic by an isomorphism that is the identity on $\G_{0}$. 

A deformation is said to be proper if $\tilde{\G}$ is a proper groupoid. 
\end{dfn}

We see that strict deformations of $\G$ correspond to deformations $\tilde{\G}$ with the property that, as manifolds, $\tilde{\G}= \G\times I$, $\tilde{M}= M\times I$ and $\pi$ is the projection in the second factor. To construct the deformation class of a deformation, we use the reinterpretation given in Proposition \ref{re-write-s-def}. First we  
need an analogue of $\frac{\partial}{\partial \epsilon}$, which makes sense as an element of $C^{1}_{\textrm{def}}(\tilde{\G})$.

\begin{dfn}\label{dfn-transversal} Let $\tilde{\G}$ be a deformation of $\G$. A transverse vector field for $\tilde{\G}$ is 
any vector field $\tilde{Y}\in \mathfrak{X}(\tilde{\G})$ which is $s$-projectable to a vector field $V\in \mathfrak{X}(\tilde{M})$ which, in turn, is $\pi$-projectable to $\frac{\partial}{\partial\epsilon}$. 
\end{dfn}

With this, we can generalize the construction from Proposition \ref{re-write-s-def} as follows. 

\begin{prop} \label{prop: deformation class} Let $\tilde{\G}$ be a deformation of $\G$. Then:
\begin{enumerate}
\item[(i)] there exist transverse vector fields for $\tilde{\G}$. 
\item[(ii)] for any $\tilde{Y}\in \mathfrak{X}(\G)$ transverse, $-\delta(\tilde{Y})\in C^{2}_\textrm{def}(\tilde{\G})$, when restricted to 
$\G_0$, induces a cocycle
\[ \xi_0\in C^{2}_{\mathrm{def}}(\G_0).\]
\item[(iii)] the cohomology class of $\xi_0$ does not depend on the choice of $\tilde{Y}$.
\end{enumerate}
\end{prop}

\begin{proof}
To produce transverse vector fields one chooses any $\tilde{V}\in \mathfrak{X}({\tilde{M}})$ that is $\pi$-projectable to $\frac{\partial}{\partial \epsilon}$ and any $\tilde{Y}\in \mathfrak{X}(\mathcal{\tilde{\G}})$ that is $\tilde{s}$-projectable to $\tilde{V}$. For (ii), given $\tilde{Y}$, we have to show that for $(g, h)\in \G_{0}$, $\delta(\tilde{Y})(g, h)\in T_g\tilde{\G}$ is tangent to $\G_{0}$ i.e. that it is killed by the differential of $\pi\circ \tilde{s}$; but
\[ d\pi  (d\tilde{s} (\delta(\tilde{Y})(g, h)))= d\pi(d\tilde{t}(\tilde{Y}(h))- d\tilde{s}(\tilde{Y}(g)))\]
and using that $d\pi\circ d\tilde{t}= d\pi \circ d\tilde{s}$ and that $\tilde{X}$ is $\tilde{s}$-projectable, the desired vanishing follows. Finally, assume that $\tilde{Y}'$ is another transverse vector field. Then $Z:= \tilde{Y}'- \tilde{Y}$ is killed by $d\pi \circ d\tilde{s}$, hence the values $Z(g)$ are already tangent to the fiber groupoids; in particular, on $\G= \G_0$, one has $Z_{0}\in C^{1}_{\textrm{def}}(\G_0)$ and the cocycles associated to $\tilde{Y}'$ and $\tilde{Y}$ are related by $\xi_{0}^{'}- \xi_{0}= \delta(Z_0)$.
\end{proof}

\begin{dfn} The resulting cohomology class $[\xi_0]\in H^{2}_{\textrm{def}}(\G)$ is called the \textit{deformation class} associated to the deformation $\tilde{\G}$ of $\G$. 
\end{dfn}

\begin{rmk}\label{rmk-to-proceed-general}
We see that, in terms of the groupoids $\tilde{\G}$, the natural approach to rigidity from the previous subsections takes a more algebraic but simpler form: one searches for vector fields $\tilde{X}$ on $\tilde{\G}$ which are both transverse as well as multiplicative, then one uses their flows $\phi_{\tilde{X}}^{\epsilon}$ to obtain isomorphisms between $\G_0$ and $\G_{\epsilon}$. Of course, there is the usual issue regarding the domain of definition of the flows; and an even more serious issue is the existence of multiplicative transverse vector fields $\tilde{X}$. By the long exact sequence of Proposition \ref{the-LOS}, the last issue is equivalent to the existence of $[V]\in \Gamma(\nu)^{\textrm{inv}}$
which is in the kernel of the map $K$ there; we see that the vanishing of cohomology in degree $2$ greatly reduces this issue. 
\end{rmk}

\subsection{Families of groupoids and the variation map}\label{subsec-var_map}

Consider a family of Lie groupoids parametrized by a manifold $B$,
\[  \G\tto M \stackrel{\pi}{\rmap} B\]
(see the introduction). Using the induced groupoids $\G_b$ with $b\in B$, any curve $\gamma : I \rmap B$, induces a deformation $\gamma^{\ast}\G$ of $\G_{\gamma(0)}$. 

\begin{prop}
Let $\G \tto M \rmap B$ be a family of Lie groupoids, $b\in B$. Then, for any curve $\gamma: I\rmap B$ with $\gamma(0)=b$, the deformation class of $\gamma^{\ast}\G$ at time $0$ only depends on $\dot{\gamma}(0)$, and this defines a linear map
\[ \textrm{Var}_b: T_bB \longrightarrow H^2_{\mathrm{def}}(\G_b)\]
\end{prop}

\begin{proof}
Let $v = \dot{\gamma}(0)$ and choose an arbitrary extension $\tilde{v} \in \mathfrak{X}(B)$ of $v$, and an  $s$-projectable $\tilde{Y} \in \mathfrak{X}(\G)$ such that $d(\pi\circ s)(\tilde{Y})=\tilde{v}$. Define $[\tilde{\xi}_b] \in H^2_{\mathrm{def}}(\G_b)$ as the cohomology class of
\begin{equation}\label{eq: alternative var}
\tilde{\xi_b} = -(\delta\tilde{Y})\big{\vert}_{\G_b^{(2)}}.
\end{equation}
We will show that:
\begin{itemize}
\item The cohomology class $[\tilde{\xi}_b] \in H^2_{\mathrm{def}}(\G_b)$ does not depend on the choices of $\tilde{v}$ and $\tilde{Y}$;
\item Under the canonical identification of $(\gamma^\ast\G)_0$ with $\G_b$, the deformation class $[\xi_0]$ coincides with $[\tilde{\xi}_b]$. 
\end{itemize}

For the first statement, the proof that the class does not depend on the lift $\tilde{X}$ is identical to the proof of proposition \ref{prop: deformation class} (iii). To show that $[\tilde{\xi}_b]$ does not depend on the extension $\tilde{v}$, let $\tilde{v}'$ be another extension of $v$ and choose a lift $\tilde{Y}'$ of $\tilde{v}'$ such that $\tilde{Y}(g) = \tilde{Y}'(g)$ for all $g \in \G_b$. Then, for this choice of lift it follows from the linearity of $\delta$ that
\[\delta\tilde{Y}(g,h) = \delta\tilde{Y}'(g,h) \text{ for all } g,h\in \G_b^{(2)}.\]

Next we prove the second statement. Since the statement is local (in $\epsilon$), we can assume without loss of generality the $\gamma: I \rmap B$ is an embedding. We choose $\tilde{v}$ to coincide with $\dot{\gamma}$ on all points of the curve, and we take a lift $\tilde{Y}$ of $\tilde{v}$ as above. Then the vector field
\[Y_{(\epsilon, g)} = \frac{\partial}{\partial \epsilon} + \tilde{Y}_g \in T_{(\epsilon, g)}\gamma^\ast\G \]
is transverse, and thus $\xi_0$ is the restriction to $(\gamma^\ast\G)_0^{(2)}$ of $-\delta Y$. 

However, since the multiplication and inversion on $\gamma^\ast\G$ are given by
\[m_{\gamma}(\epsilon, g),(\epsilon, h) = (\epsilon, gh), \quad (\epsilon, g)^{-1} = (\epsilon, g^{-1})\]
we obtain that 
\[d\bar{m}_{\gamma}(Y_{(0,gh)}, Y_{(0,h)}) = \frac{\partial}{\partial \epsilon} + d\bar{m}(\tilde{Y}_{gh},\tilde{Y}_h),\]
from where it follows that (for these choices of lift, and transverse vector field) $\tilde{\xi}_b = \xi_0$.
\end{proof}

\begin{rmk}
Equation \ref{eq: alternative var} in the proof above gives an alternative description of the variation map.
\end{rmk}

\section{The proper case}

\begin{thm}\label{vanish} Let $\G$ be a proper Lie groupoid. Then 
\[ H_{\mathrm{def}}^0(\G)\cong \Gamma(\mathfrak{i})^{\mathrm{inv}}, \ \ H_{\mathrm{def}}^1(\G)\cong \Gamma(\nu)^{\mathrm{inv}}, \ \ \textrm{and} \ \
H_{\mathrm{def}}^k(\G)= 0 \ \textrm{for\ all}\ k\geq 2.\]
\end{thm}
\begin{proof}
We proceed similarly to the proof of the vanishing of differentiable cohomology from \cite{VanEst}. As there we appeal to the fact that for any proper Lie groupoid $\G$ over $M$ one can use a Haar system and a cut-off function to construct a family of linear maps 
\[ \int_{t^{-1}(x)}: C^\infty(t^{-1}(x))\rmap \mathbb{R},\ \ f\mapsto \int_{t^{-1}(x)} f\stackrel{\text{notation}}{=}\int_x f(g)dg\] which depends smoothly on $x\in M$, (i.e., applying them to a smooth function on $\G$ produces a smooth function on $M$), is normalized (it sends the constant function 1 on $\G$ to 1 on $M$) and left-invariant (the integral does not change under composition with $L_h$, for any $h\in\G$).

For $k\geq 2$, let $c\in C_\text{def}^k(\G)$ be a cocycle and define a map $X:\G^{k-1}\longrightarrow T\G$ by 
\begin{equation}\label{transgr-form} 
 X(g_1,\ldots,g_{k-1})= (-1)^k\int_{s(g_{k-1})}c(g_1,\ldots,g_{k-1},h)dh.
\end{equation}
The map $X$ is an element of $C_\text{def}^{k-1}(\G)$ and we will now show that $\delta X = c$:

\begin{align*}\delta X(g_1,\ldots,g_k) & = - d\bar{m}(X(g_1g_2,\ldots,g_k),X(g_2,\ldots, g_k)) \\
& +\sum_{i=2}^{k-1}(-1)^{i}X(g_1,\ldots,g_ig_{i+1},\ldots,g_k) \\
& + (-1)^kX(g_1,\ldots,g_{k-1})\\
& = (-1)^{k} \int_{s(g_k)}-d\bar{m}(c(g_1g_2,\ldots,g_k,h),c(g_2,\ldots, g_k,h)) \\
& +\sum_{i=2}^{k-1}(-1)^{i}c(g_1,\ldots,g_ig_{i+1},\ldots,g_k,h) dh \\
& + \int_{s(g_{k-1})}c(g_1,\ldots,g_{k-1},h) dh \\
& \stackrel{\delta c=0}{=} \int_{s(g_k)} - c(g_1,\ldots,g_{k-1},g_kh)+ c(g_1,\ldots, g_k) dh \\
& + \int_{s(g_{k-1})}c(g_1,\ldots,g_{k-1},h) dh \\
& \stackrel{\text{left-inv.}}{=} c(g_1,\ldots,g_k).
\end{align*}

Note that exactly the same formulas applied to cocycles in $C^{*}(\G, \mathfrak{i})$ imply that $H^k(\G, \mathfrak{i})= 0$ for all $k\geq 0$ - and this is basically the proof of the vanishing of differentiable cohomology with coefficients in $\mathfrak{i}$ from \cite{VanEst}, with the extra remark that the integrals still define smooth cochains: indeed, if $u\in C^{k}(\G, \mathfrak{i})$, it is smooth as a map 
\[ (g_1, \ldots, g_k) \mapsto u(g_1, \ldots, g_k)\in A_{t(g_1)}\]
hence so is 
\[ (g_1, \ldots, g_{k-1}) \mapsto \int_{s(g_{k-1})}u(g_1,\ldots,g_{k-1},h)dh \in A_{t(g_1)} \]
($t(g_1)$ stays constant in the integration process). The vanishing of $H^{*}(\G, \mathfrak{i})$ in degrees $1$ and $2$ combined with the exact sequence from Proposition \ref{the-LOS}, implies the desired isomorphism 
$H_{\textrm{def}}^1(\G)\cong \Gamma(\nu)^{\textrm{inv}}$. The isomorphism $H_{\textrm{def}}^0(\G)\cong \Gamma(\mathfrak{i})^{\textrm{inv}}$ holds by Proposition \ref{degree-0}.
\end{proof}

\begin{rmk}
Alternatively, the vanishing part of Theorem \ref{vanish} can be immediately obtained by using 
Theorem 3.35 from \cite{camilo}, after the isomorphism of the deformation cohomology and the cohomology with values in the adjoint representation is established (Lemma \ref{lemma-iso-ad}).
\end{rmk}

\subsection{Some consequences (relevant for rigidity results)} 
Here are some consequences of the previous result (and its proof).

\begin{crl}\label{crl-vanishing-proper} If $\tilde{\G}= \{\G_{\epsilon}: \epsilon\in I\}$ is a proper family of groupoids, $k\geq 2$ and 
and $u^{\epsilon}\in C^{k}_{\mathrm{def}}(\G_{\epsilon})$ is a smooth family of deformation cocyles, then the 
equations
\[ \delta(X^{\epsilon})= u^{\epsilon} \]
admit solutions $X^{\epsilon}\in C^{k-1}_{\mathrm{def}}(\G_\epsilon)$ that are smooth with respect to $\epsilon$.
\end{crl}

\begin{proof} This follows from the vanishing part of the theorem, applied to the full groupoid $\tilde{\G}$. 
\end{proof}

\begin{rmk}\label{rmk-imp-proper} It is important to realise that assuming only that each $\G_{\epsilon}$ is proper (which ensures that the equation for each $\epsilon$ has solution) does not suffice. One can actually find families of compact groupoids for which the previous Corollary fails (because, as a family, it is not proper).
\end{rmk}

In order to handle general deformations, we know from the previous section (see e.g. Remark \ref{rmk-to-proceed-general}) that it is important to ensure the existence of multiplicative transverse vector fields.

\begin{lemma}\label{exist-mult-transv1} Any proper family $\G\tto M\stackrel{\pi}{\rmap} I$ of Lie groupoids parametrized by an interval $I$, admits a multiplicative vector field $X\in \mathfrak{X}(\G)$ that is transverse (in the sense of Definition \ref{dfn-transversal}).
\end{lemma}

\begin{proof}
By the isomorphism $H^{1}_{\textrm{def}}(\G)\cong \Gamma(\nu)^{\textrm{inv}}$ it suffices to show that one can find 
$V\in \mathfrak{X}(M)$ which is $\pi$-projectable to $\frac{\partial}{\partial \epsilon}$ and with the property that the induced $[V]\in \Gamma(\nu)$ is invariant. Start with any transverse vector field $X\in \mathfrak{X}(\G)$. We define the vector field $V$ on $M$ by
\[ V_{x}:= \int_{x} dt(X_a) da .\]
We have
\[ d\pi (V_x)= \int_x d\pi(dt(X_a)) da= \int_x d\pi(ds(X_a)) da= \int_x \frac{\partial}{\partial \epsilon} da= \frac{\partial}{\partial \epsilon} .\]
To check that $[V]$ is invariant, it suffices to find a vector field $X^{'}$ on $\G$ which is both $s$ and $t$-projectable to $V$. We claim that 
\begin{equation}\label{change-X-int} 
X_{g}^{'}:= \int_{s(g)} d \bar{m}(X(ga), X(a)) da \in T_g\G 
\end{equation}
does the job. Indeed,
\[ ds(X_{g}^{'})= \int_{s(g)} ds(d \bar{m}(X(ga), X(a))) da= \int_{s(g)} dt(X(a)) da= V_{s(g)}\]
and a similar computation combined with the invariance of the integral shows that also $dt(X_{g}^{'})= V_{t(g)}$. 
\end{proof}

Moreover, in order to handle deformations semi-locally (relative versions), one needs a relative version of the previous lemma. 

\begin{lemma}\label{exist-mult-transv2}
With the same notations as in the previous lemma, if $N\subset M$ is an invariant submanifold so that $\pi|_{N}: N\rmap I$ is still a submersion and $X_{N} \in \mathfrak{X}(\G|_{N})$ is a given multiplicative transverse vector field, then $X$ can be chosen so that it extends $X_{N}$. 
\end{lemma}

\begin{proof}
The proof is just a careful analysis of the previous proofs; here are the details. Start with $X_{N}$ and its base vector field $V_N$. Choose an extension $V$ of $V_{N}$ that is $\pi$-projectable to $\frac{\partial}{\partial \epsilon}$ and choose $X\in \mathfrak{X}(\G)$ which is $s$-projectable to $V$. We modify $X$ in steps. First we make sure that $X$ also extends $X_{N}$: since $X_{N}(g)- X(g)$ is killed by $ds$, it is of type $r_g(\eta(g))$ with $\eta(g)\in A_{t(g)}$ defined for $g\in \G|_{N}$; choosing a smooth extension $\tilde{\eta}\in \Gamma(\G, t^*A)$ of $\eta$ one then replaces $X$ by $g\mapsto X(g)+ r_g(\tilde{\eta}(g))$.

Hence we may assume that $X$ extends $X_{N}$ and is $s$-projectable. One can also arrange it so that it is also $t$-projectable: the replacement is the one given by formula (\ref{change-X-int}) - indeed, when applied to $g\in \G|_{N}$, the integration variable $a$ stays in $\G|_{N}$, hence one deals with 
\[ \int_{t(g)} d \bar{m}(X_N(ga), X_N(a)) da= \int_{t(g)} X_N(g) da= X_{N}(g)   \]
where the multiplicativity of $X_{N}$ was used. 

Hence we may assume that $X$ extends $X_{N}$ and is $s$ and $t$-projectable to some $V$. Then the expressions
\[ d\bar{m}(X(gh), X(h))- X(g) \in T_g\G\]
are killed by $ds$ and $dt$ hence they arise by right translations with respect to $g$ of some elements
\[ \zeta(g, h)\in \mathfrak{i}_{t(g)} .\]
As before one can find $\eta\in C^{1}(\G, \mathfrak{i})$ so that $\zeta= \delta(\eta)$. However, the integral formula for $\eta$ (of type (\ref{transgr-form})) shows that $\eta$ vanishes on elements of $\G|_{N}$ (because the multiplicativity of $X_{N}$ implies that $\zeta$ does). We can now change $X$ to 
\[ X'(g)= X(g)+ r_g(\eta(g)) \]
which has the same properties as $X'$ and is also multiplicative; the multiplicativity is implicit in the last part of the proof of Theorem \ref{vanish} but also follows easily by a direct computation. 
\end{proof}

\begin{rmk}\label{rmk-mult-lift-families} Note that the last two lemmas apply (basically with no changes in the proof) to general proper families $\G\tto M\stackrel{\pi}{\rmap} B$ parametrized by a manifold $B$; the conclusion is that any vector field $W$ on $B$ admits a multiplicative lift $X$ to $\G$ (lift in the sense that $X$ is projectable, via $\pi \circ s= \pi\circ t$ to $W$). 

Moreover, using the explicit argument from the last proof (applied to $N= \emptyset$ and with $I$ replaced by $B$) shows that the choice of $X$ can be made smooth in $W$. Indeed, all the steps involved to produce $X$ are given by explicit formulas that are clearly smooth, except maybe for the starting step which starts with the choice of a vector field $V$ on $M$ that is $\pi$-projectable to $W$ and $X\in \mathfrak{X}(\G)$ that is $s$-projectable to $V$. But that can be done smoothly as well, by fixing smooth splittings of $d\pi$ and $ds$ (i.e. Ehresmann connections on $\pi$ and $s$). 

As analogue of Lemma \ref{exist-mult-transv2} one can start with any sub-family $\G|_{N}\tto N\stackrel{\pi}{\rmap} B_0$ with $B_0\subset B$, $N\subset M$ submanifolds and $N$ invariant and then, if 
$W\in \mathfrak{X}(B)$ is tangent to $B_0$ and one is given a multiplicative lift $X_{N}$ of $W|_{B_0}$ to $\G|_{N}$, then one can find a multiplicative lift $X\in \mathfrak{X}(\G)$ that extends $X_{N}$.
\end{rmk}

\section{Applications to rigidity}

As explained in the previous section, the main idea to obtain rigidity results is to use the vanishing of the deformation cohomology and the flows of the resulting vector fields. Here are some immediate illustrations of this idea. We start with the cases of $(s, t)$-constant and $s$-constant deformations.

\begin{thm} \label{thm: compact rigid} The following hold true:
\begin{itemize}
\item Any $(s, t)$-constant deformation of a proper Lie groupoid is trivial.
\item Any $s$-constant deformation of a compact Lie groupoid is trivial.
\end{itemize}
\end{thm}

\begin{proof}
It is clear that any deformation $\tilde{\G}= \{\G_{\epsilon}\}$ as in the statement is proper. Then we follow the general plan:
\begin{itemize}
\item we appeal to Corollary \ref{crl-vanishing-proper} to obtain $\tilde{X}= \{X^{\epsilon}\}$ smoothly depending on $\epsilon$ transgressing the deformation cocycles $\xi_{\epsilon}$.
\item we use the first part of Lemma \ref{rigidity-lemma} and the resulting flows $\phi_{\tilde{X}}^{t, s}$ as candidates for isomorphisms between $\G_{s}$ and $\G_{t}$. 
\item to make sure that $\phi_{\tilde{X}}^{u, v}$ is defined on the entire $\G_{v}$ for $u$ and $v$ small enough, we use the second part of Lemma \ref{rigidity-lemma}; we are left with proving 
that $\phi_{\tilde{V}}^{u, v}$ is defined on the entire $M$ for $u$ and $v$ small enough.
\end{itemize}
The very last part is clear in the second case since $M$ is compact. For $(s, t)$-constant deformations, we have seen that the resulting deformation cocycles $\xi_{\epsilon}$ live in the subcomplex $C^{2}(\G_{\epsilon}, \mathfrak{i})$ 
and we can solve the equations (\ref{cocy-eq}) inside this subcomplex; i.e. we can arrange that  $V^{\epsilon}= 0$ hence, again, there are no problems with the flow. 
\end{proof}

\begin{rmk}\label{rmk-palais}
It is not true that $s$-constant deformations of proper Lie groupoids are trivial. An example of this can be obtained by carefully analysing a construction that Palais gave in \cite{palais2}. In {\it loc. cit} it is shown that for any non-trivial compact Lie group $G$, there exists an uncountable family of inequivalent $G$ actions on an euclidean space $\mathbb{R}^n$, such that any of them have isomorphic linearised actions at a fixed point. More specifically:
\begin{itemize}
\item It is shown in \cite{mcmillan} that there exists an open 3-manifold W which is not diffeomorphic to euclidean 3-space, and such that $W\times \R^n$ is diffeomorphic to $n+3$-euclidean space (this is really the important point!).
\item If one considers the diagonal action on $W\times \R^{n}$  formed from a faithful representation of a compact group $G$ on $\R^n$ and the trivial action of $G$ on $W$, the corresponding (non-linear) $G$-action on $\R^{n+3}$ has a fixed point set diffeomorphic to $W$.
\item  If one takes a fixed point of the action (which without loss of generality can be assumed to be the origin) the homotety deformation to the linear action is a non-trivial deformation: the fixed point set of the linear action is a vector space of dimension 3 (i.e., $\R^3$) and hence cannot be diffeomorphic to the fixed point set of the the non-linear action (which is diffeomorphic to $W$).
\item In fact, in \cite{mcmillan} it is shown that there is an uncountable family of non-diffeomorphic manifolds $W_\alpha$ with the property above. We thus obtain uncountably many inequivalent $s$-constant deformations of the proper action groupoid associated to the linearised action.
\end{itemize}
\end{rmk}

The last part of the theorem \ref{thm: compact rigid} also has relative version; moreover, passing to semi-local statements, one can deal with general proper groupoids. The first one in this direction is the following:

\begin{thm}\label{thm-rigid-technical} Let $\G\tto M$ be a proper groupoid, let $N\subset M$ be an invariant submanifold.

Then for any $s$-constant deformation $\tilde{\G}= \{\G_{\epsilon}: \epsilon\in I\}$ of $\G$ which is constant on $N$ 
($\G_{\epsilon}|_{N}= \G|_N$ for all $\epsilon$) and any compact interval $I_0\subset I$ one can find a smooth family 
of groupoid isomorphisms
\[ F_{\epsilon}: \G|_{U_0}\rmap \G_{\epsilon}|_{U_{\epsilon}}, \ \ (\epsilon\in I_0).\]
which restrict to the identity on $\G|_{N}$. Moreover, if $M$ is compact then one may choose $U_{\epsilon}= M$.
\end{thm}

\begin{proof}
We need to look a bit closer at the previous arguments. First of all, because of the hypothesis, the deformation cocycles $\xi_{\epsilon}$ vanish on all pairs $(g, h)$ coming from $\G|_{N}$. Looking at the formulas (\ref{transgr-form}) defining the transgressing cochains we see that also $X^{\epsilon}$ vanishes at points of $\G|_{N}$. Let us now move to the product $M\times I$ and interpret $\tilde{V}$ as a vector field there. Since $\tilde{V}$ agrees with $\frac{\partial}{\partial \epsilon}$ on $N\times I$, we see that  $\phi^{\epsilon}_{\tilde{V}}(x, v)= \phi^{v+ \epsilon, v}_{\tilde{V}}(x), v+ \epsilon)$ is defined for all $x\in N$ as long as $v+ \epsilon, v\in I$. In particular, $\phi^{\epsilon, 0}_{\tilde{V}}(x)$ is defined for all $x\in N$, $\epsilon\in I$. With $I_0$ as in the statement, since any open in $M\times I$ containing $N\times I_0$ 
also contains $U_0\times I_0$ for some open neighbourhood $U_0$ of $N$ in $M$, we find such an $U_0$ such that 
$\phi^{\epsilon, 0}_{\tilde{V}}(x)$ is defined for all $x\in U_0$ and $\epsilon\in I_0$. Set 
\[ U_{\epsilon}= \phi^{\epsilon, 0}_{\tilde{V}}(U_0).\]
Of course, $\phi^{0, \epsilon}_{\tilde{V}}$ is defined on $U_{\epsilon}$. All together, using again Lemma \ref{lemma-on-flows} we find that 
\[ \phi_{\tilde{X}}^{\epsilon, 0}: \G|_{U_0} \rmap \G_{\epsilon}|_{U_{\epsilon}} \]
is a groupoid isomorphism, with inverse $\phi_{\tilde{X}}^{0, \epsilon}$ defined on the entire $\G_{\epsilon}|_{U_{\epsilon}}$. Since $X^{\epsilon}$ is zero on $\G|_{N}$, this isomorphism is the identity on $\G|_{N}$. 

We are left with proving the last part. But this follows from a general property of flows of vector fields $\tilde{V}$ on $M\times I$ when $M$ is compact:  $\phi^{\epsilon, 0}_{\tilde{V}}(x)$ is defined for all $x\in M$ and $\epsilon\in I$.
\end{proof}

The following shows that proper deformations are trivial locally around compact invariant submanifolds.

\begin{thm}\label{thm-rigid-technical2} Let $\G\tto M$ be a proper groupoid and let $N\subset M$ be a compact invariant submanifold. Then for any proper deformation $\tilde{\G}= \{\G_{\epsilon}\}$ of $\G$, there exists a smooth family $\tilde{U}= \{U_{\epsilon}\}$ of opens $U_{\epsilon}\subset M_{\epsilon}$ such that $N$ is contained in $U_0$ and such that the deformation $\tilde{\G}|_{\tilde{U}}= \{\G_{\epsilon}|_{U_{\epsilon}}\}$ is trivial. 

In particular, any proper deformation of a compact groupoid is trivial.
\end{thm}

\begin{proof} 
We work on the large groupoid $\tilde{\G}\tto \tilde{M}$ to which we apply Lemma \ref{exist-mult-transv1} to obtain a multiplicative transverse vector field $\tilde{X}$ with base field denoted $\tilde{V}$. Since $N$ is compact we can choose a smaller $I_0\subset I$ and an open neighbourhood $U_0$ of $N$ in $M_0$ such that $\phi_{\tilde{V}}^{\epsilon}(y)$ is defined for all $y\in U_0$, $\epsilon\in I_0$. We set 
\[  U_{\epsilon}= \phi_{\tilde{V}}^{\epsilon}(U_0) \ \ (\epsilon\in I_0).\]
The second part of Lemma \ref{lemma-on-flows} insures that $\phi_{\tilde{X}}^{\epsilon}$ is defined on the entire 
$\G|_{U}= \G_{0}|_{U_0}$, while the multiplicativity of $\tilde{X}$ implies (cf. the first part of Lemma \ref{lemma-on-flows}) that
\[ \phi_{\tilde{X}}^{\epsilon}: \G|_{U} \rmap \G_{\epsilon}|_{U_{\epsilon}} \]
is a groupoid isomorphism (note that the inverse $\phi_{\tilde{X}}^{-\epsilon}$ is defined on the entire $\G_{\epsilon}|_{U_{\epsilon}}$ again because of Lemma \ref{lemma-on-flows} and the fact that $\phi_{\tilde{X}}^{-\epsilon}$ is defined on $U_{\epsilon}$). 
\end{proof}

Theorem \ref{thm-rigid-technical} can also be extended to general proper deformations.

\begin{thm}\label{thm-rigid-technical3}
Assume that $\tilde{\G}= \{\G_{\epsilon}: \epsilon\in I\}$ is a proper family of Lie groupoids and 
$\tilde{N}= \{N_{\epsilon}\}$ is a smooth family of submanifolds $N_{\epsilon}\subset M_{\epsilon}$  invariant with respect to $\G_{\epsilon}$ and such that $\G|_{\tilde{N}}= \{\G|_{N_{\epsilon}}\}$ is trivial, with given trivializing diffeomorphisms
\[ \psi_{\epsilon}: \G_{0}|_{N_0}\rmap \G_{\epsilon}|_{N_{\epsilon}}.\]
Then for any $I_0$ which is the interior of a compact interval contained in $I$, there exists a smooth family $\tilde{U}= \{U_{\epsilon}: \epsilon\in I_0\}$ of open neighbourhoods of $N_{\epsilon}$ in $M_{\epsilon}$ and a smooth family 
of groupoid isomorphisms
\[ F_{\epsilon}: \G_{0}|_{U_0}\rmap \G_{\epsilon}|_{U_{\epsilon}}, \ \ (\epsilon\in I_0).\]
extending $\psi_{\epsilon}$.
\end{thm}

\begin{proof} We proceed as in the proof of Theorem \ref{thm-rigid-technical} but working directly on the large groupoid $\tilde{\G}$; in particular, 
we interpret the smooth family $\{\psi_{\epsilon}\}$ as a map
\[ \psi: \G|_{N}\times I\rmap \tilde{\G},\ \ (g, \epsilon)\mapsto \psi(g, \epsilon):= \psi_{\epsilon}(g)\]
and similarly for the base map, denoted $f$. 
For $g\in \tilde{\G}|_{\tilde{N}}$, consider $\epsilon= \pi\circ s(g)$ i.e. with the property that $g\in \G_{\epsilon}|_{N_{\epsilon}}$. Move $g$ back to $\epsilon= 0$, i.e. consider $g_0= \psi_{\epsilon}^{-1}(g)\in \G|_{N}$ and consider 
\[ \tilde{X}_0(g):= \frac{d}{dv}|_{v= \epsilon} \psi(g_0, v) .\]
Since the values of the curve $v\mapsto \psi(g_0, v)$ belong to $\G_{v}|_{N_{v}}\subset \tilde{\G}|_{\tilde{N}}$, $\tilde{X}_0$ is
a vector field on $\tilde{\G}|_{\tilde{N}}$ which is $s$ and $t$-projectable to the similar vector field 
$\tilde{V}_0$ on $\tilde{N}$ (constructed using $f$ instead of $F$). Also, it is clear that $\tilde{V}_{0}$ is $\pi$-projectable to $\frac{\partial}{\partial \epsilon}$. Since each $\psi_{\epsilon}$ is a groupoid homomorphism, we deduce that $\tilde{X}_0$ is a multiplicative transverse vector field on $\tilde{\G}|_{\tilde{N}}$. Use now  
Lemma \ref{exist-mult-transv2} to extend $\tilde{X}_0$ to a similar vector field $\tilde{X}$ on $\tilde{\G}$.  Note also that, for $g_0\in \G|_{N}$, the curve
$\gamma(\epsilon)= \psi(g_0, \epsilon)$ has 
\[ \dot{\gamma}(\epsilon)= \frac{d}{dv}|_{v= \epsilon} \psi(g_0, v)= \tilde{X}_0(\psi_{\epsilon}(g_0))= \tilde{X}(\gamma(\epsilon))\]
hence the resulting flows $\phi_{\tilde{X}}^{\epsilon}$ satisfy $\phi_{\tilde{X}}^{\epsilon}(g_0)= \psi(g_0, \epsilon)$ for all $\epsilon$ and for $g_0\in \G|_{N}$. Then one proceeds like in the previous two proofs, 
using the flow of $\tilde{X}$. 
\end{proof}

Finally, let us present an application to the linearisation problem. However, we would like to emphasize that our aim  here is not so much to prove linearisation results but to show that, for proper groupoids, the linearisation follows from a stronger property: rigidity.

In general, for any Lie groupoid $\G\tto M$ and any $N\subset M$ invariant submanifold, one can make sense of the linear normal form of $\G$ around $N$, denoted $\mathcal{N}_{N}(\G)$. For instance, when $N= \{x\}$ is a fixed point of $\G$ (i.e. all the arrows that start at $x$ also end at $x$), the action of $\G$ on $\nu$ restricts to a linear action of the isotropy group $\G_x$ at $x$ on the tangent space $T_xM$ and the linear normal form of $\G$ around $x$, $\mathcal{N}_{N}(\G)$,  is the resulting action groupoid. For an arbitrary invariant $N\subset M$, similar to the action of $\G$ on $\nu$, one has a canonical action of $\G|_{N}$ on the normal bundle $\mathcal{N}(N)$ of $N$ in $M$ and the linearisation of $\G$ around $N$, denoted $\mathcal{N}_{N}(\G)$, is defined as the resulting action groupoid. A more conceptual description is obtained by considering the normal bundle $\mathcal{N}(\G|_{N})$ of $\G|_{N}$ in $\G$; 
as for $T\G\tto TM$ (and as a quotient of it), the differentials of the structure maps of $\G$ make $\mathcal{N}(\G|_{N})$ into a groupoid over $\mathcal{N}(N)$, canonically isomorphic to $\mathcal{N}_{N}(\G)$. Note that  $N$, identified with the zero section of $\mathcal{N}_N(M)$, is invariant with respect to $\mathcal{N}_{N}(\G)$. The linearisation problem asks whether $\G$ and $\mathcal{N}_{N}(\G)$ are isomorphic when restricted to neighbourhoods of $N$. When this happens one says that $\G$ is linearisable around $N$. The relationship with deformations is provided by the following type of topological remarks:

\begin{lemma}
If $\G$ is an $s$-proper Lie groupoid and $N\subset M$ is an invariant submanifold, then there exists an open neighbourhood $W$ of $N$ in $M$ and a smooth proper family $\{\G_{\epsilon}\}$ of groupoids over $W$ such that:
\begin{itemize}
\item $\G_1= \G|_{W}$.
\item $\G_{0}$ is (isomorphic to) the linear model $\mathcal{N}_{N}(\G)$. 
\item $\G_{\epsilon}|_{N}= \G|_{N}$ as groupoids for all $\epsilon$. 
\end{itemize}
(even more: one can ensure that $\G_{\epsilon}$ is isomorphic to $\G|_{W}$ for all $\epsilon \neq 0$).
\end{lemma}

\begin{proof}(sketch) One proceeds exactly as in the case of fixed points which was explained in full detail in \cite{ivan}. First of all, the topological arguments from Proposition 2.2 of {\it loc.cit} go through using the following remarks: for a smooth map $\pi: P\rmap M$ and $N\subset M$ a submanifold,
\begin{itemize}
\item If $\pi$ is proper then any open $\mathcal{U}\subset P$ which contains $\pi^{-1}(N)$ also contains $\pi^{-1}(V)$ for some neighbourhood $V$ of $N$.
\item If $\pi$ is a submersion, then one can find tubular neighbourhoods of $N$ in $M$ and of $\pi^{-1}(N)$ in $P$ 
which are compatible with $\pi$ in the sense that the restriction of $\pi$ to the tubular neighbourhood corresponds to the map induced by $d\pi$ on the normal bundles. 
\end{itemize}
One then finds an embedding of type
\[ i: \G|_{\mathcal{N}(N)} \hookrightarrow \mathcal{N}_{N}(\G)= (\G|_{N})\times_{N} \mathcal{N}(N)\]
which is the identity on $\G|_{N}$, where $\mathcal{N}(N)$ is identified with a tubular neighbourhood of $N \subset M$. Then work on the left side, making use of the multiplication by scalars in the fibers of the normal bundles to 
set\[ m_{\epsilon}(e, f)= \frac{1}{\epsilon}m(\epsilon e, \epsilon f) \]
(and similarly for the other structure maps), defined on
\[ \G_{\epsilon}:= \{e\in  \mathcal{N}_{N}(\G): \epsilon e\in \G|_{\mathcal{N}(N)} \}.\]
Like in \cite{ivan}, it is not difficult to see that the limit at $\epsilon= 0$ gives rise to the linearised groupoid structure on $\G_{0}= \mathcal{N}_{N}(\G)$ and that the resulting groupoid $\tilde{\G}$ is proper. 
\end{proof}

We see that we are almost in the position of applying Theorem \ref{thm-rigid-technical}; however, for that one would first have to improve the previous topological argument to insure that the inclusion $i$ above is an isomorphism,
so that the resulting deformation is strict (the fact that the source is constant is clear). Instead, we can just apply theorem \ref{thm-rigid-technical3} and we deduce the following linearization theorem (proved first in \cite{matias_rui_metrics}):

\begin{thm} If $\G$ is an s-proper groupoid and $N\subset M$ is invariant, then $\G$ is linearizable around $N$. 
\end{thm}

Finally, our techniques can be applied also to the study of families of Lie groupoids. As before, there are several variations (semi-local or relative versions). Here we present the most restrictive but simplest statement.

\begin{thm}
Any compact family $\G\tto M\stackrel{\pi}{\rmap} B$ of Lie groupoids is locally trivial. 
\end{thm}

\begin{proof}

Consider a point $b_0$ in $B$, and a coordinate chart $(U,\psi)$ around $b_0$ in $B$, with $\psi=(x_1,\ldots,x_n)$, such that $\psi$ maps $U$ to an open disk $D$ around $0$ in $\R^n$, and $\psi(b_0)=0$.
For any $b\in U$, we consider the vector field \[\sum_{i=1}^nx_i(b)\frac{d}{dx_i}\] on $D$, and the corresponding (via $\psi$) vector field $W^b$ on $U$. The flow of $W^b$ at time one maps $b_0$ to $b$ and moreover, from the construction it is clear that the choice of $W^b$ is smooth in $b$.
As explained in remark \ref{rmk-mult-lift-families}, the vector fields $W^b$  on $U$ admit multiplicative lifts $X^b$ to $\G_{|U}$ and moreover the choice of $X^b$ can be made smooth in $W^b$. By taking the flow of $X^b$ at time $1$ (which is defined for all $g\in \G_{b_0}$ because of compactness), we find a family of isomorphisms
\[ \phi^1_{X^b}: \G_{b_0} \rmap \G_{b} ,\] parametrized smoothly by $b\in U$, giving the desired local trivialisation.
\end{proof}

\begin{rmk}
Using the Reeb stability theorem, it is enough that $\G_{b_0}$ be compact to ensure that there is a neighbourhood $U$ of $b_0$ in $B$ such that $\G_b$ is diffeomorphic to $\G_{b_0}$ for all $b \in U$. It then follows that the same proof above ensures that $\G$ is trivial in a (possibly smaller) neighbourhood of $b_0$.
\end{rmk}

\begin{rmk}
Del Hoyo and Fernandes have recently obtained (independently) the theorem above \cite{matias_rui_fibrations}. Their proof follows from a version of the Ehresmann stability theorem for Lie groupoids.
\end{rmk}

\section{The regular case}

When $\G$ is a regular Lie groupoid, i.e., having all leaves of the same dimension, it has natural representations on the bundle of isotropy Lie algebras of $\G$, denoted $\mathfrak{i}$ and on the normal bundle to the orbits, denoted $\nu$. An arrow $g\in \G$ acts on $\alpha\in \g_{s(g)}$ by conjugation, $g\cdot\alpha=r_{g^{-1}}\circ l_{g}\alpha$ and it acts on $[v]\in \nu_{s(g)}$ by the isotropy representation: if $g(\epsilon)$ is a curve on $\G$ with $g(0)=g$ such that $[v]=\left[\dezero s(g(\epsilon)) \right]$, then $g\cdot[v]=\left[\dezero t(g(\epsilon)) \right]$.

\begin{prop}\label{regular} The deformation cohomology of a regular Lie groupoid $\G$ fits into a long exact sequence 
\[ \cdots\rmap H^{k}(\G,\mathfrak{i})\stackrel{r}{\rmap} H^k_{\operatorname{def}}(\G)\stackrel{\pi}{\rmap} H^{k-1}(\G,\nu)\stackrel{K }{\rmap} H^{k+1}(\G,\mathfrak{i})\rmap\cdots \]
where $r$ is induced by the canonical inclusion $(\ref{inj})$, $\pi$ associates to a deformation cocycle $c$ the
class modulo $\mathrm{Im}(\rho)$ of the $s$-projection $s_{c}$ of $c$, and $K$ will be discussed below.  
\end{prop}

\begin{proof} 
We will construct two cochain complexes $\mathcal{C}$ and $\mathcal{A}$  which fit into two short exact sequences
\begin{equation*}
0\rmap C^*(\G,\mathfrak{i})\stackrel{r}{\rmap} C^*_{\operatorname{def}}(\G)\stackrel{R}{\rmap} \mathcal{C}^* \rmap 0,
\end{equation*}
\begin{equation}\label{pr-seq-2} 
0\rmap \mathcal{C}^*\rmap \mathcal{A}^* \stackrel{S}{\rmap} C^*(\G,\nu)\rmap 0
\end{equation}
and with the property that $\mathcal{A}^*$ is acyclic. For $\mathcal{A}^*$, define
\[ \mathcal{A}^k= C^{k}(\G, TM)\oplus C^{k-1}(\G, TM).\]
(see also Remark \ref{for-later-use}), with the differential given by 
\[ \delta (\phi, \psi)= (-\delta'(\phi), -\phi+ \delta'(\psi)),\]
where
\[  \delta': C^{k}(\G, TM)\rmap C^{k+1}(\G, TM),\]
\[ (\delta' \phi)(g_1,\ldots,g_{k+1})  = \sum_{i=1}^{k}(-1)^{i+1} \phi(g_1,\ldots, g_ig_{i+1},\ldots, g_{k+1}) +(-1)^{k+1} \phi(g_1,\ldots, g_{k}).\]
One can check directly that $\mathcal{A}^*$ is acyclic. More conceptually: $\mathcal{A}^{*}$ is the cylinder of the complex $(C^{*}(\G, TM), \delta')$, and $\delta'$ is clearly acyclic (with homotopy $\phi(g_1, \ldots, g_k)\mapsto \phi(1, g_1, \ldots, g_k)$). 

The map $S: \mathcal{A}^*\rmap C^*(\G;\nu)$ is defined by 
\[ S(\phi, \sigma)(g_1, \ldots, g_k)= g_1\cdot [\sigma(g_2, \ldots, g_k)] -[\phi(g_1, \ldots, g_k)]\]
where $[V]\in \nu$ denotes the class of the tangent vector $V\in TM$ and ``$g_1\cdot$'' refers to the canonical action of $\G$ on $\nu$. It is straightforward to check that $S$ is compatible with the differentials. 
Next, $\mathcal{C}^*$ is defined as the kernel of $S$ and $R$ associates to a cochain $c\in C^k_{\operatorname{def}}(\G)$ the pair
\[ R(c)= (\phi_c, \psi_c)\in C^{k}(\G, TM)\oplus C^{k-1}(\G, TM)\]
characterized by:
\[ \phi_c(g_1, \ldots, g_k)= dt(c(g_1, \ldots, g_k)), \ \ \psi_c(g_2, \ldots, g_k)= ds(c(g_1, \ldots, g_k)).\]
Lemma \ref{invariance-nu} implies that $R$ takes values in $\mathcal{C}^*$. Again, a straightforward computation shows that $R$ is compatible with the differentials. 
It is also clear that $\textrm{Ker}(R)= \textrm{Im}(r)$. Finally, we show that $R$ is surjective; let $(\phi, \psi)\in \mathcal{C}^k$. One can first find $c\in \mathcal{C}^{k}_{\textrm{def}}(\G)$ so that $\psi= \psi_c$ (e.g. use a splitting of $ds: T\G\rmap TM$
to lift the expressions $\psi(g_2, \ldots , g_k)$ to $T_{g_1}\G$). Using again Lemma \ref{invariance-nu} we see that
\[ g_1\cdot [\psi(g_2, \ldots, g_k)]= [dt(c(g_1, \ldots, g_k)] \]
hence the condition that $(\phi, \psi)\in \mathcal{C}^k$ implies that $\phi- dt\circ c$ takes values in the image of the anchor map $\rho$, hence (also using regularity) we can write
$\phi= dt\circ c+ \rho\circ \xi$ from some $\xi\in C^k(\G, A)$. Then $c'\in  C^k_{\operatorname{def}}(\G)$ defined by 
\[ c'(g_1, \ldots, g_k)= c(g_1, \ldots, g_k)+ r_{g_1}(\xi(g_2, \ldots, g_k))\]
is sent by $R$ to $(\phi, \psi)$.

Back to the original sequences, consider the long exact sequences in cohomology that they induce. The one induced by (\ref{pr-seq-2}) gives rise to isomorphisms 
\[ \partial: H^{k-1}(\G, \nu) \rmap H^{k}(\mathcal{C}) ,\]
which combined with the one induced by  (\ref{pr-seq-2}) gives rise to a long exact sequence as in the statement. 

We still have to show that, modulo the isomorphism $\partial$, 
$R$ is identified with $\pi$ i.e., in cohomology, $\partial \circ \pi= R$. For that, let $c\in C^{k}_{\textrm{def}}(\G)$ be a cocycle representing the cohomology class $[c]$. By definition of the connecting map $\partial$, the cohomology class $\partial\circ \pi [c]$ can be represented by  $\delta( a)$ (seen as a cochain in the subspace $\mathcal{C}^{k}$), where $a\in\mathcal{A}^{k-1}$ is any cochain such that $S(a)=\pi(c)$. From the definitions of $S$ and of $\delta$ it is easy to see that $S(-s_c,0)=\pi (c)$, and so $\partial\circ \pi [c]$ can be represented by the cochain $ \eta = (\delta' s_c, s_c)$.
The second component of $R(c)$ is also $s_c$, hence the second component of $R(c)-\eta$ is then equal to zero. On the other hand, since $c$ is closed, $R(c)-\eta$ is closed as well. Finally, it is an easy computation to check that for $(\phi,0)\in\mathcal{C}^k$, the second component of $\delta(\phi,0)$ is equal to $-\phi$. With this we conclude that $R(c)-\eta$ is in fact equal to zero, so $R$ and $\partial \circ \pi$ agree in cohomology.
\end{proof}

Although the previous proof may seem rather ad-hoc at first sight, it becomes very natural if one follows the intuition given by the adjoint representation interpretation of deformation cohomology (see Remark \ref{Yoneda-rmk} below). Furthermore, it reveals several new aspects (see the remarks below), including a possible route to the discovery of the full structure of the adjoint representation (next section).

\begin{rmk}\label{def-morphisms} ({\it A deformation complex for morphisms}) 
First of all, the complex $\mathcal{A}^*$ is similar to the deformation complex and this indicates the presence of a more general construction: one has a deformation complex $C^{*}_{\textrm{def}}(F)$ associated to any morphism of Lie groupoids
\[ F: \G\rmap \H\]
covering some base map $f: M\rmap N$ (cf. \cite{nr} for Lie groups). With the notation from Remark \ref{for-later-use},
\[ C^{k}_{\textrm{def}}(F)\subset C^{k}(\G, F^*T\H)\]
so that $k$-cochains $c$ are smooth maps
\[ \G^{(k)}\ni (g_1, \ldots, g_k)\mapsto c(g_1, \ldots, g_k)\in T_{F(g_1)} \H.\]
The condition that $c$ belongs to $C^{k}_{\textrm{def}}(F)$ is that $ds(c(g_1, \ldots, g_k))$ does not depend on $g_1$. Moreover, the differential $\delta$ of $C^{*}_{\textrm{def}}(F)$ is given by exactly the same formula as for 
the deformation complex, but using the division $\bar{m}$ of $\H$. Denote the resulting cohomology by 
$H^{*}_{\textrm{def}}(F)$. Of course, 
\[ H^{*}_{\textrm{def}}(\G)= H^{*}_{\textrm{def}}(\textrm{Id}_{\G}).\]

More generally, the relation between the deformation complexes of $F$, $\G$ and $\H$ is given by the maps $$C^{k}_{\textrm{def}}(\G) \stackrel{F_*}{\rmap} C^{k}_{\textrm{def}}(F) \stackrel{F^*}{\lmap} C^{k}_{\textrm{def}}(\H),$$ defined by $F_*(c)(g_1,\ldots g_k)=dF\circ c(g_1,\ldots,g_k)$ and $F^*(c')(g_1,\ldots,g_k)=c'(F(g_1),\ldots,F(g_k))$, for any $c\in C^{k}_{\textrm{def}}(\G)$ and $c'\in C^{k}_{\textrm{def}}(\H)$.

Similar to the deformation cohomology of $\G$, $H^{*}_{\textrm{def}}(F)$ can be seen as ``the differentiable cohomology of $\G$ with coefficients in $F^{*}Ad_{\H}$ - the pull-back by $F$ of the adjoint representation of $\H$''.

With this, the complex $\mathcal{A}^*$ from the previous proof is $C^{*}_{\textrm{def}}(F)$ where $F= (s, t): \G\rmap \Pi:= M\times M$ is the canonical map into the pair groupoid of $M$; hence it is related to $F^*Ad_{\Pi}$; the acyclicity of $\mathcal{A}$ is related to the fact that $Ad_{\Pi}$, i.e., the complex $TM\stackrel{\textrm{Id}}{\rmap} TM$, is acyclic. 
\end{rmk}

\begin{rmk}\label{rmk-curvature-map}({\it The curvature map }) One can also go on and compute the ``curvature map $K$'' in the sequence. In degree $0$ one finds the curvature discussed in Subsection \ref{The first manifestation of curvature}. 
A careful analysis in higher degrees shows that $K$ is the cup-product with a canonical cohomology class, still denoted by $K$,
\[ K\in H^{2}(\G, \textrm{Hom}(\nu, \mathfrak{i})).\]
Here we use the induced $\textrm{Hom}$-representation: for $g: x\rmap y$, its action of $\xi: \nu_x\rmap \mathfrak{i}_x$ is $g\cdot \xi: \nu_y\rmap \mathfrak{i}_y$ given by
\[ (g\cdot \xi)(v)= g \cdot \xi(g^{-1}\cdot v). \]
Also, the cup-product operation that we refer to is
\[ C^{k}(\G, \textrm{Hom}(\nu, \mathfrak{i}))\times C^{k'}(\G, \nu)\rmap C^{k+ k'}(\G, \mathfrak{i}),\ (\xi, v)\mapsto \xi\cdot v \]
defined by the same formula as (\ref{eq-cup-pr}) where the pointwise product of $u$ and $v$ is replaced by the evaluation of $u$ on $v$. We will explain in the next section how the point of view of representations up to homotopy can be used to describe the class $K$ (see remark \ref{rmk-K}). However, it is worth having in mind that one can proceed directly and analyse $K$ as it arises from the previous proposition; the analysis is not completely straightforward (e.g. one has to realise the relevance of connections on groupoids) but, ultimately, it reveals the full structure on $Ad= A\oplus TM$ (and the notion of representation up to homotopy). Although we do not give here the full details of such a direct approach, we hope that our comments motivate and clarify the next section. 
\end{rmk}

\begin{rmk}\label{Yoneda-rmk}({\it Yoneda extensions}) The heart of the previous proof is the exact sequence
\[ 0\rmap C^*(\G,\mathfrak{i})\stackrel{r}{\rmap} C^*_{\operatorname{def}}(\G)\stackrel{R}{\rmap}  \mathcal{A}^* \stackrel{S}{\rmap} C^*(\G,\nu)\rmap 0 .\]
With Remark \ref{def-morphisms} in mind, this sequence represents a sequence involving the adjoint and the related representations: $\mathfrak{i}$, $Ad_{\G}$, $Ad_{\Pi}$ and $\nu$; at the level of chain complexes, one simply deals with:
\begin{equation*}
\begin{tikzpicture}[description/.style={fill=white,inner sep=2pt},bij/.style={above,sloped,inner sep=.5pt}]	
	\matrix (m) [matrix of math nodes, row sep=0.3cm, column sep=2.5em, 
		     text height=1.5ex, text depth=0.25ex]
	{ 
		\ &\mathfrak{i} & A & TM & \nu \\
		\ &\underbrace{\ 0\ } & \underbrace{TM} & \underbrace{TM} & \underbrace{\ 0\ } \\
		\ &\mathfrak{i} & Ad_{\mathcal{G}} & F^*Ad_{\Pi} & \nu \\
	};
	\path[->,font=\scriptsize]
		(m-1-3) edge node[auto] {$ \rho $} (m-1-4)
		(m-1-4) edge                       (m-1-5)
		(m-2-2) edge  			            (m-2-3)
		(m-2-3) edge node[auto] {$ Id $}   (m-2-4)
		(m-2-4) edge  (m-2-5) 
		(m-3-2) edge  (m-3-3)
		(m-3-3) edge  (m-3-4)
		(m-3-4) edge  (m-3-5)
		(m-1-2) edge[draw=none] node[sloped, auto=false] {$\oplus$} (m-2-2)
		(m-1-3) edge[draw=none] node[sloped, auto=false] {$\oplus$} (m-2-3)
		(m-1-4) edge[draw=none] node[sloped, auto=false] {$\oplus$} (m-2-4)
		(m-1-5) edge[draw=none] node[sloped, auto=false] {$\oplus$} (m-2-5)
		(m-3-1) edge[draw=none] node[sloped, auto=false] {\large$\mathcal{E}$:} (m-3-2);
\path[right hook->,font=\scriptsize]
		(m-1-2) edge (m-1-3);

\end{tikzpicture}
\end{equation*}

This gives an interpretation of the curvature map from the previous proposition in terms of Yoneda extensions (in the sense of homological algebra), as the cup-product with the element in the $\textrm{Ext}$-group represented by $\mathcal{E}$; moreover, since we basically deal with vector bundles (for which $\textrm{Hom}(E, F)= E^*\otimes F$) the relevant group $\textrm{Ext}^2(\nu, \mathfrak{i})$ is simply $H^{2}(\G, \textrm{Hom}(\nu, \mathfrak{i}))$. Again, all these can be made precise within the framework of representations up to homotopy, providing another way of looking at the cohomology class $K$. 
\end{rmk}

\section{Relation with  the adjoint representation}\label{s-ad}
\label{sec-adjoint}
In this section we describe the relationship between $H^{*}_{\textrm{def}}(\G)$ and the adjoint representation of \cite{camilo}. Actually, we will explain that, once a connection $\sigma$ is fixed, $C^{*}_{\textrm{def}}(\G)$ gives rise right away to a representation up to homotopy
$Ad_{\sigma}$ and then we will identify it with the adjoint representation from \cite{camilo}.

We fix a Lie groupoid $\G\tto M$ and we start by briefly recalling the notion of representation up to homotopy. As in  Remark \ref{for-later-use}, for a vector bundle $E$ over $M$ we consider the space $C^k(\G,E)$ of $E$-valued differentiable cochains. If $E$ 
is graded, then we consider $C(\G,E)^*$ with the total grading 
\[ C(\G,E)^n= \bigoplus_{k+l= n} C^k(G, E^l).\]
As in subsection \ref{sub-Differentiable cohomology}, the cup-product makes $C(\G,E)$ into a right graded $C(\G)$-module.
By definition, a \textit{representation up to homotopy of $\G$} is a graded vector bundle $E$ together with a differential $D$ on $C(\G,E)$ (the structure operator) which makes it into a 
 differential graded $(C(\G), \delta)$-module (with respect to the total grading and the cup-products).
This is precisely what the deformation complex of $\G$ gives us once we fix a splitting of the short exact sequence:
\begin{equation}\label{ehr-seq} 
t^*A\stackrel{r}{\rmap} T\G\stackrel{ds}{\rmap} s^*TM.
\end{equation}
Such a splitting can be seen as a right inverse  $\sigma: s^*TM\rmap T\G$ of $(ds)$, i.e. an Ehresmann connection on the bundle $s: \G\rmap M$; we say that $\sigma$ is an Ehresmann connection on $\G$ if, moreover,  at the units it coincides with the canonical splitting $(du)$.

\begin{lemma}\label{lemma-iso-ad} Consider the graded vector bundle 
\[ Ad= A\oplus TM\] 
with $A$ in degree $0$ and $TM$ in degree $1$. Then any Ehresmann connection $\sigma$ induces isomorphisms
\[ I_{\sigma}: C^{k}_{\mathrm{def}}(\G) \cong C(\G, Ad)^k=  C^k(\G, A)\oplus C^{k-1}(\G, TM),\ \ c\longleftrightarrow (u, v)\]
characterized by 
\begin{equation}\label{splitting-def-cochain}
c(g_1, \ldots, g_k)= r_{g_1}(u(g_1, \ldots, g_k))- \sigma_{g_1}(v(g_2, \ldots, g_k)).
\end{equation}
Moreover, this is an is an isomorphism of right $C(\G)$-modules (see lemma \ref{def-DG-mod} for the module structure on $C_{\mathrm{def}}(\G)$). In particular, for any Ehresmann connection $\sigma$ on $\G$, there is a unique operator $D_{\sigma}$ on $C(\G, Ad)$ which makes $Ad$ into a representation up to homotopy of $\G$ and such that $I_{\sigma}$ is an isomorphism between $(C_{\mathrm{def}}(\G), \delta)$ and $(C(\G, Ad), D_{\sigma})$.
\end{lemma}

The resulting representation up to homotopy will be denoted $Ad_{\sigma}$. To make it more explicit (and identify it with the one of \cite{camilo}), let us first be more explicit about the structure of representations up to homotopy of length $2$, i.e. of type
\[ E= E^0\oplus E^1\]
In this case, the structure operator 
\[ D: C^{k}(\G, E^0)\oplus C^{k-1}(\G, E^1)\rmap C^{k+1}(\G, E^0)\oplus C^{k}(\G, E^1)\]
is necessarily of type
\[ D(u, v)= (\delta_{\lambda}(u)+ K\cdot v, -\delta_{\lambda}(v)+ \partial(u) )\]
where:
\begin{itemize}
\item $\partial: E^0\rmap E^1$ is a vector bundle morphism and we use the same letter for
\[ \partial: C^k(\G, E^0)\rmap C^k(\G, E^1), \ \partial(c)= \partial \circ c.\] 
\item $\lambda$ is  a quasi-action of $\G$ on $E= E^{0}\oplus E^{1}$ acting componentwise and
\[ \delta_{\lambda}: C^{*}(\G, E)\rmap C^{*+1}(\G, E)\]
is the induced operator (cf. Remark \ref{for-later-use}). 
\item the ``curvature term'' $K$ is a smooth section which associates to  a pair $(g, h)$ of composable arrows a linear map $K(g, h): E^{1}_{s(h)}\rmap E^{0}_{t(g)}$
and we use the cup-product operation 
\[ C^{k-1}(\G, E^1)\rmap C^{k+1}(\G, E^0), \ c\mapsto K\cdot c\]
\[ (K\cdot c)(g_1, \ldots, g_{k+1})= K(g_1, g_2)(c(g_3, \ldots, g_{k+1})).\]
\end{itemize}
Note that the condition that $D^2= 0$ breaks into 
\begin{equation}\label{cocycle-eq-D2=0-1} \partial \circ \lambda_{g}= \lambda_{g}\circ \partial \ \ \ (\textrm{on}\ E^0),\end{equation}
\begin{equation}\label{cocycle-eq-D2=0-2}  \lambda_{g}\lambda_{h}- \lambda_{gh} + K(g, h)\circ \partial= 0  \ \ \ (\textrm{on}\ E^0),\end{equation}
\begin{equation}\label{cocycle-eq-D2=0-3} \lambda_{g}\lambda_{h}- \lambda_{gh}+ \partial \circ K(g, h)= 0 \ \ \ (\textrm{on}\ E^1),\end{equation}
\begin{equation}\label{cocycle-eq-D2=0-4} \lambda_{g}K(h, k)- K(gh, k)+ K(g, hk)-  K(g, h)\lambda_k= 0     \ \ \ (\textrm{on}\ E^1).\end{equation}

Hence, lemma \ref{lemma-iso-ad} implies that $Ad$ comes with operators $\partial,\lambda,K$ associated to $\sigma$. It is not difficult to check that $\partial$ is the anchor of $A$. We claim that $\lambda$ and $K$ coincide with the quasi-action and the basic curvature of $\sigma$, introduced in \cite{camilo}:

\begin{itemize}\item the quasi-actions associate to every $g: x\rmap y$ the maps
\[ \lambda_g: T_xM\rmap T_yM, \ \lambda_g(X)= (dt)_g(\sigma_g(X)),\]
\[ \lambda_g: A_x\rmap A_y, \ \lambda_g(\alpha)= - \omega_g(\overleftarrow{\alpha}(g)),\]
where $\overleftarrow{\alpha}$ is given by (\ref{left-inv})) and where $\omega:T\G\rmap t^*A$ is the induced left splitting of the sequence \ref{ehr-seq}, $\omega_g(X)=r_{g^{-1}}(X-\sigma_gds(X))$;

\item the basic curvature, $ K= K^{\textrm{bas}}_{\sigma}$, associates to a pair of composable arrows
\[ y\stackrel{g}{\lmap} \stackrel{h}{\lmap} x\]
and a vector $v\in T_xM$, the element
\[ K(g, h)v\in A_y.\]
characterized by 
\begin{equation}\label{def-eq-K-bas} 
r_{gh}(K(g, h)v)= \sigma_{gh}(v)- (dm)_{g, h}(\sigma_g(\lambda_h(v)), \sigma_h(v)) \in T_{gh}\G.
\end{equation}
\end{itemize}
In other words, we have:

\begin{prop} The structure operator of the  representation up to homotopy $Ad_{\sigma}= A\oplus TM$ arising from Lemma \ref{lemma-iso-ad}
is the one associated to the quasi-actions and the basic curvature of $\sigma$ and the bundle map $\partial= \rho: A\rmap TM$, i.e. 
\[ D_{\sigma}(u, v)=  (\delta_{\lambda}(u)+ K\cdot v, -\delta_{\lambda}(v)+ \rho(u) )\]
\end{prop}

\begin{proof}
Before we do the actual computation, let us first point out three general formulas. First we note that, with respect to the decomposition induced by $\sigma$, the expressions of type $\overleftarrow{\beta}(g)$ correspond to the pair $(-\lambda_g(\beta), \rho(\beta))\in A_{t(g)}\oplus T_{s(g)}M$; indeed, 
the $A$ component is obtained by applying $\omega_g$, hence it is $-\lambda_g(\beta)$, while the $TM$-component is obtained by applying $ds$, hence it is $ds(l_g(di(\beta)))= dt(\beta)= \rho(\beta)$. In conclusion, 
\begin{equation*}
\overleftarrow{\beta}(g)= - r_g (\lambda_g(\beta))+ \sigma_g(\rho(\beta)).
\end{equation*}

Next, we claim that, for all $\alpha, \beta\in A$ and $(g, h)$ composable arrows, one has

\begin{equation} \label{aux-eq-1}
(d\bar{m})_{gh, h}(r_{gh}(\alpha), r_h(\beta))= \overrightarrow{\alpha}(g)+ \overleftarrow{\beta}(g)
= r_g(\alpha)- r_g (\lambda_g(\beta))+ \sigma_g(\rho(\beta))
\end{equation}
Due to the previous discussion, we only have to show the first equality. Since both $(r_{gh}(\alpha), 0_h)$ and $(0_g, r_h(\beta))$ are tangent to the domain $\G^{[2]}$ of $\bar{m}$, to prove this formula it suffices to consider separately the cases when $\beta= 0$ and then when $\alpha= 0$. When $\beta= 0$, the left hand side is the differential of the map $s^{-1}(x)\ni a\mapsto \bar{m}(r_{gh}(a), h)= ag= r_g(a)$ ($x= t(g)$), at the unit at $x$, applied to $\alpha$; hence it gives $r_{g}(\alpha(x))= \overrightarrow{\alpha}(g)$. Similarly, when $\alpha= 0$, we deal with the differential of the map 
$s^{-1}(y)\ni b\mapsto \bar{m}(gh, r_h(b))= g b^{-1}= l_g(i(b))$ ($y= t(h)= s(g)$), at the unit at $y$ applied to $\beta$, i.e. (see (\ref{left-inv})), $\overleftarrow{\beta}(g)$.

Finally we also need the fact that, for $(g, h)$ composable and $V\in T_{s(h)}M$, 
\begin{equation} \label{aux-eq-3}
(d\bar{m})_{gh, h}(\sigma_{gh}(V), \sigma_h(V))= \sigma_g(\lambda_h(V))+ r_{g}(K(g, h)(V))
\end{equation}
Indeed, using the formula for $\sigma_{gh}(V)$ that results from (\ref{def-eq-K-bas}), we find
\[ (d\bar{m})_{gh, h}((dm)_{g, h}(\sigma_g(\lambda_h(V)), \sigma_h(V)), \sigma_h(V)))+ (d\bar{m})_{gh, h}(r_{gh}(K(g, h)(V), 0_h).\]
The first term is just $\sigma_g(\lambda_h(V))$ because $\bar{m}(m(g, h), h)= g$, while the second term is $r_{g}(K(g, h)(V))$.

Consider now a deformation cochain $c$ written as (\ref{splitting-def-cochain}); we will compute $\delta(c)$ in terms of $u$ and $v$. \\

The first term in the formula for $\delta(c)(g_1, \ldots, g_{k+1})$ is 
\begin{align*} - (d\bar{m})(r_{g_1g_2}(u(g_1g_2, g_3, \ldots, g_{k+1}))&- \sigma_{g_1g_2}(v(g_3, \ldots, g_{k+1})), 
 r_{g_2}(u(g_2, \ldots, g_{k+1}))\\ &- \sigma_{g_2}(v(g_3, \ldots, g_{k+1})).\end{align*}
Using (\ref{aux-eq-1}) we find that $- (d\bar{m})(r_{g_1g_2}(u(g_1g_2, g_3, \ldots, g_{k+1})), 
r_{g_2}(u(g_2, \ldots, g_{k+1})))$ is 
\[ - r_{g_1}(u(g_1g_2, g_3, \ldots, g_{k+1}))+ r_{g_1}(\lambda_{g_1}(u(g_2, \ldots, g_{k+1})))- \sigma_{g_1}(\rho(u(g_2, \ldots, g_{k+1}))) \]
and using (\ref{aux-eq-3}) we find that $(d\bar{m})(\sigma_{g_1g_2}(v(g_3, \ldots, g_{k+1})), \sigma_{g_2}(v(g_3, \ldots, g_{k+1})))$ is 
\[ \sigma_{g_1}(\lambda_{g_2}(v(g_3, \ldots, g_{k+1})))+ r_{g_1}(K(g_1, g_2)(v(g_3, \ldots, g_{k+1}))) .\]
Hence the components of the first term in the formula for $\delta(c)$ are 
\[ - u(g_1g_2, g_3, \ldots, g_{k+1})+ \lambda_{g_1}(u(g_2, \ldots, g_{k+1}))+ K(g_1, g_2)(v(g_3, \ldots, g_{k+1}))\in A_{t(g_1)},\]
\[ - \rho(u(g_2, \ldots, g_{k+1}))+ \lambda_{g_2}(v(g_3, \ldots, g_{k+1}))\in T_{s(g_1)}M .\]

The next terms in the formula for $\delta(c)(g_1, \ldots, g_{k+1})$ are, for each $i\in \{2, \ldots, k\}$:
\[ (-1)^{i} (r_{g_1}(u(g_1, \ldots, g_ig_{i+1}, \ldots, g_{k+1}))- \sigma_{g_1}(v(g_2, \ldots, g_ig_{i+1}, \ldots, g_{k+1})))\]
hence give the components
\[ (-1)^i u(g_1, \ldots, g_ig_{i+1}, \ldots, g_{k+1})\in A_{t(g_1)}, (-1)^{i+1}v(g_2, \ldots, g_ig_{i+1}, \ldots, g_{k+1})\in T_{s(g_1)}M .\]
And similarly for the last term. All together, we see that the components corresponding to $\delta(c)$ are
\[ \delta_{\lambda}(u)(g_1, \ldots, g_{k+1})+ K(g_1, g_2)(v(g_3, \ldots, g_{k+1}))\in A,\]  
\[- \rho(u(g_2, \ldots, g_{k+1}))+ \delta_{\lambda}(v)(g_2, \ldots, g_{k+1}) \in TM.\]
Therefore, with respect to the splitting given by (\ref{splitting-def-cochain}) (which introduces a minus sign on the $TM$ component) we find
\[ (\delta_{\lambda}(u)+ K\cdot v, \rho(u)- \delta_{\lambda}(v)).\]
\end{proof}

\begin{rmk}\label{sign convention} When comparing with \cite{camilo}, note that one has to take care with adjusting signs: the formula in \textit{loc. cit.} for the differential $\delta u$ of a differentiable $k$-cochain (formula (\ref{eq-diff-action})) differs from ours by a factor of $(-1)^k$, and the formula for the cup product of a $k$-cochain $u$ and a $k'$-cochain $v$ (formula (\ref{eq-cup-pr})) differs by a factor of $(-1)^{kk'}$. By sending $f\in C^ l(\G)$ and  $(u,v)\in C(\G,Ad)^k$ to $(-1)^{\lfloor\frac{l}{2}\rfloor}f$ and $((-1)^{\lfloor\frac{k}{2}\rfloor}u, (-1)^{\lfloor\frac{k-1}{2}\rfloor}v)$ respectively, one obtains an isomorphism of chain complexes between the two versions, respecting the DG-module structures. Here $\lfloor x\rfloor$ denotes the largest integer that is smaller or equal to $x$.
\end{rmk}

\begin{rmk} It is straightforward to check that the isomorphism $I_{\sigma}$ restricts to an isomorphism between the normalized subcomplexes $(\widehat{C}_{\textrm{def}}(\G), \delta)$ and $(\widehat{C}(\G, Ad), D_{\sigma})$. A cochain $c\in C^k(\G,Ad)$ is said to be normalized if $c(g_1,\cdots,g_k)=0$ whenever any of the arrows $g_i$ is a unit, and $\widehat{C}^k(\G, Ad)$ denotes the subspace of normalized $k$-cochains.
\end{rmk}

\begin{rmk}\label{rmk-K} We now return to case of a regular Lie groupoid $\G$ and to the sequence from Proposition \ref{regular}, in order to give another description of the curvature map $K$. Choose splittings $\mu: \nu \rmap TM$ of the canonical projection $\pi$ and $\tau: \textrm{Im}(\rho)\rmap A$ of the anchor map, and a \textit{compatible} Ehresmann connection $\sigma$ - i.e., with the property that $\mu$ and $\tau$ are equivariant with respect to the canonical actions of $\G$ on $\nu$ and $\mathfrak{i}$, and the quasi-actions $\lambda$ induced by $\sigma$ on $A$ and $TM$.
Such a $\sigma$ can always be obtained, by starting with any Ehresmann connection $\sigma'$ on $\G$, and modifying it via a cochain $\eta\in C^1(\G,\mathrm{Hom}(TM,A))$, as explained in Lemma 6.10 of \cite{mehta}.

Then, composing $K_{\sigma}$ with $\mu$ and $\tau$, we have that any pair $(g, h)$ of composable arrows induces a map $\tau K_{\sigma}(g, h) \mu: \nu_{s(h)}\rmap \mathfrak{i}_{t(g)}$; moving from $\nu_{s(h)}$ to $\nu_{t(g)}$ using the action by $gh$, we end up with a map $K_{\sigma}^{\mu, \tau}(g, h): \nu_{t(g)}\rmap \mathfrak{i}_{t(g)}$, therefore with a differentiable cochain
\[ K_{\sigma}^{\mu, \tau}\in C^{2}(\G, \textrm{Hom}(\nu, \mathfrak{i})).\]

It can be proven that $K_{\sigma}^{\mu, \tau}$ is a cocycle, the resulting cohomology class
\[ [K_{\sigma}^{\mu, \tau}]\in  H^2(\G, \textrm{Hom}(\nu, \mathfrak{i}))\]
does not depend on the choice of $\sigma, \mu$ and $\tau$ (as long as $\sigma$ is compatible with $\mu$ and $\tau$), and the map $K$ from Proposition \ref{regular}
is given by the cup-product with this class.
\end{rmk}

\section{Relation with the infinitesimal theory  (the van Est map)}
\label{sec-Van-Est}

In this subsection we show that the deformation cohomology $H^{*}_{\textrm{def}}(\G)$ is the global (groupoid) analogue of the similar cohomology $H^{*}_{\textrm{def}}(A)$ from the deformation theory of Lie algebroids \cite{marius_ieke}.

Recall first the definition of the deformation complex $(C^*_\text{def}(A),\delta)$ (and the deformation cohomology $H^{*}_{\textrm{def}}(A)$) of a Lie algebroid $A$ over a manifold $M$. The $k$-cochains are antisymmetric multilinear maps $D:\Gamma(A)^k\rmap \Gamma(A)$ which are \textit{multiderivations}, i.e., such that there is a map $\sigma_D:\Gamma(A)^{k-1}\rmap\mathfrak{X}(M)$, called the \textit{symbol} of $D$, which is multilinear and satisfies \[D(\alpha_1,\alpha_2,\ldots,f\alpha_k)=fD(\alpha_1,\alpha_2,\ldots,\alpha_k)+\L_{\sigma_D(\alpha_1,\ldots,\alpha_{k-1})}(f)\alpha_k\]
 for any $f\in C^\infty (M)$ and $\alpha_i\in \Gamma(A)$.
 The differential $\delta:C^k_\text{def}(A)\rmap C^{k+1}_\text{def}(A)$ is given by 
\begin{align*}\delta(D)(\alpha_1,\ldots,\alpha_{k+1}) &= \sum_i(-1)^{i+1}[\alpha_i,D(\alpha_1,\ldots,\hat{\alpha}_i,\ldots, \alpha_{k+1})] \\ & + \sum_{i<j}(-1)^{i+j}D([\alpha_i,\alpha_j],\alpha_1,\ldots,\hat{\alpha}_i,\ldots,\hat{\alpha}_j,\ldots,\alpha_{k+1}).
\end{align*}

Next, we relate the deformation cohomology $H^{*}_{\textrm{def}}(\G)$ of a Lie groupoid $\G$ with the deformation cohomology $H^{*}_{\textrm{def}}(A)$ of the Lie algebroid $A$ of $\G$. This can be seen as a generalization of the more ordinary van Est map relating differentiable cohomology to Lie algebroid cohomology. Given a section $\alpha\in \Gamma(A)$ we can define a map $R_\alpha: \widehat{C}^{k+1}_\text{def}(\G)\rmap \widehat{C}^k_\text{def}(\G)$ by $R_\alpha(c)=[c,\overrightarrow{\alpha}]_{|M}$ when $k=0$ and by

 \begin{equation*}R_{\alpha}(c)(g_1,\ldots,g_k)=(-1)^k\dezero c(g_1,\ldots, g_k,\phi_\epsilon^\alpha(s(g_k))^{-1})
\end{equation*} when $k\geq 1$, where $\phi^\alpha_\epsilon$ denotes the flow of the right-invariant vector field on $\G$ associated to $\alpha$.
The van Est map is the map 
\[ \V:\widehat{C}^*_\text{def}(\G)\rmap C^*_\text{def}(A)\] 
defined by $$(\V c)(\alpha_1,\ldots, \alpha_k)=\sum_{\tau\in S_k}(-1)^{|\tau|}R_{\alpha_{\tau(k)}}\circ\ldots\circ R_{\alpha_{\tau(1)}}(c).$$

Note that the van Est map is only defined on the subcomplex of normalized deformation cochains.

In analogy to the usual van Est theorem relating Lie groupoid cohomology to Lie algebroid cohomology, we obtain the following result:

\begin{thm}\label{VE} For any Lie groupoid $\G$ the van Est map $\V$ is a chain map; hence it induces a map in cohomology
\[ \V:H^p_\mathrm{def}(\G)\rmap H^p_\mathrm{def}(A).\]
Moreover, if $\G$ has $k$-connected $s$-fibres then this map is an isomorphism in all degrees $p\leq k$.
\end{thm}

In section \ref{s-ad} we have seen the construction of the adjoint representation of a Lie groupoid, and the isomorphism of $C^*_\text{def}(\G)$ with $C^*(\G,Ad_\sigma)$; the infinitesimal analogue of this construction gives rise to the adjoint representation $ad_\nabla$ of a Lie algebroid $A$ \cite{camilo3}; it is proven in \textit{loc. cit.} that $C^*_\text{def}(A)$ is isomorphic to $C^*(A;ad_\nabla)$.

We prove Theorem \ref{VE} by showing that under the isomorphisms of $C^*_\text{def}(\G)$ and $C^*_\text{def}(A)$ with $C^*(\G,Ad_\sigma)$ and $C^*(A;ad_\nabla)$, it translates to the van Est theorem of Arias Abad and Sch\"atz (Theorem 4.7 in \cite{camilo2}) for the van Est map  relating the complexes $C^*(\G,Ad_\sigma)$ and $C^*(A;ad_\nabla)$.


We use the concepts and notations from \cite{camilo3, camilo2, marius_ieke} needed for the proof.
Similarly to the setting for Lie groupoids, $\Omega(A,E)=\Gamma(\Lambda^* (A^*)\otimes E)$ is a right graded $\Omega(A)$-module; a representation up to homotopy of $A$ is a graded vector bundle $E$ together with a differential $D$ on $\Omega(A,E)$ which makes it into a differential graded $(\Omega(A), d_A)$-module.

The adjoint representation of a Lie algebroid $A$, seen as a representation up to homotopy, is given by the graded vector bundle $ad=A\oplus TM$, where $A$ has degree $0$ and $TM$ has degree $1$, so that $\Omega(A,ad)^k=\Omega^k(A;A)\oplus \Omega^{k-1}(A;TM)$, and a differential $D_\nabla$ which is defined using a connection $\nabla$ on $A$. When $A$ is the Lie algebroid of a Lie groupoid $\G$, one can take $\nabla$ to be induced by a Ehresmann connection $\sigma$ on $\G$ by letting
\begin{equation}\label{eq-induced_connection}\nabla_X \alpha = [\sigma(X),\overrightarrow{\alpha}]\big{|}_{M},
 \end{equation}
 for any vector field $X$ on $M$ and $\alpha\in \Gamma(A)$. The resulting chain complex is denoted $C^*(A,ad_\nabla)$.


We recall that the isomorphism $C_\text{def}(A)\cong C^*(A;ad_\nabla)$ induced by a  connection $\nabla$ on $A$ is given by $\Psi_\nabla: C_\text{def}(A)\rmap C^*(A;ad_\nabla), \ D\mapsto (L_D,-\sigma_D)$, where $L_D$ is defined by $$L_D(\alpha_1,\ldots,\alpha_k)=D(\alpha_1,\ldots,\alpha_k) + (-1)^{k-1}\sum_i (-1)^i\nabla_{\sigma_D(\alpha_1,\ldots,\widehat{\alpha}_i,\ldots,\alpha_k)}\alpha_i.$$


We now recall the definition of the van Est map from \cite{camilo2}, for cochains with values in representations up to homotopy of length $2$. Let $E=E^0\oplus E^1$ be a graded vector bundle over $M$, and $\alpha$ a section of $A$. Define the map $R_\alpha^l:\widehat{C}^{k+1}(\G;E^l)\rmap \widehat{C}^k(\G;E^l)$ by \[R^l_{\alpha}(c)(g_1,\ldots,g_k)=\dezero c(g_1,\ldots, g_k,\phi_\epsilon^\alpha(s(g_k))^{-1}).\] The van Est map of \cite{camilo2} is the map $\V^{E^l}:\widehat{C}^k(\G;E^l)\rmap \Omega^k(A;E^l)$ which is given by $$(\V^{E^l} c)(\alpha_1,\ldots, \alpha_k)=(-1)^{kl}\sum_{\tau\in S_k}(-1)^{|\tau|}R^{l}_{\alpha_{\tau(k)}}\circ\ldots\circ R^{l}_{\alpha_{\tau(1)}}(c),$$ when $k\geq 1$ and by the identity map when $k=0$. The van Est map for the adjoint representation is obtained by taking $E^0=A$ and $E^1=TM$.

We will show that for an Ehresmann connection $\sigma$ on $\G$, and the  connection $\nabla$ on $A$ induced by $\sigma$ (see formula \ref{eq-induced_connection}), The van Est map $\V$ of deformation cohomology corresponds to the van Est map of \cite{camilo2}, under the isomorphisms $I_\sigma:C_\text{def}(\G)\rmap C(\G,Ad)$ and $\Psi_\nabla:C_\text{def}(A)\rmap C(A,ad)$ induced by $\sigma$. Recall that $I_\sigma(c)=(\omega (c),-s_c)$, where $s_c=ds\circ c$ and $(\omega (c)) (g_1\ldots,g_k)=r_{g_1}^{-1}(c-\sigma\circ s_c)(g_1,\ldots,g_k)$. In order to accommodate for the sign convention of \cite{camilo2}, we will actually use a slightly different map (see remark \ref{sign convention}) $$\tilde{I}_\sigma(c)=((-1)^{\lfloor \frac{k}{2}\rfloor}\omega (c),(-1)^{\lfloor \frac{k-1}{2}\rfloor+1}s_c).$$ 

\begin{lemma}[($R_\alpha$ in degree 1)]\label{VE-degree1} For $c\in \widehat{C}^1_\mathrm{def}(\G)$, it holds that
\begin{equation}\label{R_a1}R_{\alpha}(c):=[c,\overrightarrow{\alpha}]_{|M}=R^0_\alpha(\omega (c))+\nabla_{s_c}\alpha.\end{equation}
\end{lemma}
\begin{proof} Since $c(g)=r_g\omega(c)(g)+\sigma_g s_c(s(g))$, it is enough to prove the equation in the two following cases:

Case 1: $\omega(c)=0$, or equivalently, $c=\sigma_g s_c$; in this case equation (\ref{R_a1}) holds since it is exactly the defining expression for $\nabla$.

Case 2: $s_c=0$, or equivalently, $\omega(c)(g)=r_{g^{-1}}c(g)$; in this case we want to prove that \begin{equation*}[c,\overrightarrow{\alpha}]_{|M}(x)=R^0_\alpha(\omega (c))(x):=\dezero r_{\phi_\epsilon^\alpha(x)}c(\phi_\epsilon^\alpha(x)^{-1}).\end{equation*}

Given any $1$-cochain, normalized or not, consider the section $\gamma_c=c_{|M}-s_c\in\Gamma(A)$. Define a projection $\pi:C^1_\text{def}(\G)\rmap\widehat{C}^1_\text{def}(\G)$ by $\pi(c)=c-\overrightarrow{\gamma_c}$. For a deformation cochain $c$ satisfying $s_c=0$, we will prove that $[c,\overrightarrow{\alpha}]_{|M}(x)=R^0_\alpha(\omega (c))(x)$, which in the case of a normalized cochain $c$ gives the desired result. To do so, first note that $[\pi(c),\overrightarrow{\alpha}]_{|M}(x)-R^0_\alpha(\omega (c))(x)$ is $C^\infty(\G)$-linear in $c$. Indeed,

\begin{align*}[fc,\overrightarrow{\alpha}]_{|M}(x)-[\overrightarrow{\gamma_{fc}},\overrightarrow{\alpha}]_{|M}(x)-R^0_\alpha(\omega (fc))(x)=&f[c,\overrightarrow{\alpha}]_{|M}(x)-f[\overrightarrow{\gamma_{c}},\overrightarrow{\alpha}]_{|M}(x)-fR^0_\alpha(\omega (c))(x)\\
&-\left( \L_{\overrightarrow{\alpha}}f\right)c(x)+\left( \L_{\rho(\alpha)}f\right)c(x)-\left( \L_{\overleftarrow{\alpha}}f\right)c(x),
\end{align*}
and the sum of the last three terms is zero since $\overrightarrow{\alpha}_x+\overleftarrow{\alpha}_x=\rho(\alpha_x).$ To see how the term $\left( \L_{\overleftarrow{\alpha}}f\right)c(x)$ shows up by expanding $R^0_\alpha(\omega (fc))(x)$, it is enough to notice that $\phi_\epsilon^\alpha(x)^{-1}=\phi_\epsilon^{\overleftarrow{\alpha}}(x)$ and apply the chain rule.

Any cochain $c$ with $s_c=0$ is a linear combination of right-invariant ones, with coefficients in $C^\infty(\G)$, and since $[\pi(c),\overrightarrow{\alpha}]_{|M}(x)-R^0_\alpha(\omega (c))(x)$ is $C^\infty(\G)$-linear, it is enough to check that it is zero for a right-invariant $c$. This clearly holds because in this case both $[\pi(c),\overrightarrow{\alpha}]_{|M}(x)$ and $R^0_\alpha(\omega (c))(x)$ will be zero.
\end{proof}

Note that lemma \ref{VE-degree1} says precisely that the map $\V$ on $\widehat{C}^1_\text{def}(\G)$ corresponds to $(\V^A,\V^{TM})$ on $\widehat{C}(\G,Ad)^1$, through the isomorphisms $\tilde{I}_\sigma$ and $\Psi_\nabla$.

\begin{lemma}[$R_\alpha$ in higher degrees]\label{R_ak}
For all $k\geq 1$, the maps $R_\alpha: \widehat{C}^{k+1}_\text{def}(\G)\rmap \widehat{C}^k_\text{def}(\G)$ satisfy \begin{equation*}
R_\alpha=\tilde{I}_\sigma^{-1}\circ (R_\alpha^0,-R_\alpha^1) \circ \tilde{I}_\sigma.\end{equation*}
\end{lemma}
\begin{proof} If $c\in \widehat{C}^{k+1}_\text{def}(\G)$, then 
\begin{align*}\left( \tilde{I}_\sigma^{-1}\circ (R_\alpha^0,\right.&\left.-R_\alpha^1) \circ \tilde{I}_\sigma(c)\right)(g_1,\ldots,g_k)=\\
&=(-1)^{\lfloor \frac{k}{2}\rfloor+\lfloor \frac{k-1}{2}\rfloor}r_{g_1}R_\alpha^0(\omega(c))(g_1,\ldots,g_k)\\
&+(-1)^{\lfloor \frac{k-1}{2}\rfloor+\lfloor \frac{k-2}{2}\rfloor+1}\sigma_{g_1} R_\alpha^1(s_c)(g_2,\ldots,g_k)\\
&=(-1)^k(r_{g_1}R_\alpha^0(\omega(c))(g_1,\ldots,g_k)+\sigma_{g_1} R_\alpha^1(ds\circ c)(g_1,\ldots,g_k))\\
&=(-1)^k(r_{g_1}\omega_{g_1}+\sigma_{g_1}ds)\left( \dezero c(g_1,\ldots, g_k,\phi_\epsilon^\alpha(s(g_k))^{-1}) \right),
\end{align*}
where the last equality follows from checking that, since both $\omega_{g_1}$ and $(ds)_{g_1}$ are linear and do not depend on $\epsilon$, they commute with the operation $\dezero$ in the definitions of $R_\alpha^0,R_\alpha^1$ and $R_\alpha$.
\end{proof}

\begin{proof} (Theorem 10.1)
We can compare the van Est maps $\V$ and $(\V^A,\V^{TM})$ in arbitrary degree as follows. Consider a normalized deformation cochain $c$ of degree $k$, and a permutation $\tau\in S_k$. Lemma \ref{R_ak} implies that $$R_{\alpha_{\tau(k)}}\circ\ldots\circ R_{\alpha_{\tau(1)}}(c)=R_{\alpha_{\tau(k)}}\circ \tilde{I}_\sigma^{-1}\circ ( R^0_{\alpha_{\tau(k-1)}}\circ\ldots\circ R^0_{\alpha_{\tau(1)}},(-1)^{k-1}R^1_{\alpha_{\tau(k-1)}}\circ\ldots\circ R^1_{\alpha_{\tau(1)}})\circ \tilde{I}_\sigma (c),$$
and since we are applying $R_{\alpha_{\tau(k)}}$ to a deformation $1$-cochain, using lemma \ref{VE-degree1} we see that \begin{align}\label{R-tau}
R_{\alpha_{\tau(k)}}\circ\ldots\circ R_{\alpha_{\tau(1)}}(c)
&=(-1)^{\lfloor \frac{k}{2}\rfloor}R^0_{\alpha_{\tau(k)}}\circ R^0_{\alpha_{\tau(k-1)}}\circ\ldots\circ R^0_{\alpha_{\tau(1)}}(\omega(c)) \\
&+(-1)^{\lfloor \frac{k-1}{2}\rfloor+1}\nabla_{(-1)^{k-1}R^1_{\alpha_{\tau(k-1)}}\circ\ldots\circ R^1_{\alpha_{\tau(1)}}(s_c)}\alpha_{\tau(k)} \nonumber.
\end{align}

The van Est map $\V$ is obtained by summing the previous expressions over all permutations in $S_k$. We can split this into a double sum, summing over $i=1,\dots,k$ and over the permutations $\tau\in S_k$ such that $\tau(k)=i$. Using this resummation and equation (\ref{R-tau}), we have that

\begin{align}\label{VE-aux}(\V c)(\alpha_1,\ldots, \alpha_k)
&=\sum_{i=1}^k\sum_{\stackrel{\tau\in S_k}{\tau(k)=i}}(-1)^{|\tau|}R_{\alpha_{\tau(k)}}\circ\ldots\circ R_{\alpha_{\tau(1)}}(c)\nonumber \\
&=(-1)^{\lfloor \frac{k}{2}\rfloor}\sum_{i=1}^k(-1)^{|\tau|}R^0_{\alpha_{\tau(k)}}\circ R^0_{\alpha_{\tau(k-1)}}\circ\ldots\circ R^0_{\alpha_{\tau(1)}}(\omega(c)) \\
&+ (-1)^{\lfloor \frac{k-1}{2}\rfloor+1}\sum_{i=1}^k\sum_{\stackrel{\tau\in S_k}{\tau(k)=i}}(-1)^{|\tau|}\nabla_{(-1)^{k-1}R^1_{\alpha_{\tau(k-1)}}\circ\ldots\circ R^1_{\alpha_{\tau(1)}}(s_c)}\alpha_i\nonumber
\end{align}

If $\tau\in S_k$ is a permutation with $\tau(k)=i$, by composing the cycle $r_i=(k\ \ k-1\ \ldots \ i+1\ \ i)$ with it, we obtain a permutation $\tau'=r_i\circ \tau$ that fixes $k$, so it can be seen as element of $S_{k-1}$ (and any element of $S_{k-1}$ is of this form), for which we have $(-1)^{|\tau'|}=(-1)^{|\tau|+|r_i|}=(-1)^{|\tau|+(k-(i-1))}$. From this, and the definitions of $\V^A, \V^{TM}$ we see that (\ref{VE-aux}) is equal to

$$(-1)^{\lfloor \frac{k}{2}\rfloor}\V^A\omega(c)(\alpha_1,\ldots,\alpha_k)+(-1)^{\lfloor \frac{k-1}{2}\rfloor+1}\sum_i^k (-1)^{|r_i|}\nabla_{\V^{TM}s_c(\alpha_1,\cdots,\widehat{\alpha}_i,\cdots,\alpha_k)}(\alpha_i),$$ which is precisely the expression for $\Psi_\nabla^{-1}\circ (\V^A,\V^{TM})\circ \tilde{I}_\sigma(c)(\alpha_1,\ldots,\alpha_k)$, meaning that $\V$ does indeed correspond to $(\V^A, \V^{TM}).$ To finish the proof, we apply the van Est theorem (Theorem 4.7 in \cite{camilo2}) for $(\V^A, \V^{TM}).$
\end{proof}

\begin{rmk}
A version of the theorem above, valid for more general representations up to homotopy, was proved by Arias Abad and Sch\"atz in \cite{camilo2}. The main point of our theorem is that it relates \emph{directly and canonically} the deformation complexes of groupoids and algebroids. Using the result of \cite{camilo2}, one would only obtain the desired van Est map after choosing connections and identifying the deformation complexes with the respective adjoint complex. 

An alternative proof using the VB-interpretation of the adjoint cohomology  (see \cite{mehta} or subsection \ref{poisson}) has been recently communicated to us by Cabrera and Drummond, and will appear in \cite{ale_thiago}.
\end{rmk}

\section{Morita invariance}

In this section we prove that Morita equivalent Lie groupoids have isomorphic deformation cohomologies. This is useful not only for computations but also for conceptual reasons: it shows that the deformation cohomology of a Lie groupoid $\G$ is an invariant of the 
differentiable stack presented by $\G$.

\begin{rmk} 
In this direction, it is interesting to note that the notion of invariant vector fields up to isomorphism of \cite{lerman} is related to $H^1_{\mathrm{def}}(\G)$, where $\G$ is an action groupoid.
\end{rmk}

We recall briefly some basics on Morita equivalences. Detailed expositions can be found in \cite{ieke_mrcun, mrcun}.
If $\G$ is a Lie groupoid over $M$, a \textit{(left) action of $\G$} on a manifold $P$ is given by a smooth map $\mu: P\rmap M$, called \textit{moment map}, and a smooth map $\G\times_M P\rmap P$, with $(g,p)\mapsto g\cdot p\in \mu^ {-1}(t(g))$, satisfying the usual action axioms. Here, the fibred product is taken over the source map of $\G$ and over $\mu$.

A \textit{left $\G$-bundle} is a left $\G$-space $P$ together with a $\G$-invariant surjective submersion $P\rmap B$. It is called \textit{principal} if the map $\G\times_M P\rmap P\times_B P,\ (g,p)\mapsto (gp,p)$ is a diffeomorphism. The notions of right action, and right $\G$-bundle are defined similarly.

Lie groupoids $\G$ over $M$ and $\H$ over $N$ are said to be \textit{Morita equivalent} if there is a manifold $P$ together with a left action of $\G$ on $P$ with moment map $\alpha:P\rmap M$ and a right action of $\H$ on $P$ with moment map $\beta:P\rmap N$, such that $\beta:P\rmap N$ is a principal $\G$-bundle, $\alpha:P\rmap M$ is a principal $\H$-bundle, and the two actions commute. We then say that $P$ is a \textit{bibundle} realizing this Morita equivalence between $\G$ and $\H$.

\begin{example}(Isomorphisms) If $f: \G \rmap \H$ is an isomorphism of Lie groupoids, then $\G$ and $\H$ are Morita equivalent. A bibundle can be given by the graph $Graph(f)\subset \G\times\H$, with moment maps $t\circ pr_1$ and $s\circ pr_2$, and the natural actions induced by the multiplications of $\G$ and $\H$. 
\end{example}

\begin{example}\label{pullback-morita}(Pullback groupoids) Let $\G$ be a Lie groupoid over $M$ and a $\alpha: P\rmap M$ a surjective submersion. Then we can form the \textit{pullback groupoid} $\alpha^*\G\arrows P$, that has as space of arrows $P\times_M \G\times_M P$, meaning that arrows are triples $(p,g,q)$ with $\alpha(p)=t(g)$ and $s(g)=\alpha(q)$. The structure maps are determined by $s(p,g,q)=q$, $t(p,g,q)=p$ and  $(p,g_1,q)(q,g_2,r)=(p,g_1g_2,r)$.

The groupoids $\G$ and $\alpha^*\G$ are Morita equivalent, a bibundle being given by $\G\times_M P$. The left action of $\G$ has moment map $t\circ pr_1:\G\times_M P \rmap M$ is given by $g\cdot(h,p)=(gh,p)$ and the right action of $\alpha^*\G$  has moment map $pr_2:\G\times_M P\rmap P$ and is given by $(h,p)\cdot(p,k,q)=(hk,q)$.

\end{example}

\begin{example}\label{cech}(\v{C}ech groupoids)
The following special case of the example above arises in Mayer-Vietoris type arguments concerning Lie groupoids (see proposition \ref{prop: MV} and the proof of theorem \ref{thm: morita} bellow).

Let $ \mathcal{U} = \{ U_i\}_{i \in J}$ be an open cover of $M$ and let $\pi: \amalg U_i \rmap M$ denote the obvious surjective submersion. The \textit{\v{C}ech groupoid of $\G$ w.r.t $\mathcal{U}$} is defined to be the groupoid 
\[\check{\mathcal{\G}}(\mathcal{U}) = \pi^*\G \tto \amalg U_i,\]
and it is Morita equivalent to $\G$.  Note that an arrow $(x, g, y)$ of $\check{\mathcal{\G}}(\mathcal{U})$ can be unambiguously denoted by $(i, g, j)$ where $s(i,g,j) = y \in U_j$, and $t(i,g,j) = x \in U_i$   
\end{example}

\begin{rmk}\label{simplemorita}

Let $\G$ and $\H$ be Morita equivalent, with bibundle $P$ as above. Using that $P$ is a principal bibundle, it is easy to check that $$\alpha^* \G=P\times_M \G\times_M P\cong P\times_M P\times_N P\cong P\times_N \H\times_N P =\beta^*\H,$$ as Lie groupoids over $P$.

This means that we can break a Morita equivalence between $\G$ and $\H$, using a bibundle $P$, into a chain of simpler Morita equivalences: $\G$ is Morita equivalent to $\alpha^*\G\cong\beta^*\H$, which is Morita equivalent to $\H$. Therefore, in order to check invariance under Morita equivalences of a property or construction associated to a Lie groupoid, it is enough to check invariance under isomorphisms and under Morita equivalences between a Lie groupoid and its pullback by a surjective submersion.

\end{rmk}

\begin{thm}[(Morita invariance)]\label{thm: morita}
If two Lie groupoids are Morita equivalent, then their deformation cohomologies are isomorphic.
\end{thm}

\begin{proof} By remark \ref{simplemorita}, it is enough to prove invariance under isomorphisms and under pullback by a surjective submersion $f:P\rmap M$.

For both cases we use that given a morphism of Lie groupoids $F:\G\rmap \H$, we have maps $$C^{k}_{\textrm{def}}(\G) \stackrel{F_*}{\rmap} C^{k}_{\textrm{def}}(F) \stackrel{F^*}{\lmap} C^{k}_{\textrm{def}}(\H)$$ relating the deformation cohomologies of $\G$, $F$ and $\H$ (remark \ref{def-morphisms}). When $F$ is an isomorphism, $F_*$ and $F^*$ will be isomorphisms of chain complexes, so that takes care of invariance under isomorphims. We will now focus on the case where $f:P\rmap M$ is a surjective submersion, and $F$ is the induced map $F:f^*\G\rmap \G,\ \tilde{g}=(p,g,q)\mapsto g$. We will prove that $C^*_{\textrm{def}}(f^*\G)$ is quasi-isomorphic to $C^*_{\textrm{def}}(\G)$ by showing that both are quasi-isomorphic to $C^*_{\textrm{def}}(F)$.

\begin{claim} $F_*$ is a quasi-isomorphism between $C^*_{\textrm{def}}(f^*\G)$ and $C^*_{\textrm{def}}(F)$.
\end{claim}
\begin{proof}[Proof of Claim 1:] Recall that $F_* c=dF\circ c$. We first note that $F_*$ is surjective. In fact, if we  chose a connection on $TP$, i.e., a right splitting $\Gamma$ on the sequence $$0\rmap\ker(f)\rmap TP\rmap f^*TM\rmap 0,$$ then given $c\in C^*_{\textrm{def}}(F)$, if we set $\tilde{c}=(\Gamma(dt(c)),c,\Gamma(ds(c)))$, we obtain that $F_*\tilde{c}=c$.

We thus obtain a short exact sequence
\[0 \longrightarrow \ker F_* \longrightarrow C^*_{\textrm{def}}(f^*\G) \longrightarrow C^*_{\textrm{def}}(F) \longrightarrow 0,\]
so to prove the claim it is enough to show that $\ker F_*$ is acyclic. 

Note that $k$-cochains in $\ker F_*$ may be identified with pairs of maps
\[u: (f^\ast \G)^k \rmap \ker df, \quad v: (f^\ast \G)^{k-1} \rmap \ker df, \]
where $u(\tilde{g_1},\ldots,\tilde{g_k})\in \ker(df)_{p_1}$, $v(\tilde{g_2},\ldots,\tilde{g_k})\in \ker(df)_{q_1}$, and $\tilde{g_i}$ denotes the arrow $(p_i,g_i,q_i)$.

A simple computation shows that under this identification, the differential of the deformation complex satisfies 
$$\delta (u,0) = (w,u),$$ for some $w:  (f^\ast \G)^{k+1} \rmap \ker df$ as above. Then, if $(u,v)$ is a cocycle, it must actually be exact, with $(u,v)=\delta(v,0)$. Indeed, $(u,v)-\delta (v,0)=(w,0)$, and $(w',w)=\delta(w,0)=\delta(u,v)-\delta^2(v,0)=0$, so $w=0$. 
\end{proof}

Next, we reduce our problem of showing that $H^*_{\mathrm{def}}(f^*\G) \simeq H^*_{\mathrm{def}}(\G)$ to the case where $f:P \rmap M$ admits a global section $\sigma: M \rmap P$. For this we will need the following Mayer-Vietoris argument, where we use the \v{C}ech groupoid $\check{\mathcal{\G}}(\mathcal{U})$ associated to an open cover $\mathcal{U} = \{U_i\}$ of $M$ (example \ref{cech}).

\begin{claim}[(Mayer-Vietoris Argument)]\label{MV}
If $\mathcal{U}$ is an open cover of $M$, then $H^*_{\mathrm{def}}(\G) \simeq H^*_{\mathrm{def}}(\check{\G}(\mathcal{U}))$.
\end{claim} 

Since the proof of this claim is formally identical to the proofs of claims \ref{aux-morita} and \ref{induction} below (with the role of the section $\sigma$ to be played by a partition of unity), we will postpone it until Proposition \ref{prop: MV}.

\begin{claim}\label{global section}
We may assume without loss of generality that the map $f:P \rmap M$ admits a global section $\sigma: M \rmap P$.
\end{claim}

\begin{proof}[Proof of Claim \ref{global section}:]
Take a cover $\mathcal{U} = \{U_i\}$ of $M$ by open sets which admit local sections $\sigma_i: U_i \rmap f^{-1}U_i$, and denote by $\mathcal{V}$ the open cover of $P$ by the open sets $V_i = f^{-1}U_i$. Note that $f$ induces a surjective submersion with a global section
\[\check{f}: \amalg V_i \longrightarrow \amalg U_i.\]

Moreover, the \v{C}ech groupoid $\check{(f^*\G)}(\mathcal{V})$ of $f^*\G$ with respect to the cover $\mathcal{V}$ is isomorphic to the pullback groupoid $\check{f}^*(\check{\G}(\mathcal{U}))$ of the \v{C}ech groupoid $\check{\G}(\mathcal{U})$ by the surjective submersion $\check{f}$. Thus, by Claim \ref{MV}, and invariance under isomorphism, we have
\[H^*_{\mathrm{def}}(\check{f}^*(\check{\G}(\mathcal{U}))) \simeq H^*_{\mathrm{def}}(f^*\G) \text{ and } H^*_{\mathrm{def}}(\check{\G}(\mathcal{U})) \simeq H^*_{\mathrm{def}}(\G).\]

Thus, $H^*_{\mathrm{def}}(\G) \simeq H^*_{\mathrm{def}}(f^*\G)$ if and only if $H^*_{\mathrm{def}}(\check{\G}(\mathcal{U})) \simeq H^*_{\mathrm{def}}(\check{f}^*(\check{\G}(\mathcal{U})))$.
\end{proof}

From now until the end of the proof of the theorem we will assume that $f: P \rmap M$ admits a global section $\sigma: M \rmap P$. We then obtain a left inverse to $F^*$:
\[\Psi_\sigma: C^*_{\mathrm{def}}(F) \rmap C^*_{\mathrm{def}}(\G), \quad \Psi_\sigma(c)(g_1,\cdots,g_k) = c(\sigma(t(g_1)), g_1, \sigma(s(g_1)),\ldots, \sigma(t(g_k)), g_k, \sigma(s(g_k))). \]
It follows that the map induced by $F^*$ in cohomology is injective, and it is our task to show that it is surjective. For this we consider the following decreasing sequence of subcomplexes of $C^*_{\mathrm{def}}(F)$.

We will call a cochain $c\in C^k_{\textrm{def}}(F)$ \emph{strongly normalized in the $j-th$ position}, if it satisfies the following two conditions:
\begin{equation}\label{eq: l-basic}
c(\tilde{g_1},\ldots,\tilde{g}_j, \ldots,\tilde{g_k}) = c(\tilde{g_1},\ldots, \tilde{g}'_j,\ldots,\tilde{g_k}) \text{ whenever } F(\tilde{g}_j) = F(\tilde{g}'_j)\end{equation}
and if $F(\tilde{g}_j)$ is a unit in $\G$, then 
\begin{equation}\label{eq: l-normalized}
\begin{array}
c c(\tilde{g_1},\ldots,\tilde{g}_j, \ldots,\tilde{g_k}) = 0\ \text{ if } 1<j \leq k\\
c(\tilde{g_1},\ldots,\tilde{g_k}) \text{ is a unit of $T\G$\ if } j = 1.
\end{array}
\end{equation}
We denote by $\mathcal{N}^*\subset C^*_{\textrm{def}}(F)$ the complex of \emph{strongly normalized cochains}, i.e., those that are strongly normalized in all positions.

We obtain a decreasing sequence of subcomplexes 
\[C^*_{\textrm{def}}(F) = \mathcal{N}^*_0 \supseteq \cdots \supseteq \mathcal{N}^*_\ell \supseteq \cdots\supseteq \mathcal{N}^*,\] where we set
\[\mathcal{N}^k_\ell = \{c \in C^k_{\textrm{def}}(F) \text{ such that } c\text{ is strongly normalized in position } j \text{ for all } k-\ell <  j \}. \] 

\begin{claim}\label{N=CdefG}
The complex $\mathcal{N}^*$ is isomorphic to the complex $\widehat{C}^*_{\mathrm{def}}(\G)$ of normalized cochains of $\G$.
\end{claim}

\begin{proof}[Proof of Claim \ref{N=CdefG}:]
The map $F^*:C^k_{\textrm{def}}(\G)\rmap C^k_{\textrm{def}}(F)$ is given by \[(F^*c)((p_1,g_1,q_1),\ldots,(p_k,g_k,q_k))=c(g_1,\ldots,g_k),\] so restricting it to normalized cochains on $\G$, we obtain a map $F^*:\widehat{C}^k_\text{def}(\G)\rmap \mathcal{N}^k$. This restriction is injective because of surjectivity of $f$, and is surjective because of condition \eqref{eq: l-basic} which holds for all $1 \leq \ell \leq k$. (Alternatively, the restriction of $\Psi_\sigma$ to $\mathcal{N}^*$ is the inverse of $F^*$). 
\end{proof}

In order to show that $F^*$ is surjective in cohomology, we will prove that every cocycle in $\mathcal{N}^k_\ell$ is cohomologous to a cocycle in $\mathcal{N}^k_{\ell + 1}$. We break this into two steps (the next two claims).

\begin{claim}
\label{aux-morita} If a cocycle $c\in C^k_{\textrm{def}}(F)$ satisfies condition \eqref{eq: l-normalized} for all $\ell \leq j \leq  k$, then it also satisfies condition \eqref{eq: l-basic} for all $\ell \leq j \leq  k$.
\end{claim}

\begin{proof}[Proof of Claim \ref{aux-morita}:] For $j\geq \ell$, non-dependence on $q_j$ follows from spelling out the cocycle equation $$\delta c((p_1,g_1,q_1),\ldots,(p_j,g_j,q_j),(q_j,1,q'_j),(p_{j+1},g_{j+1},q_{j+1}),\ldots,(p_k,g_k,q_k))=0.$$ If $\ell>1$, since $c$ satisfies condition \eqref{eq: l-normalized} in positions $\ell,\ldots,k$, only two terms survive, and the result is
$$c(\tilde{g}_1,\ldots,(q_j,g_{j+1},q_{j+1}),\ldots,\tilde{g}_k) = c((\tilde{g}_1,\ldots,(q'_j,g_{j+1},q_{j+1}),\ldots,\tilde{g}_k),$$
where $\tilde{g}_i$ denotes the arrow $(p_i, g_i, q_i)$.

If $\ell=1$ and $j = 1$, the only non-zero terms in the cocycle condition are the first two, and the result is that \begin{align*}-d\bar{m}(c((p_1,g_1,q'_1),(q_1',g_2, q_2),\ldots,\tilde{g}_k )&, c((q_1,1,q'_1),(q_1',g_2, q_2),\ldots\tilde{g}_k ))\\
&+ c((p_1,g_1,q_1), (q_1,g_2, q_2),\ldots, \tilde{g}_k )=0.
\end{align*} Since $c$ satisfies condition \eqref{eq: l-normalized} in position $1$, it follows that $c(q_1,1,q'_1,\ldots)$ is a unit, and thus
\[c((p_1,g_1,q'_1),(q_1',g_2, q_2),\ldots,\tilde{g}_k ) =  c((p_1,g_1,q_1),(q_1,g_2, q_2),\ldots,\tilde{g}_k ).\]

Non-dependence on $p_1$ follows simillarly from the cocycle equation $$\delta c((p_1,1,p'_1),(p'_1,g_1,q_1),\ldots,(p_k,g_k,q_k))=0.$$ Spelling it out, we obtain $$-d\bar{m}(c((p_1,g_1,q_1),\ldots,\tilde{g}_k), c((p'_1,g_1,q_1),\ldots,\tilde{g}_k))= \sum_i U_i,$$ where condition \eqref{eq: l-normalized} implies that each $U_i$ is a unit of $T\G$. The result then follows by multiplying (in $T\G$) both sides of the equation by $c((p'_1,g_1,q_1),\ldots,\tilde{g}_k)$ on the right.
\end{proof}

The next claim concludes the proof of the theorem.

\begin{claim}\label{induction}
Every cocycle $c \in \mathcal{N}^k_{\ell}$ is cohomologous to a cocycle in $\mathcal{N}^k_{\ell+1}$.
\end{claim}

\begin{proof}[Proof of Claim \ref{induction}:]
The notation becomes quite heavy in the following computations, so we make some simplifications, which we now explain, and which will not be used outside of this proof.
For arrows $\tilde{g_i}=(p_i,g_i,q_i)$, we use the notation $(\tilde{g_1},\ldots,\tilde{g_k})=(p_1,g_1,q_1,g_2,\ldots,g_{k-1},q_{k-1},g_k,q_k)$, which is not ambiguous since $p_{i+1} = q_i$. We also abbreviate the expression by inserting $\tilde{g_i}=(p_i,g_i,q_i)$ in parts of expressions where the arrow is clear from its context. Moreover, we will implicitly assume when applying a cochain to a string of arrows, that they are composable, and omit the points of $M$ where we apply the unit map and the section $\sigma$ if they are uniquely determined by this requirement (composability). So, for example, we would simplify the expression $c((p,g,\sigma(s(g)),(\sigma(s(g)),1_{s(g)},q),(q,h,r))$ to $c(p,g,\sigma,1,q,h,r)$.

We consider first the case $\ell \geq 2$. Let $c \in \mathcal{N}^k_{k-\ell} \subset C^k_{\mathrm{def}}(F)$ be a cocycle. Consider $\varphi^\ell_\sigma(c) \in  C^{k-1}_{\mathrm{def}}(F)$ given by
$$\varphi^\ell_\sigma(c)(p_1,g_1,q_1,\ldots,g_{k-1},q_{k-1})=(-1)^{\ell+1}c(p_1,g_1,q_1,\ldots,g_{\ell-1},q_{\ell-1},1,\sigma,g_\ell,q_\ell,\ldots,g_{k-1},q_{k-1}).$$

We claim that $c + \delta \varphi^\ell_\sigma(c)$ belongs to $\mathcal{N}^k_{k - \ell + 1}$. In fact, let us compute 
$(c + \delta \varphi^\ell_\sigma(c))(\tilde{g}_1, \ldots \tilde{g}_k)$ when $F(\tilde{g}_i) = 1$. If $i > \ell$, then since $c \in \mathcal{N}_{k-\ell}^k$, most terms cancel and the only non-zero terms are the ones resulting from the strings containing $\tilde{g}_{i-1}\tilde{g}_i$, and $\tilde{g}_{i}\tilde{g}_{i+1}$, both with opposite signs, i.e., up to a sign we obtain
\begin{align*}c(p_1, g_1, \ldots,q_{\ell-1}, 1 , \sigma, &g_{\ell}, \ldots,  q_{i-2}, g_{i-1}, q_{i}, \ldots g_k, q_k) - \\ &c(p_1,g_1, \ldots,q_{\ell-1}, 1 , \sigma, g_{\ell}, \ldots,  q_{i-2}, g_{i-1}, q_{i-1}, \ldots g_k, q_k) = 0\end{align*}
which vanishes because $c$ satisfies condition \eqref{eq: l-basic} in position $i > \ell$. This show that $c + \delta \varphi^\ell_\sigma(c) \in \mathcal{N}_{k-\ell}^k$. 

Now let $i = \ell$. We will show that if $F(\tilde{g}_\ell) = 0$, then $(c + \delta \varphi^\ell_\sigma(c))(\tilde{g}_1, \ldots, \tilde{g}_k) = 0$, i.e., $c + \delta \varphi^\ell_\sigma(c)$ satisfies condition \eqref{eq: l-normalized} in position $\ell$. It then follows from Claim \ref{aux-morita} that $c + \delta \varphi^\ell_\sigma(c) \in \mathcal{N}^k_{k-\ell +1}$.

After a straightforward compution 
we obtain that
\begin{align*}(c + \delta \varphi^\ell_\sigma(c))(\tilde{g}_1&, \ldots, g_{\ell -1}, q_{\ell-1}, 1, q_\ell , \ldots, \tilde{g}_k) =\\
&=(-1)^{\ell+1}\delta c (\tilde{g}_1, \ldots, g_{\ell -1}, q_{\ell-1}, 1, q_\ell, 1, \sigma, g_{\ell+1} \ldots, \tilde{g}_k) = 0
\end{align*}
which vanishes because $c$ is a cocycle.

We are left with showing the cases $\ell = 1$. Let $c \in \mathcal{N}^k_{k-1}$ be a cocycle which is strongly normalized in positions $2, \ldots, k$. In this case we use the canonical splitting $T_{1_x}\G\cong A_x\oplus T_xM$, and denote by $X^A$ the $A$-component of a vector $X\in T_{1_x}\G$, to define  $\varphi_\sigma^1(c) \in C^{k-1}_{\mathrm{def}}(F)$ by $$\varphi_\sigma^1(c)(p_1,g_1,q_1,\ldots,g_{k-1},q_{k-1})=r_{g_1}c(p_1,1,\sigma,g_1,q_1,\ldots,g_{k-1},q_{k-1})^A.$$ Similarly, if $c$ has degree $1$, define $\varphi_\sigma^1(c)$ by $\varphi_\sigma^1(c)(p)=c(p,1,\sigma)^A$.

We wish to show that $(c+\delta \varphi_\sigma^1(c))(p_1,1,q_1,g_2,\ldots,g_k)$ is a unit of $T\G$. Most terms of this expression are zero and we are left with 
\begin{align*}(c+\delta \varphi_\sigma^1(c))(p_1,1&,q_1,g_2,\ldots,g_k)= \\ &=c(p_1,1,q_1,g_2,\ldots,g_k)\\
&-d\bar{m}(r_{g_2}c(p_1,1,\sigma,g_2,\ldots,g_k)^A,(r_{g_2}c(q_1,1,\sigma,g_2,\ldots,g_k)^A)
\end{align*} 
If we decompose $c(p_1,1,q_1,g_2,\ldots,g_k) = ds\circ c(p_1,1,q_1,g_2,\ldots,g_k)+c(p_1,1,q_1,g_2,\ldots,g_k)^A$ and use the cocycle identity
\begin{align*}
- d\bar{m}(r_{g_2}c(p_1,1,\sigma,g_2,\ldots,g_k)^A,(r_{g_2}c(q_1,1,\sigma&,g_2,\ldots,g_k)^A) + c(p_1,1,q_1,g_2,\ldots,g_k)^A \\ &= (\delta c(p_1,1,q_1,1, \sigma g_2,\ldots,g_k))^A = 0,
\end{align*}
we obtain that
\[(c+\delta \varphi_\sigma^1(c))(p_1,1,q_1,g_2,\ldots,g_k)= ds\circ c(p_1,1,q_1,g_2,\ldots,g_k) \text{ is a unit.}\]

The case $k=1$ works exactly in the same way: $(c+\delta \varphi_\sigma^1(c))(p,1,q)$ is the unit $ds\circ c(p,1,q)$, plus the $A$-part of $\delta c(p,1,q,1,\sigma)$.
\end{proof}

It follows from the previous two claims that $F^*: H^k_{\mathrm{def}}(\G) \rmap H^k_{\mathrm{def}}(F)$ is surjective, and thus an isomorphism. This concludes the proof of the theorem.
\end{proof}

\begin{rmk}\label{rmk: normalized}
In the course of the proofs of claims \ref{aux-morita} and \ref{induction}, we actually showed that every cocycle in $C^*_{\mathrm{def}}(F)$ is cohomologous to one in $\mathcal{N}^*$, which we identify with $\widehat{C}^*_{\mathrm{def}}(\G)$ (and not simply $C^*_{\mathrm{def}}(\G)$). Thus, applying the proof above to the identity map yields the following proposition.
\end{rmk}
\begin{prop}\label{prop: normalized}
The inclusion $\iota: \widehat{C}^*_{\mathrm{def}}(\G) \hookrightarrow C^*_{\mathrm{def}}(\G)$ is a quasi-isomorphism. 
\end{prop}

We now prove the Mayer-Vietoris argument used in the proof above (Claim \ref{MV})

\begin{prop}[(Mayer-Vietoris Argument)]\label{prop: MV}
If $\mathcal{U}$ is an open cover of $M$, then $H^*_{\mathrm{def}}(\G) \simeq H^*_{\mathrm{def}}(\check{\G}(\mathcal{U}))$.
\end{prop}

\begin{proof}
The proof is formally identical to the proofs of theorem \ref{thm: morita} above (Claims \ref{aux-morita} and \ref{induction}) with the role of the section $\sigma$ replaced by a partition of unity $\{\rho_j\}$ subordinate to the open cover $\mathcal{U} = \{U_j\}$. Here are the main ingredients: 

\begin{itemize}
\item $\check{\G}(\mathcal{U})$ is the pullback of $\G$ by a submersion, so there is an obvious chain map 
\[\pi^*: C^k_{\mathrm{def}}(\G) \rmap C^k_{\mathrm{def}}(\check{\G}(\mathcal{U})), \qquad \pi^*(c)((i_1, g_1, j_1), \ldots , (j_{k-1},g_k,j_k)) = c(g_1, \ldots, g_k).\]

\item Using the partition of unity we obtain a left inverse
$\Psi: C^k_{\mathrm{def}}(\check{\G}(\mathcal{U})) \rmap C^k_{\mathrm{def}}(\G)$, \[\Psi(c)(g_1, \ldots, g_k) = \sum\rho_{i_1}\rho_{j_1}\cdots\rho_{j_k}c((i_1, g_1, j_1), \ldots , (j_{k-1},g_k,j_k)), \]
where the sum is taken over all indexes $i_1, j_r$ such that $t(g_1) \in U_{i_1}$ and $s(g_r) \in U_{j_r}$, with $r = 1,\ldots, k$. 

\item It follows that $\pi^*$ induces an injection in cohomology and all that is left to prove is that it is also surjective.

\item The normalized complex $\widehat{C}^*_{\mathrm{def}}(\G)$ can be identified via $\Psi$ with a subcomplex $\mathcal{N}^*$ of strongly normalized cochains of $C^k_{\mathrm{def}}(\check{\G}(\mathcal{U}))$. (Defined by conditions analogous to \eqref{eq: l-basic} and \eqref{eq: l-normalized}). In order to check that the map induced by $\pi^*$ in cohomology is surjective, it is enough to show that every cocycle $c \in  C^k_{\mathrm{def}}(\check{\G}(\mathcal{U}))$ is cohomologous to a cocycle in $\mathcal{N}^k$.

\item We consider the descending sequence 
\[C^*_{\textrm{def}}(\check{\G}(\mathcal{U})) = \mathcal{N}^*_0 \supseteq \cdots \supseteq \mathcal{N}^*_\ell \supseteq \cdots\supseteq \mathcal{N}^*,\] 
defined as in the proof of theorem \ref{thm: morita}.
\item The statement and proof of Claim \ref{aux-morita} in Theorem \ref{thm: morita} follow identically with $p_r$ replaced by $i_r$, and $q_r$ replaced by $j_r$.
\item The statement and proof of Claim \ref{induction} in Theorem \ref{thm: morita} follow identically with $\varphi^{\ell}_{\sigma}$ replaced by
\[\varphi^{\ell}(c)(i_1, g_1, j_1, \ldots, g_k, j_k) = (-1)^{\ell+1}\sum_{j}\rho_j c(i_1,g_1,j_1,\ldots,g_{\ell-1},j_{\ell-1},1,j,g_\ell,j_\ell,\ldots,g_{k-1},j_{k-1})\] for $\ell>1$, where the sum is taken over all indices $j$ such that $t(g_{\ell}) \in U_j$; similarly, $\varphi_\sigma^1(c)$ is replaced by  \[\varphi^1(c)(i_1,g_1,j_1,\ldots,g_{k},j_{k})=\sum_{j}\rho_j r_{g_1}c(i_1,1,j,g_1,j_1,\ldots,g_{k},j_{k})^A.\]

\end{itemize}
\end{proof}

\begin{rmk}
We have recently learned that theorem \ref{thm: morita} can be obtained also by using an appropriate notion of Morita equivalence for VB-groupoids (and using the VB-groupoid interpretation of the deformation cohomology - see subsection \ref{poisson}). This is being worked out in an ongoing project of del Hoyo and Ortiz \cite{matias_cristian}.
\end{rmk}


\section*{Appendix: Another way of looking at groupoids}\label{a-note}


Since the main motivation (and applications) for the  deformation cohomology comes from the study of deformations of Lie groupoids, in order to gain some insight into its definition it is worth contemplating a bit on the meaning of deformations (and of Lie groupoids). Given a groupoid $\G$, we want to allow deformations 
smoothly parametrized by some real parameter $\epsilon$ of all the structure maps: $s_{\epsilon}$, $t_{\epsilon}$, $m_{\epsilon}$ etc.  The cohomology theory that controls deformations should incorporate the variation of the structure maps ($\frac{d}{d\epsilon} s_{\epsilon}$, etc); the cocycle conditions should be first order consequences 
of (i.e. obtained by applying $\frac{d}{d\epsilon}$ to) the various equations that the structure maps satisfy. 
There are two relevant points to be addressed right from the start:
\begin{itemize}
\item In the set of Lie groupoid axioms there is a certain redundancy (e.g. the target map is determined by the source and the inversion: $t= s\circ i$). In order to study deformations, it is natural to look for a minimal set of axioms.
\item Variations of type ``$\frac{d}{d\epsilon} m_{\epsilon}(g, h)$'' are problematic because they make sense only under very restrictive conditions (e.g. when $s_{\epsilon}$ and $t_{\epsilon}$ do not depend on $\epsilon$, so that the condition that $g$ and $h$ are composable does not depend on on $\epsilon$). 
\end{itemize}
Both points are answered by a very simple remark: it is better to use the 
division map $\bar{m}$ of $\G$ instead of the multiplication; moreover, all the structure maps of $\G$ can be recovered from only $\bar{m}$ and $s$ - themselves satisfying some simple axioms.

\begin{prop}\label{prop: m bar}
Given the manifolds $\G$ and $M$, giving structure maps $(m, s, t, u, i)$ making $\G$ into a Lie groupoid over $M$ is equivalent to giving pairs $(s, \bar{m})$ consisting of surjective submersions $s: \G\rmap M$ and $\bar{m}: \G\times_{s} \G\rmap \G$ (defined on the space of pairs $(g, h)$ with $s(g)= s(h)$) satisfying:
            \begin{enumerate}
            \item[(i)] For all $g, h\in \G$ with $s(g)= s(h)$ one has
            \[ s(\bar{m}(g, h))= s(\bar{m}(h, h)).\]
            (i.e. the expressions of type $s(\bar{m}(g, h))$ only depend on $h$). 
            \item[(ii)] For all $g, h, k\in \G$ with $s(g)= s(h)= s(k)$ one has
            \[ \bar{m}(\bar{m}(g, k), \bar{m}(h, k))= \bar{m}(g, h) \]
            (note: the first expression makes sense because of (i)).
            \item[(iii)] The restriction of $s$ to $\bar{m}(\Delta)= \{\bar{m}(g, g): g\in \G\}$ is injective. 
            \end{enumerate}  
\end{prop}

\begin{proof} We have to see how we can recover all the structure maps and their axioms given a pair $(s, \bar{m})$ as in 2. Note first that any $x\in M$ can be written as $s(\bar{m}(g, g))$ for some $g\in \G$. This follows by using the surjectivity of $s$ and $\bar{m}$ and the fact that $s(\bar{m}(g, h))= s(\bar{m}(h, h))$ (by (i)). We deduce that the map from (iii) is a bijection, and we denote by $u: M\rmap \bar{m}(\Delta)\subset \G$ its inverse. By construction, 
\[ u(s(\bar{m}(g, h))= \bar{m}(h, h)\]
for all $(g, h)\in \G\times_{s}\G$. The other structure maps are defined by
\begin{align*}
t(g) & = s(\bar{m}(g,g))\\
i(g) &= \bar{m}(u\circ s(g), g)\\
m(g,h)& = \bar{m}(g, i(h)).
\end{align*}
Next, we check the groupoid axioms.

First, note that $t$ is a surjective submersion since both $s$ and $\bar{m}$ are. Note also that, by definition, $s(u(s(g)))=s(g)$. On the other hand, $s(i(h))=s(\bar{m}(u(s(h)),h))$, so by the first axiom, $s(i(h))=s(\bar{m}(h,h))$ which is $t(h)$ by definition, thus implying that $m$ is well defined and $s\circ i =t$.

\begin{enumerate}
\item ($t\circ i (g) = s(g)$):

Using the second axiom we can compute
\begin{align*}
(t\circ i)(g) & = s(\bar{m}(i(g),i(g)))\\
              & = s(\bar{m}(\bar{m}(u(s(g)),g),\bar{m}(u(s(g)),g))\\
              & = s(\bar{m}(u(s(g)),u(s(g)))\\
              & = s(\bar{m}(\bar{m}(k,k),\bar{m}(k,k)))\\
              & = s(\bar{m}(k,k)),         
\end{align*}
where $u(s(g))=\bar{m}(k,k)$ and so $s(\bar{m}(k,k))=s(g)$ as desired. 

\item ($s(m(g,h)) = s(h)$, and $t(m(g,h)) = t(g)$):

For $(g,h)\in \G^{(2)}$, we have
\begin{align*}
s(m(g,h))& =  s(\bar{m}(g,i(h)))\\
         & =  s(\bar{m}(i(h),i(h))\\
         & =  t(i(h))\\
         & =  s(h).         
\end{align*} 
Here, the second equality is consequence of the first axiom, and the third is the definition of $t$.

Also, we have that
\begin{align*}
t(m(g,h))& =  t(\bar{m}(g,i(h)))\\
         & =  s(\bar{m}(\bar{m}(g,i(h)),\bar{m}(g,i(h)))\\
         & =  s(\bar{m}(g,g))\\
         & =  t(g).         
\end{align*} 
This time, the third equality is due to the second axiom.

\item ($s\circ u = id_M$, and $t\circ u = id_M$):

$s\circ u=id_{M}$ is just by definition. For the other, let $u(x)=\bar{m}(k,k)$. Then
\begin{align*}
(t\circ u)(x)& =  t(\bar{m}(k,k)))\\
             & =  s(\bar{m}(\bar{m}(k,k),\bar{m}(k,k)))\\
             & =  s(\bar{m}(k,k))\\
             & =  s(u(x)) \\
             & =id_{M}(x).         
\end{align*} 

\item ($m(i(g),g) = u(s(g)$,  $m(g,i(g)) = u(t(g)$, and $i^2 = id_{\G}$):

Again, let $u(s(g))=\bar{m}(k,k)$
\begin{align*}
m(i(g),g)& =  \bar{m}(i(g),i(g)))\\
         & =  \bar{m}(\bar{m}(u(s(g)),g),\bar{m}(u(s(g)),g))\\
         & =  \bar{m}(u(s(g)),u(s(g)))\\
         & =  \bar{m}(\bar{m}(k,k),\bar{m}(k,k)))\\
         & =  \bar{m}(k,k)=u(s(g)).         
\end{align*}

Next, suppose that $g=\bar{m}(h,k)$. Then $$i(g)=\bar{m}(u(s(\bar{m}(h,k))),\bar{m}(h,k)) = \bar{m}(\bar{m}(k,k))),\bar{m}(h,k))=\bar{m}(k,h).$$ From this, it follows that $i^2 = id_{\G}$, (and consequently that $i$ is a bijection).

Moreover,
\begin{align*}
m(g,i(g))& =  s(\bar{m}(i(i(g)),i(g)))\\
         & =  u(s(i(g))\\
         & =u(t(g)).       
\end{align*}

\item ($m(g,u(s(g)))=m(u(t(g)),g) = g$):

First we are to prove that $m(g,u(s(g)))=m(u(t(g)),g)$.
\begin{align*}
m(u(t(g)),g)& =  \bar{m}(u(t(g)),i(g)))\\
            & =  \bar{m}(u(s(\bar{m}(g,g))),\bar{m}(u(s(g)),g))\\
            & =  \bar{m}(\bar{m}(g,g),\bar{m}(u(s(g)),g))\\
            & =  \bar{m}(g,u(s(g))),         
\end{align*}
but $i(u(s(g)))=\bar{m}(u(s(g)),u(s(g)))$ and we saw the latter was equal to $u(s(g))$ thus allowing us to remove the bar in the last equation.

Finally, let $g=\bar{m}(h,k)$, where $$t(g)=s(\bar{m}(\bar{m}(h,k),\bar{m}(h,k))=s(\bar{m}(h,h)).$$ Then 
\begin{align*}
m(u(t(g)),g)& =  \bar{m}(u(t(g)),i(g)))\\
            & =  \bar{m}(u(s(\bar{m}(h,h))),\bar{m}(k,h))\\
            & =  \bar{m}(\bar{m}(h,h),\bar{m}(k,h))\\
            & =  \bar{m}(h,k) = g.         
\end{align*}

\item ($m$ is associative):

Let $(g,h),(h,k)\in\G^{(2)}$. We compute
\begin{align*}
m(g,m(h,k))& =  \bar{m}(g,i(\bar{m}(h,i(k))))\\
           & =  \bar{m}(g,\bar{m}(i(k),h))\\
           & =  \bar{m}(\bar{m}(g,i(h)),\bar{m}(\bar{m}(i(k),h)),i(h))\\
           & =  \bar{m}(m(g,h),\bar{m}(\bar{m}(i(k),h)),m(u(t(h),i(h)))\\
           & =  \bar{m}(m(g,h),\bar{m}(\bar{m}(i(k),h)),\bar{m}(u(t(h),h))\\     
           & =  \bar{m}(m(g,h),\bar{m}(i(k),u(t(h))\\  
           & =  \bar{m}(m(g,h),\bar{m}(i(k),i(u(t(h)))\\
           & =  \bar{m}(m(g,h),m(i(k),u(t(h)))\\  
           & =  \bar{m}(m(g,h),i(k)))\\
           & =  m(m(g,h),k)).
\end{align*}
\end{enumerate}

We turn now to the problem of whether the maps are smooth. We already pointed out that $t$ is a surjective submersion, in particular it is smooth. Notice that $u$ fits in the following diagram,
\begin{eqnarray*}
\xymatrix{
 \G \ar[r]^\Delta \ar[dr]_t & \G\times_s\G \ar[r]^{\bar{m}} & \G \\
 & M \ar[ur]_u & 
}
\end{eqnarray*}
and is therefore smooth. Also, since $s\circ u=id_M$, $d(s\circ u)_x=Id_x$. Let $X\in T_xM$, if $du_x(X)=0$,
\begin{eqnarray*}
ds_{u(x)}(du_x(X))& = & ds_{u(x)}(0)\\
d(s\circ u)_x (X) & = & 0\\
         Id_x (X) & = & 0\\
               X  & = & 0.         
\end{eqnarray*} 
Summing up, $u$ is an injective smooth immersion onto $u(M)=\bar{m}(\Delta)$, and again since $s$ is a left inverse and $s$ is continuous $u$ is a homeomorphism onto its image i.e. an embedding.\\
From the smoothness of $u$ and $\bar{m}$, the smoothness of $i$ and $m$ follow.\\
Finally, $i$ is its own inverse, and it is a diffeomorphism.
\end{proof}

\begin{crl}\label{crl: m bar} Given two Lie groupoids $\G$ over $M$ and $\H$ over $N$, and two smooth maps $F: \G\rmap \H$, $f: M\rmap N$, 
 then $(F, f)$ is a groupoid morphism if and only if 
\[ s(F(g))= f(s(g)), \ \bar{m}(F(g), F(h))= F(\bar{m}(g, h)\]
for all $g, h\in \G$ with $s(g)= s(h)$. 
\end{crl}

Continuing our previous motivating comments coming from deformations, here is one more explanatory remark. Assume that we have a deformation $s_{\epsilon}, t_{\epsilon}, m_{\epsilon}, u_{\epsilon}, i_{\epsilon}$ with $s= s_{\epsilon}$ not depending on $\epsilon$ ($s$-constant deformation). Proposition \ref{prop: m bar} reveals that, in order to study the variation associated to the deformation, it is enough to concentrate on the variation of $\bar{m}_{\epsilon}$ (since $s_{\epsilon}$ is constant in $\epsilon$). This will induce a cocycle $\xi_0\in C^{2}_{\textrm{def}}(\G)$ (discussed in detail in Section \ref{degree2}). The fact that $\xi_0$ is a cocycle incorporates, as we wanted, precisely the first order consequences of the axioms for $(s_{\epsilon}, \bar{m}_{\epsilon})$ mentioned in Proposition \ref{prop: m bar}:
\begin{enumerate}
\item The equation from (i) will imply that $ds(\xi_0(g, h))= ds(\xi(1, h))$, i.e., precisely the condition of $s$-projectability defining the complex $C^{*}_{\textrm{def}}(\G)$.
\item The associativity equation from (ii) will imply $\delta(\xi_0)= 0$.
\end{enumerate}

\begin{rmk}[(more natural variables)] 
The previous proposition and remark indicate that there is a more natural way of (re-)writing the deformation complex, so that it only makes use of $s$ and $\bar{m}$ in its definition; we will denote it by 
$\bar{C}^{*}_{\textrm{def}}(\G)$ (a more suggestive notation would be $C^{*}_{\textrm{def}}(\G, s, \bar{m})$). It arises by the standard change of variables 
\[ \G^{(k)} \stackrel{\sim}{\longleftrightarrow} \G^{[k]}, \ \ 
(g_1, \ldots , g_k) \longleftrightarrow (a_1, \ldots , a_k) \]
which relates strings of composable arrows to strings of arrows with the same source:
\begin{equation} 
\label{change-of-variables}
 a_{i}= g_{i}g_{i+1}\ldots g_{k}, \ \ g_{i}= \left\{
\begin{array}{ll}
 a_{i} a_{i+1}^{-1} & \ \textrm{for}\ i\leq k-1\\
 a_{k} & \ \textrm{for}\ i= k
\end{array}
\right.
\end{equation}

Explicitly, the $k$-cochains $u\in \bar{C}^{k}_{\textrm{def}}(\G)$ are the smooth maps 
\[ u:\G^{[k]} \longrightarrow T\G, (a_1,\ldots, a_k)\mapsto u(a_1,\ldots,a_k)\in T_{\bar{m}(a_1, a_2)}\G= T_{a_1a_{2}^{-1}}\G\]
defined on the strings of arrows with the same source, with the property that
\[ s(u(a_1, \ldots, a_k)) \in T_{s(\bar{m}(a_1, a_2))}\G\]
does not depend on $a_1$ (note: ``the first axiom'' for $s$ and $\bar{m}$, i.e. (i) of Proposition \ref{prop: m bar}, ensures that $s(\bar{m}(a_1, a_2))$ does not depend on $a_1$). 

The differential of $u\in \bar{C}^{k}_{\textrm{def}}(\G)$ is
\begin{align*}(\bar{\delta} u)(a_1,\ldots, a_{k+1}) & =- d\bar{m}(u(a_1, a_3, \ldots , a_{k+1}), u(a_2, a_3, \ldots , a_{k+1})  \\
&+\sum_{i=3}^{k+1} (-1)^{i+1}u(a_1, \ldots , \widehat{a_{i}}, \ldots , a_{k+1}) + (-1)^{k+1}u(a_1 a_{k+1}^{-1}, \ldots, a_ka_{k+1}^{-1}).
\end{align*}
(for $k= 0$, one keeps the same definition as for $C_{\text{def}}^{*}(\G)$). The second axiom for $s$ and $\bar{m}$, i.e. (ii) of Proposition \ref{prop: m bar}, ensures that $\delta$ is well-defined and squares to zero.
Of course, the change of variables (\ref{change-of-variables}) induces an isomorphism between $(\CG,\delta)$ and $(\bar{C}^{*}_{\textrm{def}}(\G), \bar{\delta})$.

Finally, let us mention that the reason we choose to use $C_{\text{def}}^{*}(\G)$ instead of $\bar{C}^{*}_{\textrm{def}}(\G)$  is that, searching in the existing literature involving cohomology of groupoids (and even of groups), we see that one always uses the $\G^{(k)}$'s (and the corresponding formulas) for the domains of the cochains. It is clear however that $\bar{C}^{*}_{\textrm{def}}(\G)$ and the entire view-point that arises from Proposition \ref{prop: m bar} is much more conceptual and we expect it to be useful in various related problems (e.g. for finding a non-linear analogue of the Gerstenhaber bracket that makes the deformation complex of a Lie algebroid into a DG Lie algebra structure and allows one to interpret Lie algebroid structures as Maurer-Cartan elements \cite{marius_ieke}; or in the study of higher groupoids, etc).
\end{rmk}

\bibliography{mybiblio}{}

\begin{thebibliography}{10}

\bibitem{camilo3}
C.~{Arias Abad} and M.~Crainic.
\newblock Representations up to homotopy of {L}ie algebroids.
\newblock {\em J. Reine Angew. Math.}, 663:91--126, 2012.

\bibitem{camilo}
C.~{Arias Abad} and M.~Crainic.
\newblock Representations up to homotopy and {B}ott's spectral sequence for
  {L}ie groupoids.
\newblock {\em Adv. Math.}, 248:416--452, 2013.

\bibitem{camilo2}
C.~{Arias Abad} and F.~Sch\"atz.
\newblock Deformations of {L}ie brackets and representations up to homotopy.
\newblock {\em Indag. Math. (N.S.)}, 22(1--2):27--54, 2011.

\bibitem{ale_thiago}
A.~Cabrera and T.~Drummond.
\newblock Van {E}st isomorphism for homogeneous cochains.
\newblock {\em Pacific J. Math.}, 287(2):297--336, 2017.

\bibitem{Copper}
D.~Coppersmith.
\newblock {\em Deformations of {L}ie groups and {L}ie Algebras}.
\newblock PhD thesis, Harvard {U}niversity, 1977.

\bibitem{GS}
A.~Coste, P.~Dazord, and A.~Weinstein.
\newblock Groupo\"\i des symplectiques.
\newblock In {\em Publications du {D}\'epartement de {M}ath\'ematiques.
  {N}ouvelle {S}\'erie. {A}, {V}ol.\ 2}, volume~87 of {\em Publ. D\'ep. Math.
  Nouvelle S\'er. A}, pages i--ii, 1--62. Univ. Claude-Bernard, Lyon, 1987.

\bibitem{VanEst}
M.~Crainic.
\newblock Differentiable and algebroid cohomology, van {E}st isomorphisms, and
  characteristic classes.
\newblock {\em Comment. Math. Helv.}, 78(4):681--721, 2013.

\bibitem{marius_ieke2}
M.~Crainic and I.~Moerdijk.
\newblock Foliation groupoids and their cyclic homology.
\newblock {\em Adv. Math.}, 157:177--197, 2001.

\bibitem{marius_ieke}
M.~Crainic and I.~Moerdijk.
\newblock Deformations of {L}ie brackets: cohomological aspects.
\newblock {\em J. Eur. Math. Soc. ({JEMS})}, 10(4):1037--1059, 2008.

\bibitem{ivan}
M.~Crainic and I.~Struchiner.
\newblock On the linearization theorem for proper {L}ie groupoids.
\newblock {\em Ann. Sci. \'Ec. Norm. Sup\'er. (4)}, 46(5):723--746, 2013.

\bibitem{matias_rui_metrics}
M.~del Hoyo and R.~L. Fernandes.
\newblock {Riemannian metrics on Lie groupoids.}
\newblock {\em {J. Reine Angew. Math.}}, 735:143--173, 2018.

\bibitem{matias_rui_fibrations}
M.~del Hoyo and R.~L. Fernandes.
\newblock Riemannian metrics on differentiable stacks.
\newblock {\em Math. Z.}, 292(1-2):103--132, 2019.

\bibitem{matias_cristian}
M.~del Hoyo and C.~Ortiz.
\newblock Morita equivalences of vector bundles.
\newblock {\em Int. Math. Res. Not. IMRN}, (14):4395--4432, 2020.

\bibitem{mehta}
A.~Gracia-Saz and R.~A. Mehta.
\newblock {$\mathcal{VB}$-groupoids and representation theory of Lie
  groupoids.}
\newblock {\em {J. Symplectic Geom.}}, 15(3):741--783, 2017.

\bibitem{Hei}
J.~L. Heitsch.
\newblock A cohomology for foliated manifolds.
\newblock {\em Comment. Math. Helv.}, 50:197--218, 1975.

\bibitem{Kodaira}
K.~Kodaira.
\newblock {\em Complex manifolds and deformation of complex structures}.
\newblock Classics in Mathematics. Springer-Verlag, Berlin, 2005.

\bibitem{lerman}
E.~Lerman.
\newblock Invariant vector fields and groupoids.
\newblock {\em Int. Math. Res. Not. IMRN}, (16):7394--7416, 2015.

\bibitem{macxu}
K.~C.~H. Mackenzie and P.~Xu.
\newblock Classical lifting processes and multiplicative vector fields.
\newblock {\em Quart. J. Math. Oxford Ser. (2)}, 49(193):59--85, 1998.

\bibitem{mcmillan}
D.~R. McMillan, Jr.
\newblock Some contractible open {$3$}-manifolds.
\newblock {\em Trans. Amer. Math. Soc.}, 102:373--382, 1962.

\bibitem{ieke_mrcun}
I.~Moerdijk and J.~Mr{\v{c}}un.
\newblock {\em Introduction to foliations and {L}ie groupoids}, volume~91 of
  {\em Cambridge Studies in Advanced Mathematics}.
\newblock Cambridge University Press, Cambridge, 2003.

\bibitem{mrcun}
J.~Mr{\v{c}}un.
\newblock {\em Stability and Invariants of {H}ilsum-{S}kandalis Maps}.
\newblock PhD thesis, {U}trecht {U}niversity, 1996.
\newblock ar{X}iv:math/0506484.

\bibitem{nr}
A.~Nijenhuis and R.~W. Richardson, Jr.
\newblock Deformations of homomorphisms of {L}ie groups and {L}ie algebras.
\newblock {\em Bull. Amer. Math. Soc.}, 73:175--179, 1967.

\bibitem{palais2}
R.~S. Palais and R.~W. Richardson, Jr.
\newblock Uncountably many inequivalent analytic actions of a compact group on
  {$R^{n}$}.
\newblock {\em Proc. Amer. Math. Soc.}, 14:374--377, 1963.

\bibitem{palais}
R.~S. Palais and T.~E. Stewart.
\newblock Deformations of compact differentiable transformation groups.
\newblock {\em Amer. J. Math.}, 82:935--937, 1960.

\bibitem{palais3}
R.~S. Palais and T.~E. Stewart.
\newblock The cohomology of differentiable transformation groups.
\newblock {\em Amer. J. Math.}, 83:623--644, 1961.

\bibitem{alan}
A.~Weinstein.
\newblock Linearization of regular proper groupoids.
\newblock {\em J. Inst. Math. Jussieu}, 1(3):493--511, 2002.

\end{thebibliography}
\bibliographystyle{plain}

\end{document}